\documentclass[final,3p]{elsarticle}

\usepackage{lineno,latexsym}
\PassOptionsToPackage{hyphens}{url}\usepackage{hyperref}
\modulolinenumbers[1]

\usepackage{fleqn}
\usepackage{amsmath}
\usepackage{amsthm}
\usepackage{amsfonts}
\usepackage{amssymb}
\usepackage{bm}
\usepackage{epsfig}
\usepackage{makeidx}
\usepackage{relsize}
\usepackage[export]{adjustbox}
\usepackage{graphics}
\usepackage{epstopdf}
\usepackage{subfigure}
\usepackage{graphicx}
\usepackage{xcolor}
\usepackage[normalem]{ulem}
\usepackage{algorithm}
\usepackage{algpseudocode}
\usepackage{enumerate}
\usepackage{dsfont}
\usepackage{multirow}
\usepackage[title]{appendix}

\usepackage{tikz}

\usepackage{tabularx}
\usepackage{ragged2e}
\newcolumntype{M}[1]{>{\centering\arraybackslash}m{#1}}
\usepackage{tikz}
\usetikzlibrary{plotmarks}
\usetikzlibrary{arrows}
\usetikzlibrary{shapes}
\usepackage{url}

\usepackage{mathtools}
\mathtoolsset{showonlyrefs}

\newcommand{\ie}{i.e.,\ }

\newcommand{\cB}{\mathcal{B}}
\newcommand{\cC}{\mathcal{C}}

\newcommand{\cO}{\mathcal{O}}

\newcommand{\bi}{\mathbf{i}}
\newcommand{\bj}{\mathbf{j}}
\newcommand{\bz}{\mathbf{z}}

\newcommand{\bv}{\mathbf{v}}

\newcommand{\ba}{\mathbf{a}}

\newcommand{\be}{\mathbf{e}}

\newcommand{\bff}{\mathbf{f}}

\newcommand{\beps}{{\boldsymbol{\epsilon}}}
\newcommand{\bl}{{\boldsymbol{\ell}}}

\newcommand{\R}{\mathds{R}}
\newcommand{\N}{\mathds{N}}
\newcommand{\Z}{\mathds{Z}}

\newcommand{\lra}{\longrightarrow}

\newcommand{\circled}[1]{\tikz[baseline=(char.base)]{
    \node[shape=circle,draw,inner sep=2pt] (char) {#1};}}

\newcommand{\mean}{\mathbb{E}}
\DeclareMathOperator*{\var}{\mathbb{V}}
\DeclareMathOperator*{\cov}{cov}

\theoremstyle{plain}
\newtheorem{thm}{Theorem}[section]
\newtheorem{prop}{Proposition}[section]
\newtheorem{lem}{Lemma}[section]
\newtheorem{cor}{Corollary}[section]

\theoremstyle{definition}
\newtheorem{defi}{Definition}[section]

\theoremstyle{remark}
\newtheorem{rmk}{Remark}[section]

\journal{arXiv}

\biboptions{authoryear}

\begin{document}

\begin{frontmatter}
	\title{Weighted least squares subdivision schemes \\ for noisy data on triangular meshes}
	
	\author[first]{Costanza Conti}
	\ead{costanza.conti@unifi,it}
	\author[last]{Sergio L\'opez-Ure\~na}
	\ead{sergio.lopez-urena@uv.es}
	\author[last]{Dionisio F. Y\'a\~nez\corref{cor}}
	\ead{Dionisio.Yanez@uv.es}
	\cortext[cor]{Corresponding author}
	\address[first]{Dipartimento di Ingegneria Industriale, Viale Morgagni 40/44, 50134, Universit\`a di Firenze, Italy}
	\address[last]{Departament de Matem\`atiques, Universitat de Val\`encia, Doctor Moliner Street 50, 46100 Burjassot, Valencia, Spain}
	
	\begin{abstract}
		This paper presents and analyses a new family of linear subdivision schemes to refine noisy data given on triangular meshes. The subdivision rules consist of locally fitting and evaluating a weighted least squares approximating first-degree polynomial. This type of rules, applicable to any type of triangular grid, including finite grids or grids containing extraordinary vertices, are geometry-dependent which may result in non-uniform schemes. For these new subdivision schemes, we are able to prove reproduction, approximation order, denoising capabilities and, for some special type of grids, convergence as well.
		Several numerical experiments demonstrate that their performance is similar to advanced local linear regression methods but their subdivision nature makes them suitable for use within a multiresolution context as well as to deal with noisy geometric data as shown with an example.
	\end{abstract}

	\begin{keyword}
		Bivariate subdivision scheme \sep weighted least squares regression \sep noisy data on triangulations
	\end{keyword}
	
\end{frontmatter}


\section{Introduction} \label{sec:intro}
In this paper, we design a novel family of subdivision schemes suitable to approximate bivariate scalar functions from noisy samples. As well known, a subdivision scheme iteratively generates values associated with denser and denser meshes by repeated application of local refinement rules. Whenever a subdivision scheme converges, for any set of initial data it generates a corresponding continuous function.

\smallskip In the univariate case, for data uniformly distributed on the real line, Dyn  et  al.  in \cite{DYNHORMANN} propose linear and symmetric refinement rules based on local polynomial approximation, where the polynomial is determined by a least squares fit to the data. They name the corresponding subdivision schemes \emph{least squares subdivision schemes} and show, by the help of numerical experiments, that they are particularly suited for noisy data and that the schemes' performance are comparable to advanced approximating methods, such as the local linear regression (LLR) method.
Still in the univariate setting, in \cite{MUSTAFA}  refinement rules based on $\ell_1$-optimization (rather than $\ell_2$-optimization) are considered. The resulting subdivision schemes not only mitigate the effect of noise but also the presence of \emph{outliers} on the limit function, without any prior information about the input data.
In \cite{LY24}, the authors generalize the work of \cite{DYNHORMANN} by applying \emph{weighted} least squares in the definition of the subdivision rules thus enhancing the subdivision schemes approximation capabilities while effectively managing noise. Naturally, tensor-product schemes of this univariate schemes can handle gridded multivariate data, which consist of contaminated samples of multivariate functions on tensor-product grids as done in \cite{DYNHORMANN}.

\smallskip In this paper, we extend to the bivariate setting the univariate approach introduced in \cite{LY24} but not in a tensor-product manner. Indeed, we propose a family of linear (non-uniform) subdivision schemes specifically designed for data structured on \emph{triangular meshes}. The refinement rules are grounded in local bivariate first-degree polynomial approximation through a weighted least squares fit applied to a local \emph{subset} of the data. Similar to the univariate case, the resulting least squares subdivision schemes are well-suited to fit noisy data. In addition, the proposed rules are designed to work seamlessly with any triangular grid configuration, whether finite or featuring extraordinary vertices, which is certainly a remarkable property.

Numerical experiments demonstrate that our method performs similarly to the bivariate Moving Least Squares (MLS) method, the Shepard method, the least squares Radial Basis Functions (RBF) method and the least squares Tensor Product B-Splines (TPBS) method, all designed for \emph{scattered data}. Notably, our approach stands out as one of the few subdivision techniques specifically tailored to refine noisy data on triangular meshes. In addition, its subdivision nature has potential to provide specific advantages when used for data approximation, including multiresolution representation, local support and efficiency, smoothness guarantees, natural handling of irregular topology, no parametrization requirement, and suitability for noisy geometric data. Indeed, some of these potential advantages are shown and investigated in this paper.

\smallskip The paper is organized as follows. In Section \ref{sec:preliminaries} basic concepts concerning bivariate subdivision schemes are recalled. The next Section \ref{sec:new_subdivision} describes the derivation of the refinement rules of the new subdivision schemes. General properties of the proposed  least squares subdivision schemes are investigated in Section \ref{sec:properties}. A convergence analysis based on the positivity of their coefficients is carried out in Section \ref{sec_convergence_uniform} for the special case of uniform triangulations. Section \ref{sub:regular} looks at the new least squares polynomial approximation subdivision scheme for equilateral and for triangular-rectangular grids providing examples of subdivision masks and pictures of basic limit functions. Section \ref{sec:num} compares a subdivision scheme of the new family with the MLS, the Shepard, the RBF and the TPBS methods, and discusses application to noisy geometric data.
The closing Section \ref{sec:Conclusion} is to summarize the paper's contribution and to sketch future work.

\section{Preliminaries on subdivision schemes for data on triangular meshes} \label{sec:preliminaries}

Subdivision schemes for data on triangular meshes are iterative methods to define a bivariate smooth function starting from a triangular mesh with attached values, the so-called \emph{initial data set}.

The idea behind a subdivision scheme is to recursively produce \emph{refinements} of the data set. To do so, existing faces of the triangulation are \emph{subdivided} by \emph{adding} vertices and attaching  real values to them, as shown in Figure \ref{fig:polyhedron}. A complete discussion concerning subdivision schemes for triangular meshes, out of the scope of this paper, can be found in the seminal C. Loop's Ph.D. thesis \cite{LOOP} as well as in \cite{C3, STAM1, GINKEL, STAM2}, just to mention a few references. Below, we only provide details of the specific type of subdivision schemes we propose, working on a triangulation $T = (V,E)$ defined by a set of vertices $V = \{\bv_1,\bv_2,\ldots,\bv_N\}\subset\R^2$ and a set of edges $E \subset \{1,\ldots,N\}^2$ connecting them, where $e=(i,j)\in E$ expresses that the vertices $\bv_i$ and $\bv_j$ are connected by an edge.

For readers less familiar with subdivision schemes, we recall that the valence of a vertex is the number of edges incident on it, and that in a triangular mesh, vertices with valence 6 are considered regular, while those with any other valence are called \emph{extraordinary}.

\begin{figure}[!h]
	\centering
	\resizebox{0.4\textwidth}{!}{
	\begin{tikzpicture}
\draw[line width=0.05mm,   black] (0.000000,0.476258) -- (0.501119,0.819991);
\draw[line width=0.05mm,   black] (0.501119,0.819991) -- (0.163998,1.000000);
\draw[line width=0.05mm,   black] (0.163998,1.000000) -- (0.000000,0.476258);
\draw[line width=0.05mm,   black] (0.501119,0.819991) -- (0.163998,1.000000);
\draw[line width=0.05mm,   black] (0.163998,1.000000) -- (0.794040,0.873767);
\draw[line width=0.05mm,   black] (0.794040,0.873767) -- (0.501119,0.819991);
\filldraw[black] (0.000000,0.476258) circle (0.1pt);
\filldraw[black] (0.501119,0.819991) circle (0.1pt);
\filldraw[black] (0.163998,1.000000) circle (0.1pt);
\filldraw[black] (0.794040,0.873767) circle (0.1pt);
\draw[line width=0.05mm,   black] (0.000000,0.066262) -- (0.501119,0.000000);
\draw[line width=0.05mm,   black] (0.501119,0.000000) -- (0.163998,0.385006);
\draw[line width=0.05mm,   black] (0.163998,0.385006) -- (0.000000,0.066262);
\draw[line width=0.05mm,   black] (0.501119,0.000000) -- (0.163998,0.385006);
\draw[line width=0.05mm,   black] (0.163998,0.385006) -- (0.794040,0.361272);
\draw[line width=0.05mm,   black] (0.794040,0.361272) -- (0.501119,0.000000);
\filldraw[black] (0.000000,0.066262) circle (0.1pt);
\filldraw[black] (0.501119,0.000000) circle (0.1pt);
\filldraw[black] (0.163998,0.385006) circle (0.1pt);
\filldraw[black] (0.794040,0.361272) circle (0.1pt);
\draw[densely dotted, line width=0.05mm,   black] (0.000000,0.476258) -- (0.000000,0.066262);
\draw[densely dotted, line width=0.05mm,   black] (0.501119,0.819991) -- (0.501119,0.000000);
\draw[densely dotted, line width=0.05mm,   black] (0.163998,1.000000) -- (0.163998,0.385006);
\draw[densely dotted, line width=0.05mm,   black] (0.794040,0.873767) -- (0.794040,0.361272);
\filldraw[black] (0.000000,0.476258) circle (0.1pt);
\filldraw[black] (0.501119,0.819991) circle (0.1pt);
\filldraw[black] (0.163998,1.000000) circle (0.1pt);
\filldraw[black] (0.794040,0.873767) circle (0.1pt);
\filldraw[black] (0.000000,0.066262) circle (0.1pt);
\filldraw[black] (0.501119,0.000000) circle (0.1pt);
\filldraw[black] (0.163998,0.385006) circle (0.1pt);
\filldraw[black] (0.794040,0.361272) circle (0.1pt);
\node[scale=0.2] at (0.601119,0.769991) {$(x_i,y_i,z_i)$};
\node[scale=0.2] at (0.601119,0.000000) {$(x_i,y_i)$};
\node[scale=0.2] at (0.551119,0.491995) {$z_i$};
\end{tikzpicture}
	}
	\resizebox{0.4\textwidth}{!}{
	\begin{tikzpicture}
\draw[line width=0.05mm,   black] (0.501119,0.819991) -- (0.250559,0.709624);
\draw[line width=0.05mm,   black] (0.250559,0.709624) -- (0.332559,0.868996);
\draw[line width=0.05mm,   black] (0.332559,0.868996) -- (0.501119,0.819991);
\draw[line width=0.05mm,   black] (0.163998,1.000000) -- (0.081999,0.861128);
\draw[line width=0.05mm,   black] (0.081999,0.861128) -- (0.332559,0.868996);
\draw[line width=0.05mm,   black] (0.332559,0.868996) -- (0.163998,1.000000);
\draw[line width=0.05mm,   black] (0.000000,0.476258) -- (0.250559,0.709624);
\draw[line width=0.05mm,   black] (0.250559,0.709624) -- (0.081999,0.861128);
\draw[line width=0.05mm,   black] (0.081999,0.861128) -- (0.000000,0.476258);
\draw[line width=0.05mm,   black] (0.250559,0.709624) -- (0.081999,0.861128);
\draw[line width=0.05mm,   black] (0.081999,0.861128) -- (0.332559,0.868996);
\draw[line width=0.05mm,   black] (0.332559,0.868996) -- (0.250559,0.709624);
\draw[line width=0.05mm,   black] (0.163998,1.000000) -- (0.332559,0.868996);
\draw[line width=0.05mm,   black] (0.332559,0.868996) -- (0.479019,0.977883);
\draw[line width=0.05mm,   black] (0.479019,0.977883) -- (0.163998,1.000000);
\draw[line width=0.05mm,   black] (0.794040,0.873767) -- (0.647580,0.826379);
\draw[line width=0.05mm,   black] (0.647580,0.826379) -- (0.479019,0.977883);
\draw[line width=0.05mm,   black] (0.479019,0.977883) -- (0.794040,0.873767);
\draw[line width=0.05mm,   black] (0.501119,0.819991) -- (0.332559,0.868996);
\draw[line width=0.05mm,   black] (0.332559,0.868996) -- (0.647580,0.826379);
\draw[line width=0.05mm,   black] (0.647580,0.826379) -- (0.501119,0.819991);
\draw[line width=0.05mm,   black] (0.332559,0.868996) -- (0.647580,0.826379);
\draw[line width=0.05mm,   black] (0.647580,0.826379) -- (0.479019,0.977883);
\draw[line width=0.05mm,   black] (0.479019,0.977883) -- (0.332559,0.868996);
\filldraw[black] (0.000000,0.476258) circle (0.1pt);
\filldraw[black] (0.501119,0.819991) circle (0.1pt);
\filldraw[black] (0.163998,1.000000) circle (0.1pt);
\filldraw[black] (0.794040,0.873767) circle (0.1pt);
\filldraw[red] (0.250559,0.709624) circle (0.1pt);
\filldraw[red] (0.081999,0.861128) circle (0.1pt);
\filldraw[red] (0.332559,0.868996) circle (0.1pt);
\filldraw[red] (0.647580,0.826379) circle (0.1pt);
\filldraw[red] (0.479019,0.977883) circle (0.1pt);
\draw[line width=0.05mm,   black] (0.501119,0.000000) -- (0.250559,0.033131);
\draw[line width=0.05mm,   black] (0.250559,0.033131) -- (0.332559,0.192503);
\draw[line width=0.05mm,   black] (0.332559,0.192503) -- (0.501119,0.000000);
\draw[line width=0.05mm,   black] (0.163998,0.385006) -- (0.081999,0.225634);
\draw[line width=0.05mm,   black] (0.081999,0.225634) -- (0.332559,0.192503);
\draw[line width=0.05mm,   black] (0.332559,0.192503) -- (0.163998,0.385006);
\draw[line width=0.05mm,   black] (0.000000,0.066262) -- (0.250559,0.033131);
\draw[line width=0.05mm,   black] (0.250559,0.033131) -- (0.081999,0.225634);
\draw[line width=0.05mm,   black] (0.081999,0.225634) -- (0.000000,0.066262);
\draw[line width=0.05mm,   black] (0.250559,0.033131) -- (0.081999,0.225634);
\draw[line width=0.05mm,   black] (0.081999,0.225634) -- (0.332559,0.192503);
\draw[line width=0.05mm,   black] (0.332559,0.192503) -- (0.250559,0.033131);
\draw[line width=0.05mm,   black] (0.163998,0.385006) -- (0.332559,0.192503);
\draw[line width=0.05mm,   black] (0.332559,0.192503) -- (0.479019,0.373139);
\draw[line width=0.05mm,   black] (0.479019,0.373139) -- (0.163998,0.385006);
\draw[line width=0.05mm,   black] (0.794040,0.361272) -- (0.647580,0.180636);
\draw[line width=0.05mm,   black] (0.647580,0.180636) -- (0.479019,0.373139);
\draw[line width=0.05mm,   black] (0.479019,0.373139) -- (0.794040,0.361272);
\draw[line width=0.05mm,   black] (0.501119,0.000000) -- (0.332559,0.192503);
\draw[line width=0.05mm,   black] (0.332559,0.192503) -- (0.647580,0.180636);
\draw[line width=0.05mm,   black] (0.647580,0.180636) -- (0.501119,0.000000);
\draw[line width=0.05mm,   black] (0.332559,0.192503) -- (0.647580,0.180636);
\draw[line width=0.05mm,   black] (0.647580,0.180636) -- (0.479019,0.373139);
\draw[line width=0.05mm,   black] (0.479019,0.373139) -- (0.332559,0.192503);
\filldraw[black] (0.000000,0.066262) circle (0.1pt);
\filldraw[black] (0.501119,0.000000) circle (0.1pt);
\filldraw[black] (0.163998,0.385006) circle (0.1pt);
\filldraw[black] (0.794040,0.361272) circle (0.1pt);
\filldraw[red] (0.250559,0.033131) circle (0.1pt);
\filldraw[red] (0.081999,0.225634) circle (0.1pt);
\filldraw[red] (0.332559,0.192503) circle (0.1pt);
\filldraw[red] (0.647580,0.180636) circle (0.1pt);
\filldraw[red] (0.479019,0.373139) circle (0.1pt);
\draw[densely dotted, line width=0.05mm,   black] (0.000000,0.476258) -- (0.000000,0.066262);
\draw[densely dotted, line width=0.05mm,   black] (0.501119,0.819991) -- (0.501119,0.000000);
\draw[densely dotted, line width=0.05mm,   black] (0.163998,1.000000) -- (0.163998,0.385006);
\draw[densely dotted, line width=0.05mm,   black] (0.794040,0.873767) -- (0.794040,0.361272);
\filldraw[black] (0.000000,0.476258) circle (0.1pt);
\filldraw[black] (0.501119,0.819991) circle (0.1pt);
\filldraw[black] (0.163998,1.000000) circle (0.1pt);
\filldraw[black] (0.794040,0.873767) circle (0.1pt);
\filldraw[red] (0.250559,0.709624) circle (0.1pt);
\filldraw[red] (0.081999,0.861128) circle (0.1pt);
\filldraw[red] (0.332559,0.868996) circle (0.1pt);
\filldraw[red] (0.647580,0.826379) circle (0.1pt);
\filldraw[red] (0.479019,0.977883) circle (0.1pt);
\filldraw[black] (0.000000,0.066262) circle (0.1pt);
\filldraw[black] (0.501119,0.000000) circle (0.1pt);
\filldraw[black] (0.163998,0.385006) circle (0.1pt);
\filldraw[black] (0.794040,0.361272) circle (0.1pt);
\filldraw[red] (0.250559,0.033131) circle (0.1pt);
\filldraw[red] (0.081999,0.225634) circle (0.1pt);
\filldraw[red] (0.332559,0.192503) circle (0.1pt);
\filldraw[red] (0.647580,0.180636) circle (0.1pt);
\filldraw[red] (0.479019,0.373139) circle (0.1pt);
\end{tikzpicture}
	}
	\caption[Subdivision schemes applied on data with a discontinuity]
	{An example of triangular mesh refinement.
	Left, the initial data. Right, the result of a single refinement step.
	Projected on the bottom, the triangulations to which the data are attached. It can be seen that each face is divided into four smaller faces, requiring the insertion of new vertices.}
	\label{fig:polyhedron}
\end{figure}

\smallskip
As already mentioned, we assume that some real values, $\bz^0=\{z^0_1, z^0_2,\cdots ,z^0_{N^0}\}$, are attached to an initial triangulation $T^{0} = (V^0,E^0)$, with vertices  $\bv^0_j,\ j=1,\ldots, N^0$, defining the data set in $\R^3$:
\begin{equation}\label{dataset}
	\{(\bv^0_1,z^0_1),(\bv^0_2,z^0_2),\cdots  (\bv^0_{N^0},z^0_{N^0})\},\quad (\bv^0_i,z^0_i) \in V^0 \times \R \subset \R^3.
\end{equation}
The first subdivision step consists in refining both the triangulation $T^0$ and the corresponding real values $\bz^0$ to obtain the refined triangulation $T^1$ and the corresponding refined real values $\bz^1$, and so forth for the next subdivision steps. Here, we consider the so-called \emph{semiregular case} where the successive refinements of the triangulation, say $T^{1},T^{2},T^3\ldots$, are simply obtained by adding the mid-point of each grid edge (see Figure \ref{fig:triangulation_after_2_iterations} and the reference \cite{DAUBECHIES}). On the contrary, the refinement of the corresponding data is obtained in a more complicated way: At the $k$-th iteration, given the data $\bz^k=\{z^k_\ell,\ \ell=1,\cdots,N^k\}$, the subdivision step consists in computing each new data $z^{k+1}_\ell$, attached to the vertex $\bv^{k+1}_\ell\in V^{k+1}$ of the triangulation $T^{k+1}$, as a linear combination of some of the previous data with a formula of type
\begin{equation} \label{eq:rules}
	z^{k+1}_\ell=\sum_{j\in \cB^{k+1,\ell}} \alpha^{k+1,\ell}_j z^{k}_j,\quad \ell = 1,\ldots,N^{k+1},
\end{equation}
where the coefficients $\alpha^{k+1,\ell}_j$ are real numbers and $\cB^{k+1,\ell}$ is a set of indices which is chosen so that the vertices $\{\bv^k_j\}_{j\in \cB^{k+1,\ell}}\subset T^k$ are near the new vertex $\bv^{k+1}_\ell$.

The linear operator $\{z^k_j\}_{j\in \cB^{k+1,\ell}} \mapsto z^{k+1}_\ell$ defined in \eqref{eq:rules} is called \emph{refinement rule}. When the new data is attached to a vertex of the previous iteration, i.e. $v^{k+1}_\ell\in V^k$, it is specifically called \emph{replacement} rule, whereas when it is attached to a vertex $v^{k+1}_\ell\in V^{k+1}\setminus V^k$, it is called \emph{insertion} rule.
The set of indices $\cB^{k+1,\ell}$ is the so-called \emph{stencil} of the refinement rule, which indicates the neighbouring values to be involved in the computation.
Observe in \eqref{eq:rules} that we explicitly write the dependence of its coefficients and its stencil on $k,\ell$, meaning that, in general, the refinement rule depends on the iteration and on the location.

\begin{figure}[!h]
	\centering
	\includegraphics[width=0.3\linewidth,clip,trim={170 50 170 50}]{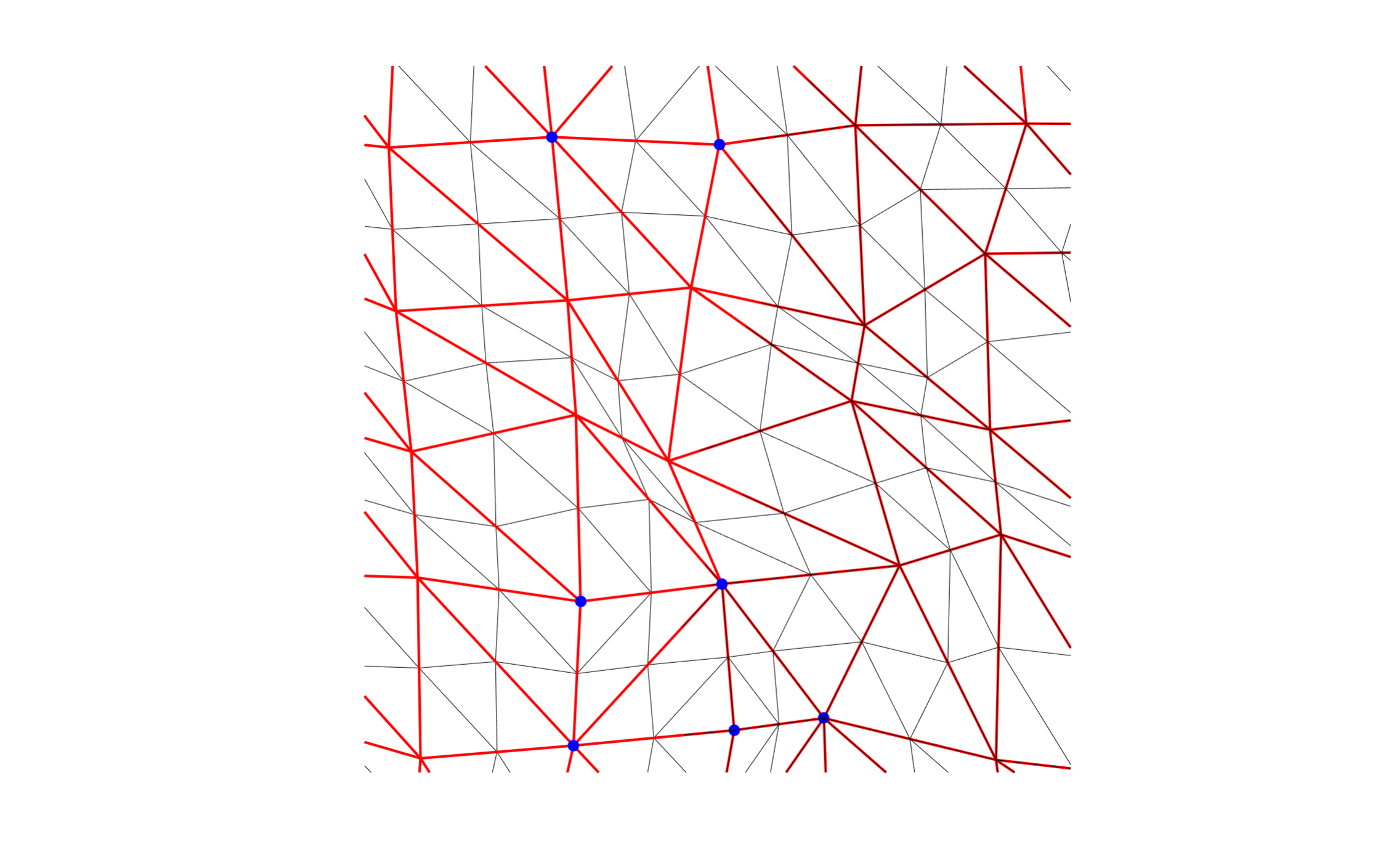}
	\hspace*{20pt}
	\includegraphics[width=0.3\linewidth,clip,trim={170 50 170 50}]{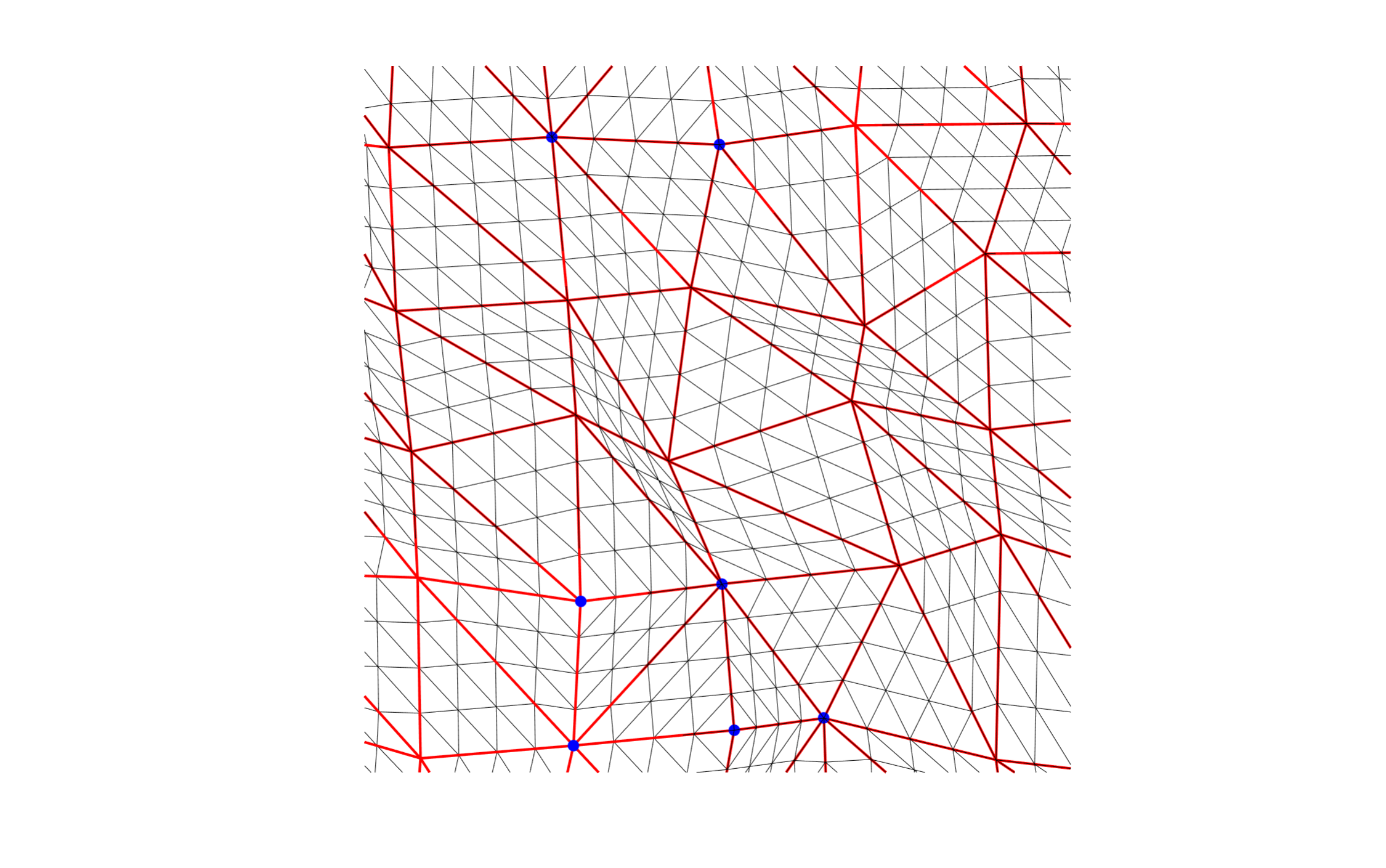}
	\caption{An example of triangulation refinement by mid-point insertion. In red, the original triangulation. In black, the result of a single refinement step (left) and two refinement steps (right). The initial triangulation is irregular (see the extraordinary vertices in blue). No other extraordinary vertices are added by the refinement and each patch defined by the initial faces constitutes a uniform triangulation (see the discussion of Section \ref{sec:properties}).
	\label{fig:triangulation_after_2_iterations}}
\end{figure}

For later use, we conclude the section by mentioning that the refinement rules just described define the $k$-level linear \emph{subdivision operator} $S^{k} : \R^{N^k} \lra \R^{N^{k+1}}$, which given $\bz^k$ returns $\bz^{k+1}$, \ie
\begin{equation}\label{linear}\bz^{k+1}=S^{k} \bz^k,\quad k\ge 0.\end{equation}
In general, the \emph{subdivision operator} $S^k$ depends on vertices location of the triangulation $T^k$, but, in case $T^k$ is \emph{uniform} (see Section \ref{sec_convergence_uniform}), the dependence vanishes. Whenever required, we highlight the dependence by writing $S^k_{T^k}$.

\section{The new weighted least squares subdivision scheme for functional data}\label{sec:new_subdivision}

This section is to describe the new least squares subdivision schemes we propose. We already mentioned that the denser and denser triangulations $T^{1},T^{2},T^{3}\ldots$ are simply obtained by subdividing each face into four similar faces by adding the edges mid-points as Figure \ref{fig:triangulation_after_2_iterations} displays. We define the refined data
$z^{k+1}_\ell,\ \ell=1,\ldots,N^{k+1}$, by computing and evaluating a weighted least squares degree-$1$ polynomial even though higher degree polynomials can be potentially used. The data involved in the least squares problem belong to what we call a `ball' centred at the inserted/replaced vertex $\bv^{k+1}_\ell$. We will name the corresponding sets of indices by $\cB^{k+1,\ell}$. The vertices locations and weights used in the least squares determine the coefficients $\alpha^{k+1,\ell}_j,\ \ell=1,\ldots,N^{k+1}$, used  in the refinement rules \eqref{eq:rules}.

\smallskip
We continue by detailing the weighted least squares degree-$1$ polynomial approximation, in Subsection \ref{sub:31}. Then, in Subsection \ref{sec:subdivision_rules}, we derive the explicit rules of the new subdivision scheme. Finally, in Subsection \ref{subsec:weights} we discuss the weights and stencils selection.

\subsection{The refinement rules based on the weighted least squares degree-$1$ polynomial approximation}\label{sub:31}

\smallskip In this section we explain how the weighted least squares degree-$1$ polynomial approximation is used to define the new subdivision schemes.

Starting from $\bz^k=\{z_1^k,z_2^k,\cdots, z^k_{N_k}\}$, values attached to the triangulation $T^k$, the data
$$\bz^{k+1}=\{z_1^{k+1},z_2^{k+1},\cdots, z^{k+1}_{N_k}\}$$ associated with the refined triangulation $T^{k+1}$ is obtained as follows. Let $\hat \bv = \bv^{k+1}_\ell$ be either a vertex or a mid-point of $T^{k+1}$, associated to $\hat z = z^{k+1}_\ell$. Given a subset of vertices of $T^{k}$ around $\hat \bv$, say $\{\bv_i=(x_i,y_i)\}_{i=1}^n$ which includes at least the three vertices of a face, let $\{z_i\}_{i=1}^n$ the associated data and $\{w_i\}_{i=1}^n$,  associated positive weights. The weighted least squares degree-$1$ polynomial approximation at $\hat \bv$ is obtained
%
%
as the solution of the following minimization problem
\begin{equation}\label{WLS}
	\min_{a_0, a_1, a_2\in \R }\sum_{i=1}^n w_i \left( a_0+a_1x_i+a_2y_i-z_i\right)^2. 
\end{equation}
Solving the previous minimization problem is equivalent to solving the linear system
\[
\begin{pmatrix}
	\displaystyle{\sum_{i=1}^n}w_i		& \displaystyle{\sum_{i=1}^n}w_i x_i 		& \displaystyle{\sum_{i=1}^n}w_i y_i \\
	\displaystyle{\sum_{i=1}^n}w_i x_i 	& \displaystyle{\sum_{i=1}^n}w_i x_i^2 	& \displaystyle{\sum_{i=1}^n}w_i x_i y_i \\
	\displaystyle{\sum_{i=1}^n}w_i y_i 	&\displaystyle{\sum_{i=1}^n}w_i x_i y_i 	& \displaystyle{\sum_{i=1}^n}w_i y_i^2
\end{pmatrix}
\begin{pmatrix} a_0\\ a_1\\ a_2
\end{pmatrix}=
\begin{pmatrix}
	\displaystyle{\sum_{i=1}^n}w_i z_i 		\\
	\displaystyle{\sum_{i=1}^n}w_i z_i x_i  	\\
	\displaystyle{\sum_{i=1}^n}w_i z_i y_i 	\\
\end{pmatrix},
\]
and evaluating the linear polynomial $\hat p_1(x,y)=a_0+a_1\, x+ a_2\, y$ at $\hat \bv$, that is computing $\hat z := p_1(\hat \bv)$.
Since $\Pi_1$ (the space of linear polynomials) is shift invariant,
this is equivalent to consider the points $\{\bv_i - \hat \bv\}_{i=1}^n$ and evaluate the polynomial at $(0,0)$, instead.
Since $\hat z = \hat p_1(0,0) = a_0$, the Cramer's formula reveals that
\begin{align*}
	\hat z = a_0=
	\frac{
	\left|
	\begin{array}{ccc}
		\displaystyle{\sum_{i=1}^n}w_i z_i 		& \displaystyle{\sum_{i=1}^n}w_i x_i 		& \displaystyle{\sum_{i=1}^n}w_i y_i \\
		\displaystyle{\sum_{i=1}^n}w_i z_i x_i 	& \displaystyle{\sum_{i=1}^n}w_i x_i^2 	& \displaystyle{\sum_{i=1}^n}w_i x_i y_i \\
		\displaystyle{\sum_{i=1}^n}w_i z_i y_i 	& \displaystyle{\sum_{i=1}^n}w_i x_i y_i 	& \displaystyle{\sum_{i=1}^n}w_i y_i^2
	\end{array}
	\right|
	}{
	\left|
	\begin{array}{ccc}
		\displaystyle{\sum_{i=1}^n}w_i		& \displaystyle{\sum_{i=1}^n}w_i x_i 		& \displaystyle{\sum_{i=1}^n}w_i y_i \\
		\displaystyle{\sum_{i=1}^n}w_i x_i 	& \displaystyle{\sum_{i=1}^n}w_i x_i^2 	& \displaystyle{\sum_{i=1}^n}w_i x_i y_i \\
		\displaystyle{\sum_{i=1}^n}w_i y_i 	& \displaystyle{\sum_{i=1}^n}w_i x_i y_i 	& \displaystyle{\sum_{i=1}^n}w_i y_i^2
	\end{array}
	\right|
	}.
\end{align*}
Using basic properties of the determinants (specifically its multilinearity), we can write
\begin{align*}
	\hat z=
	\frac{
	\displaystyle{\sum_{j=1}^n}
	\left|
	\begin{array}{ccc}
		w_j z_j 		& \displaystyle{\sum_{i=1}^n}w_i x_i 		& \displaystyle{\sum_{i=1}^n}w_i y_i \\
		w_j z_j x_j 	& \displaystyle{\sum_{i=1}^n}w_i x_i^2 	& \displaystyle{\sum_{i=1}^n}w_i x_i y_i \\
		w_j z_j y_j 	& \displaystyle{\sum_{i=1}^n}w_i x_i y_i 	& \displaystyle{\sum_{i=1}^n}w_i y_i^2
	\end{array}
	\right|
	}{
	\displaystyle{\sum_{j=1}^n}
	\left|
	\begin{array}{ccc}
		w_j		& \displaystyle{\sum_{i=1}^n}w_i x_i 		& \displaystyle{\sum_{i=1}^n}w_i y_i \\
		w_j x_j 	& \displaystyle{\sum_{i=1}^n}w_i x_i^2 	& \displaystyle{\sum_{i=1}^n}w_i x_i y_i \\
		w_j y_j 	& \displaystyle{\sum_{i=1}^n}w_i x_i y_i 	& \displaystyle{\sum_{i=1}^n}w_i y_i^2
	\end{array}
	\right|
	}
	=
	\frac{
	\displaystyle{\sum_{j=1}^n}w_j z_j
	\left|
	\begin{array}{ccc}
		1 		& \displaystyle{\sum_{i=1}^n}w_i x_i 		& \displaystyle{\sum_{i=1}^n}w_i y_i \\
		x_j 	& \displaystyle{\sum_{i=1}^n}w_i x_i^2 	& \displaystyle{\sum_{i=1}^n}w_i x_i y_i \\
		y_j 	& \displaystyle{\sum_{i=1}^n}w_i x_i y_i 	& \displaystyle{\sum_{i=1}^n}w_i y_i^2
	\end{array}
	\right|
	}{
	\displaystyle{\sum_{j=1}^n}w_j
	\left|
	\begin{array}{ccc}
		1		& \displaystyle{\sum_{i=1}^n}w_i x_i 		& \displaystyle{\sum_{i=1}^n}w_i y_i \\
		x_j 	& \displaystyle{\sum_{i=1}^n}w_i x_i^2 	& \displaystyle{\sum_{i=1}^n}w_i x_i y_i \\
		y_j 	& \displaystyle{\sum_{i=1}^n}w_i x_i y_i 	& \displaystyle{\sum_{i=1}^n}w_i y_i^2
	\end{array}
	\right|
	}.
\end{align*}
Therefore, defining the determinants
\begin{equation} \label{eq:mu}
	\mu_j := \left|
	\begin{array}{ccc}
		1		& \displaystyle{\sum_{i=1}^n}w_i x_i 		& \displaystyle{\sum_{i=1}^n}w_i y_i \\
		x_j 	& \displaystyle{\sum_{i=1}^n}w_i x_i^2 	& \displaystyle{\sum_{i=1}^n}w_i x_i y_i \\
		y_j 	& \displaystyle{\sum_{i=1}^n}w_i x_i y_i 	& \displaystyle{\sum_{i=1}^n}w_i y_i^2
	\end{array}
	\right|, \quad j=1,\ldots, n,
\end{equation}
we arrive at the formula we will use in Section \ref{sec:subdivision_rules} to define the refinement subdivision rules:
\begin{equation} \label{eq:pre_rule}
	\hat z=\sum_{i=1}^n \alpha_i z_i, \qquad \alpha_i := \frac{w_i \mu_i}{\sum_{j=1}^n w_j \mu_j}.
\end{equation}
Observe that it is well-defined as long as the denominator is not zero.

Next, we explore the geometric meaning of the value $\mu_j$. Indeed, applying again the multilinearity of the determinant we have
\begin{align*}
	\mu_j &= \left|
	\begin{array}{ccc}
		1		& \displaystyle{\sum_{i=1}^n}w_i x_i 		& \displaystyle{\sum_{i=1}^n}w_i y_i \\
		x_j 	& \displaystyle{\sum_{i=1}^n}w_i x_i^2 	& \displaystyle{\sum_{i=1}^n}w_i x_i y_i \\
		y_j 	& \displaystyle{\sum_{i=1}^n}w_i x_i y_i 	& \displaystyle{\sum_{i=1}^n}w_i y_i^2
	\end{array}
	\right|
	=
	\sum_{i=1}^n
	\left|
	\begin{array}{ccc}
		1		& w_i x_i 		& \displaystyle{\sum_{i=1}^n}w_i y_i \\
		x_j 	& w_i x_i^2 	& \displaystyle{\sum_{i=1}^n}w_i x_i y_i \\
		y_j 	& w_i x_i y_i 	& \displaystyle{\sum_{i=1}^n}w_i y_i^2
	\end{array}
	\right|
	=
	\sum_{i=1}^n w_i x_i
	\left|
	\begin{array}{ccc}
		1		& 1 		& \displaystyle{\sum_{i=1}^n}w_i y_i \\
		x_j 	& x_i 	& \displaystyle{\sum_{i=1}^n}w_i x_i y_i \\
		y_j 	& y_i 	& \displaystyle{\sum_{i=1}^n}w_i y_i^2
	\end{array}
	\right|,
\end{align*}
and using the same arguments with the last column, we arrive at
\begin{align*}
	\mu_j &=
	\sum_{i,\ell=1}^n w_i x_i w_\ell y_\ell\left|
	\begin{array}{ccc}
		1		& 1 		& 1 \\
		x_j 	& x_i 	& x_\ell \\
		y_j 	& y_i 	& y_\ell
	\end{array}
	\right|.
\end{align*}
For $i=\ell$, the determinant is zero. For the other indices, we exploit the symmetry of the cases $i<\ell$ and $\ell<i$ and get,
\begin{align*}
	\mu_j &=
	\sum_{1\leq i < \ell \leq n}^n w_i x_i w_\ell y_\ell\left|
	\begin{array}{ccc}
		1		& 1 		& 1 \\
		x_j 	& x_i 	& x_\ell \\
		y_j 	& y_i 	& y_\ell
	\end{array}
	\right|
	+
	\sum_{1\leq \ell< i \leq n} w_i x_i w_\ell y_\ell\left|
	\begin{array}{ccc}
		1		& 1 		& 1 \\
		x_j 	& x_i 	& x_\ell \\
		y_j 	& y_i 	& y_\ell
	\end{array}
	\right|,\\
	&=
	\sum_{1\leq i < \ell \leq n}^n w_i x_i w_\ell y_\ell\left|
	\begin{array}{ccc}
		1		& 1 		& 1 \\
		x_j 	& x_i 	& x_\ell \\
		y_j 	& y_i 	& y_\ell
	\end{array}
	\right|
	-
	\sum_{1\leq \ell < i \leq n} w_i x_i w_\ell y_\ell\left|
	\begin{array}{ccc}
		1		& 1 		& 1 \\
		x_j 	& x_\ell 	& x_i \\
		y_j 	& y_\ell 	& y_i
	\end{array}
	\right|.
\end{align*}
Swapping the indices ($i\leftrightarrow\ell$) in the second summation, we arrive at
\begin{align}
	\mu_j
	&=
	\sum_{1\leq i < \ell \leq n}^n w_i x_i w_\ell y_\ell\left|
	\begin{array}{ccc}
		1		& 1 		& 1 \\
		x_j 	& x_i 	& x_\ell \\
		y_j 	& y_i 	& y_\ell
	\end{array}
	\right|
	-
	\sum_{1\leq i < \ell \leq n} w_\ell x_\ell w_i y_i\left|
	\begin{array}{ccc}
		1		& 1 		& 1 \\
		x_j 	& x_i 	& x_\ell \\
		y_j 	& y_i 	& y_\ell
	\end{array}
	\right|
	\\
	&=
	\sum_{1\leq i < \ell \leq n}^n (w_i x_i w_\ell y_\ell-w_\ell x_\ell w_i y_i)\left|
	\begin{array}{ccc}
		1		& 1 		& 1 \\
		x_j 	& x_i 	& x_\ell \\
		y_j 	& y_i 	& y_\ell
	\end{array}
	\right|
	=\sum_{1\leq i < \ell\leq n} w_i w_\ell\left|
	\begin{array}{cc}
		x_i 	& x_\ell \\
		y_i 	& y_\ell
	\end{array}
	\right|\left|
	\begin{array}{ccc}
		1		& 1 		& 1 \\
		x_j 	& x_i 	& x_\ell \\
		y_j 	& y_i 	& y_\ell
	\end{array}
	\right|
	\\
	&=\sum_{1\leq i < \ell\leq n} w_i w_\ell\left|
	\begin{array}{cc}
		\bv_i 	& \bv_\ell
	\end{array}
	\right|\left|
	\begin{array}{ccc}
		1		& 1 		& 1 \\
		\bv_j 	& \bv_i 	& \bv_\ell
	\end{array}
	\right|. \label{eq:mu_final}
\end{align}
From above we see that the $2\times2$ determinant is zero whenever the evaluation point $(0,0)$ and $\bv_i,\bv_\ell$ are collinear, while the $3\times3$ determinant is zero when $\bv_j,\bv_i,\bv_\ell$ are collinear. In other words, a summand is zero only if $\bv_i,\bv_\ell$ are aligned with either the insertion point or $\bv_j$.
\begin{rmk} \label{rmk_well_defined}
	Some comments are worth making
	\begin{enumerate}[(i)]
		
		\item \label{rmk_well_defined_collinear} Since we are performing a linear polynomial fitting, the new data $\hat z$ is well-defined provided there exist three non-collinear vertices $\bv_j,\bv_i,\bv_\ell$ such that $\bv_i,\bv_\ell,\hat \bv$ are also non-collinear. This means that, $\hat z$ is well-defined whenever the stencil that the algorithm uses contains a face of the triangulation $T^k$.
		
		\item \label{rmk_well_defined_2} In the special situations when $\displaystyle{\sum_{i=1}^n}w_i x_i = \displaystyle{\sum_{i=1}^n}w_i y_i = 0$ ---and this is the case of uniform triangulations discussed in Section \ref{sec_convergence_uniform}--- the $\mu_j$ determinants do not really depend on $j$, being
\[
		\mu_j = \left|
		\begin{array}{cc}
			\displaystyle{\sum_{i=1}^n}w_i x_i^2 	& \displaystyle{\sum_{i=1}^n}w_i x_i y_i \\
			\displaystyle{\sum_{i=1}^n}w_i x_i y_i 	& \displaystyle{\sum_{i=1}^n}w_i y_i^2
		\end{array}
		\right|, \quad \forall j=1,\ldots,n.\]
		Thus, the refinement rule in \eqref{eq:pre_rule} simply becomes
		\begin{equation}\label{rulespecific}
			\hat z=\sum_{i=1}^n \alpha_i z_i, \qquad
			\alpha_i = \frac{w_i}{\sum_{j=1}^n w_j},\quad i=1,\ldots,n.
		\end{equation}
	\end{enumerate}
\end{rmk}

\subsection{The L-ball for weighted least squares subdivision schemes}\label{sec:subdivision_rules}

Based on the previous section, we here provide all the details needed to define our new family of subdivision schemes for noisy data. We name the schemes as the {\bf$L$-ball weighted least squares subdivision schemes} for data on triangulations, where \emph{$L$ is a fixed positive number} a priori selected. Different selections of $L$ generate different schemes.

On a given subdivision level $k+1$ and on a given vertex index $\ell$, we start by associating with the vertex $\bv^{k+1}_\ell$ ---which can be either a vertex or a mid-point of $T^{k}$--- the $L$-ball $\{\bv^{k}_j\}_{j\in\cB^{k+1,\ell}}$, based on the set of indices
\begin{equation} \label{eq_stencil_definition}
	\cB^{k+1,\ell} := \{ j \in \{1,2,\ldots,N^k\} \ : \ \ \|\bv^{k}_j-\bv^{k+1}_\ell\| < 2^{-k}L\},
\end{equation}
where $\|\cdot\|$ is the Euclidean norm.
Then, replacing $\bv_i$ by $\bv^k_{i} - \bv_\ell^{k+1}$ in \eqref{eq:pre_rule} and \eqref{eq:mu_final} and using $\cB^{k+1,\ell}$,
we arrive at the refinement rules
\begin{equation}\label{newrules}
	z_\ell^{k+1}=\sum_{i\in\cB^{k+1,\ell}} \alpha^{k+1,\ell}_i z^k_{i}, \quad \alpha^{k+1,\ell}_i := \frac{w^{k+1,\ell}_{i} \mu^{k+1,\ell}_i}{\sum_{j\in\cB^{k+1,\ell}} w^{k+1,\ell}_{j} \mu^{k+1,\ell}},
\end{equation}
where $w^{k+1,\ell}_{i}$ are selected positive weights and
\begin{align*}
	\mu^{k+1,\ell}_i :=\sum \Big\{w^k_{j_1} w^k_{j_2}\left|
	\begin{array}{cc} \bv^k_{j_1} - \bv_\ell^{k+1} 	& \bv^k_{j_2} - \bv_\ell^{k+1}
	\end{array}
	\right| \left|
	\begin{array}{ccc}
		1		& 1 		& 1 \\
		\bv^k_{i} - \bv_\ell^{k+1} 	& \bv^k_{j_1} - \bv_\ell^{k+1} 	& \bv^k_{j_2} - \bv_\ell^{k+1}
	\end{array}
	\right|
	\\ : \ j_1,j_2\in \cB^{k+1,\ell}, \ j_1 < j_2
	\Big\},\quad i\in\cB^{k+1,\ell}.
\end{align*}
To simplify the last definition, we use the determinant invariant under the addition of a multiple of one row to another:
\begin{align*}
	&\left|
	\begin{array}{cc} \bv^k_{j_1} - \bv_\ell^{k+1} 	& \bv^k_{j_2} - \bv_\ell^{k+1}
	\end{array}
	\right|
	=
	\left|
	\begin{array}{ccc}
		1 & 1 & 1 \\
		\bv_\ell^{k+1} & \bv^k_{j_1} & \bv^k_{j_2}
	\end{array}
	\right|
	\\
	&\left|
	\begin{array}{ccc}
		1		& 1 		& 1 \\
		\bv^k_{i} - \bv_\ell^{k+1} 	& \bv^k_{j_1} - \bv_\ell^{k+1} 	& \bv^k_{j_2} - \bv_\ell^{k+1}
	\end{array}
	\right|
	=
	\left|
	\begin{array}{ccc}
		1		& 1 		& 1 \\
		\bv^k_{i} 	& \bv^k_{j_1} 	& \bv^k_{j_2}
	\end{array}
	\right|.
\end{align*}
Then, a simplified version of the previous formula for $\mu^{k+1,\ell}_i$ is
\begin{align*}
	\mu^{k+1,\ell}_i =\sum \left\{w^k_{j_1} w^k_{j_2}
	\left|
	\begin{array}{ccc}
		1 & 1 & 1 \\
		\bv_\ell^{k+1} & \bv^k_{j_1} & \bv^k_{j_2}
	\end{array}
	\right|
	\left|
	\begin{array}{ccc}
		1		& 1 		& 1 \\
		\bv^k_{i} 	& \bv^k_{j_1} 	& \bv^k_{j_2}
	\end{array}
	\right| \ : \ j_1,j_2\in \cB^{k+1,\ell}, \ j_1 < j_2
	\right\},\quad i\in\cB^{k+1,\ell}.
\end{align*}

 { Figure \ref{fig:rings} shows an example of $L$-balls used for the refinement rules where the center of the ball is marked with a cross, while red dots show the vertices inside the balls.}

\begin{figure}[!h]
	\centering
	\begin{tabular}{cc}
		\resizebox{0.4\textwidth}{!}{
		\input{images/ring_replacement.txt}
		}
		&
		\resizebox{0.4\textwidth}{!}{
		\input{images/ring_insertion.txt}
		}
	\end{tabular}
	\caption{Example of balls used for the refinement rules. The center of the ball is marked with a cross (a vertex, in the left figure, and a mid-point, in the right figure). Red dots show the vertices inside the balls.}
	\label{fig:rings}
\end{figure}

\subsection{Weights selection}\label{subsec:weights}
We continue by discussing the selection of the weights used in the weighted least squares approximation and consequently in the subdivision schemes.  Since in our approach the polynomial $\hat p_1$ is evaluated at a {given} fixed vertex $\bv^{k+1}_{\ell}$ (either a vertex or a mid-point of $T^k$) the weights $w^{k+1,\ell}_{i}$ are set by the help of a weight function, $W : [0, 1)\rightarrow (0,1]$, that takes into consideration the Euclidean distance among the vertices $\bv^k_{i}$ and $\bv^{k+1}_{\ell}$ that is as:
\begin{equation}\label{eq_W}
	w^{k+1,\ell}_{i}=W\left (\frac{\|\bv^{k+1}_{\ell} -\bv^k_{i}\|}{2^{-k}L}\right ),\quad  i\in\cB^{k+1,\ell}.
\end{equation}
By the definition of $\cB^{k+1,\ell}$ in \eqref{eq_stencil_definition}, the weight function $W$ is always evaluated in $[0,1)$.

Among the several possible choices, in Section \ref{sec:num} we have experimented with the constant function $W(\cdot)=1$ (corresponding to the case where no weights are taken in to account), the hat function $W(\cdot)=1-\cdot$, and the Gaussian function $W(\cdot)=e^{-(2.5\cdot)^2/2}$.

\medskip
{We continue by observing several important facts.
\begin{rmk} \label{rmk_uniform_case_rule} Properties of the coefficients:
	\begin{enumerate}[(i)]
		
\item Since the coefficient $\alpha^{k+1,\ell}_i$  and the  weight $w^{k+1,\ell}_{i}$ depend on $\ell$ and $k$, the subdivision scheme is level-dependent and non-uniform meaning that, a priori, the rules change with both the level and with the location. This will be not the case of uniform grids, where the subdivision schemes become level-independent and uniform, meaning that each $\alpha^{k+1,\ell}_i$ in \eqref{newrules} only depends on $i$ and the parity of $\ell$;

		\item \label{rmk_uniform_case_rule_2}
		According to Remark \ref{rmk_well_defined}-\eqref{rmk_well_defined_2}, in a grid configuration where
		\begin{equation} \label{eq_sum_zero}
			\sum_{i\in\cB^{k+1,\ell}}w^{k+1,\ell}_{i} (\bv^k_{i} - \bv^{k+1}_\ell) = 0,
		\end{equation}
		we simply have that
		\begin{equation} \label{eq_rule_uniform}
			z_\ell^{k+1}=\sum_{i\in\cB^{k+1,\ell}} \alpha^{k+1,\ell}_i z^k_{i}, \quad \alpha^{k+1,\ell}_i = \frac{w^{k+1,\ell}_{i}}{
			\displaystyle{\sum_{i\in\cB^{k+1,\ell}}}w^{k+1,\ell}_{i}
			};
		\end{equation}
		Expression \eqref{eq_sum_zero} can be reformulated  as
\[	\bv^{k+1}_\ell = \frac{\sum_{i\in\cB^{k+1,\ell}}w^{k+1,\ell}_{i} \bv^k_{i}}{\sum_{i\in\cB^{k+1,\ell}}w^{k+1,\ell}_{i}},\]
		meaning that $\bv^{k+1}_\ell$ is the weighted average of the vertices in the $L$-ball of $\bv^{k+1}_\ell$, $\cB^{k+1,\ell}$. In particular, this holds true in uniform triangulations (see Lemma \ref{lem:sum_weights});

\item In a general grid, some coefficients $\alpha^{k+1,\ell}_i$ may be negative, in contrast with the case \eqref{eq_sum_zero}. For example, the specific location of the vertices in Figure \ref{fig:negative_weights} leads to negative coefficients (see Section Reproducibility for details).
	\end{enumerate}
\end{rmk}		
\begin{rmk}\label{rmk_uniform_case_rule_3} Properties of the new scheme:
\begin{enumerate}[(i)]
		\item\label{onL} According to Remark \ref{rmk_well_defined}-\eqref{rmk_well_defined_collinear}, it is necessary that all $L$-balls contain at least a triangular face, because it guarantees that the degree-$1$ polynomial regression can be performed. For this purpose, we demand $L$ to be larger than the \emph{triangulation diameter}:
		\begin{equation}\label{L}
			L > \sup\{\|\bv^{0}_{i} -\bv^0_{j}\| \ : \ (i,j)\in E^0\}.
		\end{equation}
		Since every triangle is subdivided on every iteration, the triangulation diameter halves at each iteration and, then, we can halve $L$ at each iteration as well, which motivates the power of 2 in the definitions of $\cB^{k+1,\ell}$ and $w^{k+1,\ell}_{i}$, in \eqref{eq_stencil_definition} and \eqref{eq_W};

		\item In contrast with other subdivision schemes, the insertion and replacement rules are conceptually the same;
		
		\item The proposed rules can be directly applied to finite triangulations, as in practice. There is no need to provide some special procedure for data near the boundary;
		
		\item Similarly, the rules can be applied even in presence of extraordinary vertices (irregular triangulations);
		
		\end{enumerate}
\end{rmk}

We continue by introducing the following notion.
\begin{defi}[0-homogeneouty]
A subdivision scheme is said to be \emph{0-homogeneous}
 with respect to the simultaneous scaling of the grid and the parameter $L$ if $$z^{k+1}_\ell(h L,h V^k)= z^{k+1}_\ell(L,V^k).$$
\end{defi}

\begin{lem}\label{item_0_homogeneous} The subdivision scheme with rules in \eqref{newrules}
is 0-homogeneous respect to the simultaneous scaling of the grid and the parameter.
\end{lem}

\begin{proof}Obviously, the quantities $z^{k+1}_\ell, \alpha^{k+1,\ell}_i, w^{k+1,\ell}_{i},\mu^{k+1,\ell}_i$ and $\cB^{k+1,\ell}$ are functions of $L$ and $V^k$. But, we easily see that $\cB^{k+1,\ell}(L,V^k)= \cB^{k+1,\ell}(hL,h V^k)$ for any $h>0$ and similarly that $w^{k+1,\ell}_{i}(L,V^k)=w^{k+1,\ell}_{i}(hL,h V^k)$ according to \eqref{eq_stencil_definition} and \eqref{eq_W}. Moreover we also see that $\mu^{k+1,\ell}_i(h L,h V^k)=h^4 \mu^{k+1,\ell}_i(L,V^k)$ and $\alpha^{k+1,\ell}_i(h L,h V^k)= \alpha^{k+1,\ell}_i(L,V^k)$. As a consequence, $z^{k+1}_\ell(h L,h V^k)= z^{k+1}_\ell(L,V^k)$, which is the \emph{0-homogeneouty}  with respect to the simultaneous scaling of the grid and the parameter $L$;
\end{proof}

\begin{rmk}\label{LW}
We finally remark that the parameter $L$ does affect the subdivision scheme and it is crucial that its selection is made in order to
 guarantees that the degree-1 polynomial regression can be performed at each step (see Remark \ref{rmk_uniform_case_rule_3}-\eqref{onL}).  Obviously different values of $L$ generate different schemes as shown in the numerical tests. Increasing $L$ improves denoising, but worsens the approximation capability.
As for the parameter $L$, the proposed subdivision schemes are sensitive to the weights selection. Even though they can be selected in many ways, our idea is to adopt a classical approach based on a weight function that takes into consideration the Euclidean distance among the vertices and therefore also of the iterations. For clarity, the weight functions we use in our examples are explicitly indicated.
\end{rmk}}

\begin{figure}[!h]
	\centering
	\resizebox{0.8\textwidth}{!}{
	\begin{tikzpicture}
\draw[line width=0.3mm,   black] (5.00,5.00) -- (0.00,6.25);
\draw[line width=0.3mm,   black] (0.00,6.25) -- (1.25,6.25);
\draw[line width=0.3mm,   black] (1.25,6.25) -- (5.00,5.00);
\draw[line width=0.3mm,   black] (5.00,5.00) -- (1.25,6.25);
\draw[line width=0.3mm,   black] (1.25,6.25) -- (2.50,6.25);
\draw[line width=0.3mm,   black] (2.50,6.25) -- (5.00,5.00);
\draw[line width=0.3mm,   black] (5.00,5.00) -- (2.50,6.25);
\draw[line width=0.3mm,   black] (2.50,6.25) -- (3.75,6.25);
\draw[line width=0.3mm,   black] (3.75,6.25) -- (5.00,5.00);
\draw[line width=0.3mm,   black] (5.00,5.00) -- (3.75,6.25);
\draw[line width=0.3mm,   black] (3.75,6.25) -- (10.00,6.25);
\draw[line width=0.3mm,   black] (10.00,6.25) -- (5.00,5.00);
\draw[line width=0.3mm,   black] (5.00,5.00) -- (10.00,6.25);
\draw[line width=0.3mm,   black] (10.00,6.25) -- (8.75,3.75);
\draw[line width=0.3mm,   black] (8.75,3.75) -- (5.00,5.00);
\draw[line width=0.3mm,   black] (5.00,5.00) -- (8.75,3.75);
\draw[line width=0.3mm,   black] (8.75,3.75) -- (0.00,6.25);
\draw[line width=0.3mm,   black] (0.00,6.25) -- (5.00,5.00);
\filldraw[black] (5.00,5.00) circle (1.5pt);
\node[anchor=south west] at (5.00,5.00) {$\frac{249}{1025}$};
\filldraw[black] (0.00,6.25) circle (1.5pt);
\node[anchor=south west] at (0.00,6.25) {$\frac{128}{1025}$};
\filldraw[black] (1.25,6.25) circle (1.5pt);
\node[anchor=south west] at (1.25,6.25) {$\frac{22}{205}$};
\filldraw[black] (2.50,6.25) circle (1.5pt);
\node[anchor=south west] at (2.50,6.25) {$\frac{92}{1025}$};
\filldraw[black] (3.75,6.25) circle (1.5pt);
\node[anchor=south west] at (3.75,6.25) {$\frac{74}{1025}$};
\filldraw[black] (10.00,6.25) circle (1.5pt);
\node[anchor=south west] at (10.00,6.25) {$-\frac{16}{1025}$};
\filldraw[black] (8.75,3.75) circle (1.5pt);
\node[anchor=south west] at (8.75,3.75) {$\frac{388}{1025}$};
\end{tikzpicture}		
	}
	\caption{An example of stencil configuration leading to negative coefficients $\alpha^{k+1,\ell}_i$ (the numbers in the graphic), which has been obtained considering all weights equal to 1 and the vertex coordinates (from top to bottom, from left to right) $(-4,1),(-3,1),(-2,1),(-1,1),(4,1),(0,0),(3,-1)$. {The new vertex would be inserted at the origin.} See Section Reproducibility to recreate this example.}
	\label{fig:negative_weights}
\end{figure}

%
%

\section{Properties of the new weighted least squares degree-$1$ polynomial subdivision schemes}\label{sec:properties}
This section is dedicated to the analysis of the properties satisfied by the new subdivision schemes presented in Section \ref{sec:new_subdivision}. It is well known that, for subdivision schemes dealing with a general type of grid,  the properties investigation is very challenging
since the needed tools depend on the type of grid, and the more general the geometry and the topology of the grid are, the harder is to derive properties. We recall that, generally speaking, the topology of the grid is related to the \emph{regularity} of a grid, while the geometry of the grid is related to its \emph{uniformity}.

\subsection{Types of grids}

The most simple type of grid is the so-called \emph{triangular-rectangular} grid, consisting of the $\Z^2$ vertices connected by edges moving in the directions $(1,0),(1,1),(0,1)$ (see Figure \ref{fig:regular}-(a)). The next level of complexity is the \emph{uniform} grid, which is obtained by applying an invertible linear transformation to the triangular-rectangular grid,
which is therefore a uniform grid itself. Hence, a uniform grid consists of the vertices connected by edges along three principal directions.
Uniform grids are \emph{regular}, since all the vertices has \emph{valence}\footnote{The number of edges that meet at a vertex is it valence} 6. A well studied  example of uniform grid is the \emph{equilateral} grid, shown in Figure \ref{fig:regular}-(b).

\begin{figure}[!h]
	\centering
	\begin{tabular}{cccc}
		\resizebox{0.23\textwidth}{!}{
		\input{images/rectangular_grid.txt}
		}&
		\resizebox{0.23\textwidth}{!}{
		\input{images/equilateral_grid.txt}
		}&
		\resizebox{0.23\textwidth}{!}{
		\input{images/regular_nonuniform.txt}
		}&
		\resizebox{0.23\textwidth}{!}{
		\input{images/irregular.txt}
		}\\
		(a) & (b) & (c) & (d)
	\end{tabular}
	\caption
	{
	Examples of grids with different regularities and uniformities. (a) The triangular-rectangular grid and (b) the equilateral grid are regular and uniform. A regular non-uniform grid (c) and an irregular non-uniform grid (d) are also shown.
	}
	\label{fig:regular}
\end{figure}

A further level of complexity is given by a \emph{regular non-uniform} grid, meaning that all vertices has valence six but the vertex distribution and the vertices are more freely located, as in Figure \ref{fig:regular}-(c). Finally, the most general grid is the \emph{irregular non-uniform} grid, where the valence of the vertices varies along the grid as in Figure \ref{fig:regular}-(d). As already mentioned, regularity refers to the topology of the grid, while uniformity refers to the geometry of the grid.

The proposed subdivision scheme belongs to a special class of irregular schemes called \emph{semi-regular}, as many of the subdivision schemes investigated in the literature. Semi-regular means that, even if the initial grid is non-regular, the grid refinement rules are such that, after some iterations, the refined grid is regular and uniform `by patches'. Figure \ref{fig:triangulation_after_2_iterations} shows a 2-step refinement of a non-regular initial grid refined by the mid-point addition procedure, where the \emph{semi-regularity} is evident.

For this type of schemes, as ours, the analysis must be conducted differently according to different types of grid regions: In the inner of a patch, at the edge where two patches intersect and at the vertices where several patches meet.
In the inner of each patch, the grid is uniform and so is the scheme, and therefore it can be analysed according to the next Theorem \ref{tmh:uniform}. At the edge separating two patches, excluding the extremes of that edge, all the vertices are regular hence the analysis can be conducted as in \cite{LEVINLEVIN} involving two patches.
At the vertices where several patches meet, convergence analysis can be done by using the notion of \emph{characteristic map} proposed to show convergence around extraordinary vertices of any valence \cite{Reif95,Umlauf00}.  Observe that, if there are extraordinary vertices in the initial grid, they become  \lq isolated\rq\,  from each other along the iterations.

As we see, the three situations have been addressed in the literature, so that no further theory development is required. However, the application of the characteristic map is non-trivial and out of the scope of this paper. For the analysis of \lq semi-regular\rq\, subdivision scheme we refer to \cite{LEVINLEVIN}. Here, we continue with the analysis for uniform grids {while, in case of non-regular grids, we provide numerical validation of continuity in the Numerical Results section.}

\subsection{Convergence for uniform grids \label{sec_convergence_uniform}}

In this subsection, we first state a general notion of convergence and then use it to show that our new scheme is convergent in case of unbounded uniform grids, whose vertices at level $k$ form a lattice:
\begin{equation}\label{grid}
	\bv^k_\bl = 2^{-k} [\be^1, \be^2] \bl = 2^{-k} \ell_1 \be^1 + 2^{-k} \ell_2 \be^2 \bl, \quad  \bl\in\Z^2,
\end{equation}
where $[\be^1, \be^2]$ is the matrix with columns $\be^1$ and $\be^2$, which must be invertible.
Note that, for convenience, and without loss of generality, now we are indexing the vertices by $\bl\in\Z^2$ and not by $\ell\in \N$.

Let $\bz^0=\{z^0_\bl\}_{\bl\in\Z^2}\subset\R$ be the initial data set and let $\bz^k=\{z^k_\bl\}_{\bl\in\Z^2}\subset\R$ be the $k$-refined data set, obtained via one member of the new family of subdivision schemes specified in Section 3.
We say that the subdivision scheme is uniformly convergent if there exists $F\in C(\R^2,\R)$ (called the limit function) such that
$$
\lim_{k\rightarrow \infty}\sup_{\bl\in\Z^2} \left|z^k_\bl-F(\bv^k_\bl) \right|=0.
$$

In the following statements, we assume that $k\geq0$ is indicating the subdivision step, while $\bl\in\Z^2$ a generic location. We also assume that $\cB^{k+1,\bl}$ is denoting the $L$-ball of $\bv^{k+1}_\bl$ and that $w^{k+1,\bl}_ \bi,\ \bi\in\cB^{k+1,\bl},$ are the weights of the rules, respectively defined in as \eqref{eq_stencil_definition} and \eqref{eq_W}, but for indices in $\Z^2$. That is,
\begin{equation} \label{eq_stencil_definition2}
	\cB^{k+1,\bl} := \{ \bi \in \Z^2: \|\bv^{k+1}_\bl - \bv^{k}_\bi\| < L\},\quad
	w^{k+1,\bl}_{\bi}=W\left (\frac{\|\bv^{k+1}_{\bl} -\bv^k_{\bi}\|}{2^{-k}L}\right ).
\end{equation}

{We continue with a technical lemma discussing invariance properties of the stencils and of the vertices distance in case of uniform grids.}

\begin{lem} \label{lem_parity}
	Let $k\geq0$, $\bl\in\Z^2$. For a uniform grid, the stencil $\cB^{k+1,\bl}$ defined in \eqref{eq_stencil_definition2} is a shift of some stencil $\cB^{k+1,\bar \bl}$, where $\bar \bl\in\{0,1\}^2$. More in details, for $\bl = 2\dot \bl + \bar \bl\in\Z^2$, with $\dot \bl\in\Z^2$ and $\bar \bl\in\{0,1\}^2$, we have
	$$\cB^{k+1,\bl} = \cB^{k+1,\bar \bl} + \dot \bl.$$
	In addition, we have
	\begin{equation} \label{eq_difference_v}
		\bv^{k+1}_\bl - \bv^k_{\bi}  = \bv^{k+1}_{\bar\bl} - \bv^k_{\bi - \dot \bl}, \quad \bi\in \Z^2.	
	\end{equation}
\end{lem}
\begin{proof}
	The proof straightforwardly follows from \eqref{grid}.
\end{proof}

{ The next lemma discusses other important properties of uniform grids.}

\begin{lem} \label{lem:sum_weights}
	Let $k\geq0$, $\bl\in\Z^2$. On a uniform grid, let $\{\bv^k_\bi \}_{\bi\in\cB^{k+1,\bl}}$ be associated to the $L$-ball of the vertex $\bv^{k+1}_\bl$ and
	let $\{w^{k+1,\bl}_\bi\}_{\bi\in\cB^{k+1,\bl}}$ be the corresponding weights. Then,
\[
	w^{k+1,\bl}_\bi = w^{k+1,\bl}_{\bl - \bi}\quad \hbox{and}\quad	\sum_{\bi\in\cB^{k+1,\bl}} w^{k+1,\bl}_\bi (\bv^k_\bi - \bv^{k+1}_\bl) = 0.
	\]
\end{lem}
\begin{proof}
	Let $\bi\in\cB^{k+1,\bl}$ be. By the vertices definition \eqref{grid}, we deduce that
	\[
	\bv^k_{\bl-\bi} - \bv^{k+1}_\bl = \bv^{k+1}_\bl -\bv^k_{\bi}.
	\]
	Then, both $\bv^k_\bi,\bv^k_{\bl-\bi}$ are at the same distance from $\bv^{k+1}_\bl$, thus $\bl-\bi\in\cB^{k+1,\bl}$ and
	$$w^{k+1,\bl}_\bi = w^{k+1,\bl}_{\bl-\bi}.$$
	Then
	$$w^{k+1,\bl}_\bi (\bv^k_\bi-\bv^{k+1}_\bl) + w^{k+1,\bl}_{\bl-\bi} (\bv^k_{\bl-\bi}-\bv^{k+1}_\bl) = 0.$$
	Summing over all the vertices in the ball, we obtain the result.
\end{proof}

{ Lemmas \ref{lem_parity} and \ref{lem:sum_weights} are next used in Theorem \ref{tmh:uniform} to prove that, when dealing with uniform grids, the subdivision schemes are uniform, level-independent and with
positive coefficients. The latter are sufficient conditions for convergence, as expressed in Corollary \ref{coro41} which concludes the section.}

\begin{thm} \label{tmh:uniform}
	On a uniform grid, any subdivision scheme of the family based on the rules in \eqref{eq_rule_uniform} is uniform, level-independent and its coefficients are positive.
\end{thm}
\begin{proof}
	Due to Lemma \ref{lem:sum_weights}, we meet the requirements of Remark \ref{rmk_uniform_case_rule}-\eqref{rmk_uniform_case_rule_2}. Thus, the subdivision rules are given by:
	\begin{equation}\label{somme}
		z_\ell^{k+1} = \sum_{i\in\cB^{k+1,\ell}} \alpha^{k+1,\ell}_i z^k_{i}, \quad \alpha^{k+1,\ell}_i = \frac{w^{k+1,\ell}_{i}}{
		\displaystyle{\sum_{i\in\cB^{k+1,\ell}}}w^{k+1,\ell}_{i}
		}.
	\end{equation}
	Since the coefficients only depend on the weights, they are positive. Moreover, as a consequence of Lemma \ref{lem_parity}, the weights $w^{k+1,\bl}_i$ only depend on the parity of $\bl$, say on $\bar \bl\in\{0,1\}^2$: For each $\bl = 2\dot\bl + \bar\bl\in\Z^2$ and $\bi\in\cB^{k+1,\bl}$, we have that
	\[
	w^{k+1,\bl}_{\bi} \overset{\eqref{eq_stencil_definition2}}= W\left (\frac{\|\bv^{k+1}_{\bl} -\bv^k_{\bi}\|}{2^{-k}L}\right ) \overset{\eqref{eq_difference_v}}{=} W\left(\frac{\|\bv^{k+1}_{\bar\bl} - \bv^k_{\bi - \dot\bl}\|}{2^{-k}L}\right) \overset{\eqref{grid}}{=} W\left(\frac{\|\bv^{1}_{\bar\bl} - \bv^0_{\bi - \dot\bl}\|}{L}\right) = w^{1,\bar\bl}_{\bi - \dot\bl}.
	\]
	In other words, only the weights $w^{1,\bar\bl}_{\bj}$, $\bj\in\cB^{k+1,\bar \bl}$, $\bar \bl\in\{0,1\}^2$ are crucial meaning that the scheme is uniform and level-independent.
\end{proof}

\begin{cor}\label{coro41}
	On a uniform grid, any subdivision scheme of the family based on the rules in \eqref{eq_rule_uniform} is convergent.
\end{cor}
\begin{proof}
	Since the scheme is uniform, level-independent and all its coefficients are positive and due to \eqref{somme} sum up to $1$, the convergence is guaranteed, according to \cite{JiaZhou}.
\end{proof}

\subsection{Polynomial reproduction}

\medskip The simplest reproduction property of our new subdivision schemes based on the linear operators $S^{k}_{T^k}$  in \eqref{linear},  is their  $\Pi_1$-stepwise-reproduction, independently of the grid geometry.

\begin{prop} \label{prop:reproduction}
	Any subdivision scheme of the family presented in Section \ref{sec:subdivision_rules} stepwise reproduces $\Pi_1$, meaning that
	\[
	S^{k}_{T^k} \left( p|_{V^{k}} \right) = p|_{V^{k+1}}, \qquad \forall k\geq 0, \ \forall p\in\Pi_{1},
	\]
	which implies that the subdivision process converges to $p$ when the initial data is $p|_{V^{0}}$.
\end{prop}
\begin{proof}
	If $\bz^k = p|_{V^k}$, then the polynomials solving the least squares problems at the $k$ iteration must be $p$. Since $\bz^{k+1}$ is obtained by evaluating those polynomials at the vertices of $V^{k+1}$, then $\bz^{k+1} = p|_{V^{k+1}}$.
\end{proof}

\subsection{Approximation order}

Along this section, we will apply scaling operations to the grid in order to observe the approximation capability of the scheme. Let $T = (V,E)$ be any triangulation and let $T_h = (h V,E)$  be the same triangulation but
with all vertices scaled by a factor $h>0$ and $T_h^k = (h V^k, E^k)$ be the grid resulting after $k$ grid-refinements. By Lemma \ref{item_0_homogeneous}, we have that the scheme is $0$-homogeneous with respect to the scaling of the grid, meaning that $S_{T_h} = S_{T}$, for any $h>0$, provided that $L$ is scaled by the same factor, which we assume. Thus, 
we will simply write 
$S_k$ instead of $S_{T_h^k} = S_{T^k}$, for any $k\geq 0$.
Let us further simplify the notation by denoting $F_k := F|_{h V^k}$.

\begin{lem} \label{lem:step_approximation}
	Let $F:\R^{2}\lra\R$ be a $\cC^{2}$ function with bounded derivatives.
	Any subdivision scheme of the family presented in Section \ref{sec:subdivision_rules} has second order approximation, stepwisely, meaning that
	\[
	\exists C(V^0,F)>0, \, h_0(V^0,F)>0 \ : \ \| S_{k} F_k - F_{k+1} \|_\infty \leq C \left(2^{-k}h\right)^2, \quad \forall h<h_0, \ \forall k\geq 0.
	\]
\end{lem}
\begin{proof}
	This result follows from two key observations. First, each triangle in $T_h^k$ is subdivided into four \emph{similar} triangles to form $T^{k+1}_h$, each having edges that are half the previous length, thereby producing a \emph{semiregular} triangulation (see \cite{DAUBECHIES}). Second, since local degree-$1$  polynomial regression is applied, $C$ is a constant that depends on $V^0$ and on the bounds on the second derivatives of $F$.
\end{proof}

We continue with a result dealing with the approximation error in the limit of the subdivision process, assuming it is convergent.
\begin{prop} \label{prop:approximation}
	Let $F:\R^{2}\lra\R$ be a $\cC^{2}$ function with bounded derivatives. If a scheme in Section \ref{sec:subdivision_rules} convergences for any data attached to $T^0$, then it has second order of approximation, meaning that
	\[
	\exists C_1(V^0,F)>0, \, h_0(V^0,F)>0 \ : \ \left \| \left( \prod_{l=0}^k S_{l} \right) F_0 - F_{k+1} \right \|_\infty \leq C_1 h^2, \qquad \forall h<h_0, \ \forall k\geq 0.
	\]
\end{prop}
\begin{proof}
	Due the linearity and convergence of the scheme, we know that
	\begin{equation} \label{eq:boundS}
		\exists D\geq 1 : \ \left \| \left( \prod_{l=k-k_1}^{k} S_{l} \right) F_{k-k_1}\right \|_\infty \leq D \left \| F_{k-k_1}\right \|_\infty, \qquad \forall 0 \leq k_1 \leq k.
	\end{equation}
	Given $k\geq 0$, we will prove by induction on $k_1$ that
	\begin{equation} \label{eq_induction}
		\left \| \left( \prod_{l=k-k_1}^k S_{l} \right) F_{k-k_1} - F_{k+1} \right \|_\infty \leq C D h^2\sum_{l=k-k_1}^{k} 4^{-l}, \qquad 0\leq k_1\leq k,
	\end{equation}
	where $C(V,F)$ is the constant of Lemma \ref{lem:step_approximation}. If we do so, the claim will be proven taking $C_1 = \frac43 CD$ and $k_1 = k$.
	
	First, observe that the case $k_1=0$ consists in asserting that
	\[
	\left \| S_k F_{k} - F_{k+1} \right \|_\infty \leq C D h^2 4^{-k},
	\]
	which certainly holds by Lemma \ref{lem:step_approximation} and the fact that $D\geq 1$. Next, assuming that \eqref{eq_induction} holds for $k_1-1$, we will prove it for $k_1$. Indeed,
	\begin{align*}
		\left \| \left( \prod_{l=k-k_1}^{k} S_l \right) F_{k-k_1} - F_{k+1} \right \|_\infty
		\leq& \left\| \left( \prod_{l=k-k_1}^{k} S_l \right) F_{k-k_1}-\left( \prod_{l=k-k_1+1}^{k} S_l \right) F_{k-k_1+1}  \right\|_\infty\\
		&+ \left\| \left( \prod_{l=k-k_1+1}^{k} S_l \right) F_{k-k_1+1}  - F_{k+1} \right\|_\infty\\
		=& \left\| \left( \prod_{l=k-(k_1-1)}^{k} S_l \right) \left(S_{k-k_1} F_{k-k_1}- F_{k-(k_1-1)} \right)  \right\|_\infty\\
		&+ \left\| \left( \prod_{l=k-(k_1-1)}^{k} S_l \right) F_{k-(k_1-1)} - F_{k+1} \right\|_\infty.
	\end{align*}
	Now, using \eqref{eq:boundS} and the induction hypothesis, and Lemma \ref{lem:step_approximation} right after, we obtain
	\begin{align*}
		\left \| \left( \prod_{l=k-k_1}^{k} S_l \right) F_{k-k_1} - F_{k+1} \right \|_\infty
		&\leq D \left\| S_{k-k_1} F_{k-k_1}- F_{k-(k_1-1)} \right\|_\infty + C D h^2\sum_{l=k-(k_1-1)}^{k} 4^{-l}\\
		&\leq D C h^2 4^{-(k-k_1)} + C D h^2\sum_{l=k-(k_1-1)}^{k} 4^{-l}
		= C D h^2\sum_{l=k-k_1}^{k} 4^{-l}.
	\end{align*}
\end{proof}

\subsection{Computational cost}

{
We have already observed that on uniform grids, where all vertices are regular and form a lattice, the resulting subdivision schemes are uniform and stationary and, thus, they are very efficient, as all subdivision schemes are. On these types of grids, the number of vertices inside an $L$-ball is basically proportional to the area of a circle with radius $L$, hence $\cO(L^2)$. By uniformity, the subdivision coefficients can be computed only once and quickly obtained using the formula of Remark  \ref{rmk_well_defined}-\eqref{rmk_well_defined_2}.

On general grids, the computational cost is higher. Indeed, on the one hand, the refinement process around a vertex requires identifying the surrounding vertices and triangles, which is an expensive task that certainly needs to be optimized. On the other hand, the subdivision coefficients must be obtained for each new vertex, for example by solving the linear system related to the polynomial regression. In the univariate case investigated in \cite{LY24}, the authors found a much more efficient way to compute the coefficients by observing that they can be obtained through direct polynomial evaluation. Future research will explore whether this property also holds in the bivariate case.

A reference to our MATLAB implementation, though not optimized, can be found in the Reproducibility section.}

\subsection{Denoising capability}

This section is to analyse the noise reduction capability of the schemes presented  in Section \ref{sec:subdivision_rules}. To this purpose, let us consider $\bz^0 = (z^0_i)_{i=1}^{N^0}$ such that
\[ z^0_i = F(\bv^0_i) + \epsilon^0_i,\quad i=1,\ldots,N^0\]
where $F$ is a smooth function and each $\epsilon^0_i$ is a random value that represents noise in the data, which we suppose \emph{independent and identically distributed (iid)} with certain distribution  $X$, that is, $\epsilon^0_i \sim X$, $\forall i \in\{1,\ldots,N^0\}$ where
the expectancy of the noise is zero, that is $\mean(X)=0$.
Due to the linearity of the subdivision operator, we have that
\[
\bz^{k+1} = S^k(\bz^k)  = S^k( \bff^k ) + S^k (\beps^k)  = \bff^{k+1} + \beps^{k+1},\quad \bff^{k+1} := S^k( \bff^k ),\quad \beps^{k+1} := S^k( \beps^k ),
\]
so that we can be studied separately the case when the data is smooth and the data is pure noise.

The first case has been already studied in last subsection.
The second case is studied in the next proposition. Let us denote by $X^{k}_\ell$ the probability distribution of the random value $\beps^k_\ell$. We are going to see that the variance $X^{k}_\ell$ is lesser than the variance of the initial noise $X$, meaning that the subdivision schemes are able to reduce the noise in the data.
\begin{prop}
	If $\beps^0$ is noise iid and the coefficients $\{\alpha^{k+1,\ell}_i\}_{i,\ell,k}$ are non-negative, then
\[
	\var(X^{k}_\ell) \leq \theta \var(X),\]
	where
\begin{equation}\label{theta}
\theta = \max_{\ell\in \{1,\ldots,N^{1}\}}\sum_{i\in\cB^{1,\ell}} (\alpha^{1,\ell}_i)^2
\end{equation} and $\var(X)$ is the variance of $X$.
\end{prop}
\begin{proof}
	Let us start by checking it for the first step. By \eqref{newrules},
	\[ \beps^1_\ell = \sum_{i\in\cB^{1,\ell}} \alpha^{1,\ell}_i \epsilon^0_{i}.\]
		
	Since we assume that the initial noise $\beps^0$ is uncorrelated, using  \(\var(\sum_i X_i) = \sum_i\var(X_i) + \sum_{i\neq j}\cov(X_i,X_j)\), where $X_i$ are random variables and \(\cov(X_i,X_j)\) is the covariance, and \(\var(\alpha X_i) = \alpha^2 \var(X_i) \), with $\alpha\in\R$, we arrive at

	\[
	\var{(X^1_\ell)} = \sum_{i\in\cB^{1,\ell}} \var(\alpha^{1,\ell}_i X) = \sum_{i\in\cB^{1,\ell}} (\alpha^{1,\ell}_i)^2 \var(X) \leq \theta \var(X).
	\]
	From above, we prove the claim for $k=1$. It is not possible to apply the same argument again for $k > 1$, since the random variables \((X^k_\ell)_{\ell=1}^{N^{k}}\) are correlated (because neighbours $\epsilon^k_\ell,\epsilon^k_j$ are defined from common values of $\beps^{k-1}$). Instead, for $k>1$
	we use that $X^{k+1}_\ell$ is a convex combination of the $X^k$ elements, $X^{k+1}_\ell = \sum_{i\in\cB^{k+1,\ell}} \alpha^{k+1,\ell}_i X^k_i$, since the coefficients are non-negative and sum to one. Thus,
	\[
	\var{(X^{k+1}_\ell)} = \var\left(\sum_{i\in\cB^{k+1,\ell}} \alpha^{k+1,\ell}_i X^k_i\right) \leq \sum_{i\in\cB^{k+1,\ell}}\alpha^{k+1,\ell}_i \var(X^k_i),
	\]
	and now, applying this argument recursively, we obtain
	\[
	\var{(X^{k+1}_\ell)} \leq \sum_{i\in\cB^{k+1,\ell}}\alpha^{k+1,\ell}_i \var(X) = \var(X).
	\]
	\end{proof}
{	
\begin{rmk}
		If $\alpha^{1,\ell}_i>0$ $\forall i,\ell$ and $|\cB^{1,\ell}|>1$ $\forall \ell$ (where $|\cdot|$ tis he cardinality of the set) then it must hold $0< \alpha^{1,\ell}_i < 1$. As a consequence, we have that
		\[
		\theta = \max_\ell \sum_{i\in\cB^{1,\ell}} (\alpha^{1,\ell}_i)^2 < \max_\ell \sum_{i\in\cB^{1,\ell}} \alpha^{1,\ell}_i = 1.
		\]
		This means that the noise reduction is guaranteed, provided that $L$ is large enough to have $|\cB^{1,\ell}|>1$ $\forall \ell$, which is a reasonable assumption. Observe also that this formula can be used to choice $W$ such that $\theta$ is minimized. \end{rmk}
		

We conclude this section mentioning that the noise-reduction properties of different choices of the function $W$ is certainly worth to be investigated as done in \cite{LY24} for the univariate case. Indeed, in \cite{LY24}, the impact of the choice of $W$ on noise reduction is analyzed in detail, including, among other aspects, the computation of the $\theta$ values. Here, Table~\ref{tab:coeff.} lists the values of $\theta$ associated with the different weight functions. From the table, where different values of $L$ are considered, we see that the scheme corresponding to $W(x)=1$ guarantees the greatest noise reduction at each iteration. However, as shown in \cite{LY24}, a high noise-removal capability is compensated by a low approximation.}

\begin{table}[h]
 \centering
 {
	\begin{tabular}{lrrrrr}
	\multicolumn{6}{c}{\text{Rectangular grid}}\\\hline
		 $W(\cdot)$         & $(L=1.6)\,\,$   $\theta$ & $(L=2.5)\,\,$   $\theta$ & $(L=5)\,\,$   $\theta$ & $(L=10)\,\,$   $\theta$ & $(L=20)\,\,$ $\theta$ \\\hline
		 $ 1$           &    1.2500e-01                   &        6.2500e-02               &       1.4493e-02                 &       3.2787e-03                 &    8.0321e-04                \\
         $1-\cdot$          &    2.3035e-01                   &        7.8806e-02               &       1.9427e-02                 &       4.7859e-03                 &    1.1940e-03                    \\
		 $1-\cdot^2$        &    2.2361e-01                   &        7.1389e-02               &       1.7294e-02                 &       4.2551e-03                 &    1.0614e-03                   \\
		 $(1-\cdot^2)^2$     &    2.4915e-01                   &        9.2644e-02               &       2.2937e-02                 &       5.7301e-03                 &    1.4324e-03      \\
		 $(1-\cdot^3)^3$     &    2.4994e-01                   &        9.7569e-02               &       2.3558e-02                 &       5.8849e-03                 &    1.4712e-03      \\
		  $(1-\cdot^2)^3$      &    3.1794e-01                   &        1.1803e-01               &       2.9111e-02                 &       7.2757e-03                 &    1.8189e-03      \\
		 $e^{-\cdot}$       &    1.3191e-01                   &        6.5471e-02               &       1.5307e-02                 &       3.4810e-03                 &    8.5365e-04                   \\\hline  \hline
	\multicolumn{6}{c}{\text{Equilateral grid}}\\\hline	
 $W(\cdot)$         & $(L=1.6)\,\,$   $\theta$ & $(L=2.5)\,\,$   $\theta$ & $(L=5)\,\,$   $\theta$ & $(L=10)\,\,$   $\theta$ & $(L=20)\,\,$ $\theta$ \\\hline
		 $ 1$           &       1.4286e-01                   &           5.2632e-02               &         1.1236e-02                 &       2.7473e-03                 &       6.8871e-04               \\
         $1-\cdot$          &       1.7456e-01                   &           6.7894e-02               &         1.6558e-02                 &       4.1378e-03                 &        1.0341e-03                   \\
		 $1-\cdot^2$        &       1.4889e-01                   &           5.9571e-02               &         1.4688e-02                &        3.6792e-03                 &       9.1931e-04                  \\
		 $(1-\cdot^2)^2$    &       1.9957e-01                   &           7.9818e-02               &         1.9859e-02                &         4.9622e-03                 &      1.2405e-03      \\
		 $(1-\cdot^3)^3$    &       2.1317e-01                   &           8.2085e-02               &         2.0389e-02                &         5.0965e-03                 &      1.2741e-03      \\
		 $(1-\cdot^2)^3$    &       2.5064e-01                   &           1.0097e-01               &         2.5207e-02                &         6.3010e-03                 &      1.5752e-03      \\
		 $e^{-\cdot}$       &       1.5329e-01                   &           5.5776e-02               &         1.1948e-02                &         2.9221e-03                 &      7.3239e-04                   \\\hline
	\end{tabular}}
	\caption{ {
Values of $\theta$ as in eq.\eqref{theta} with different functions
$W$ and  values of $L$; the same functions as those used in \cite{LY24} have been selected.}}
	\label{tab:coeff.}
\end{table}

\section{The new least squares $1$-polynomial approximation subdivision scheme for equilateral and for triangular-rectangular grids}\label{sub:regular}

This section is dedicated to the study of the new $L$-ball least squares subdivision schemes for equilateral and for triangular-rectangular grids. The equilateral grid is a uniform grid with $\be^1 = (1,0)$ and $\be^2 = (\frac12, \frac{\sqrt3}2)$ and the triangular-rectangular grid is a uniform grid with $\be^1,\be^2$ the canonical basis of $\R^2$. For both these special types of grids, in view of Theorem \ref{tmh:uniform}, only four sets of weights (according to the parity of the index) are used to refine the data associated to the replacement and insertion rules, because the weights only depend on the parity of the vertex location. Therefore, the coefficients of the linear combination can be explicitly computed and the resulting subdivision scheme is linear, level-independent (or stationary) and uniform. As a result, we can represent it simply by giving its mask $\ba$ consisting of the whole sets of rules coefficients (see \cite{C2,C3,C4} for all details concerned the use of masks in linear and  uniform subdivision schemes). Below, we specify the masks corresponding to some choices of $L$and of weight functions and show the picture of the corresponding basic limit functions. {In the masks, the indexing of the coefficients such that the circled number
is at position $(0,0)$.}

\subsection{Equilateral grid}

We set $\frac32 < L < \sqrt3$, leading to the stencils shown in Figure \ref{fig:equilateral_rings}. For $W(\cdot)=1$ the bivariate mask is
\begin{equation}
	\ba=\frac 1{70}\left(\begin{array}{cccccccc}
		0 & 0 & 0  &7 &  7 & 7  &7\\
		0  &0  &7  &10& 7  &10& 7\\
		0  &7  &7 & 7 & 7&  7  &7\\
		7  &10 &7  &\circled{10}& 7  &10& 7\\
		7  &7 & 7&  7  &7  &7 & 0\\
		7 & 10 &7 &  10 &7  &0 & 0\\
		7  &7  &7  &7  &0  &0 & 0
	\end{array}
	\right),
\end{equation}
while for $W(\cdot)=1-\cdot$, the mask is
\begin{equation}
	\ba=\left(\begin{array}{ccccccc} 0 & 0 & 0 & a & b & b & a\\ 0 & 0 & b & c & d & c & b\\ 0 & b & d & e & e & d & b\\ a & c & e & \circled{f} & e & c & a\\ b & d & e & e & d & b & 0\\ b & c & d & c & b & 0 & 0\\ a & b & b & a & 0 & 0 & 0 \end{array}\right),
\end{equation}
where $a = \frac{2\,L-3}{g}$, $b=\frac{2\,L-\sqrt{7}}{g}$, $c=\frac{L-1}{7\,L-6}$, $d=\frac{2\,L-\sqrt{3}}{g}$, $e=\frac{2\,L-1}{g}$, $f=\frac{L}{7\,L-6}$, $g=-2\,\left(\sqrt{3}-10\,L+2\,\sqrt{7}+4\right)$. It can be checked that all the coefficients are positive since $\frac32 < L < \sqrt3$, as more generally proved in Theorem \ref{tmh:uniform}.

\begin{figure}[!h]
	\centering
	\begin{tabular}{cccc}
		\resizebox{0.23\textwidth}{!}{
		\input{images/equilateral_ring_r.txt}
		}
		&
		\resizebox{0.23\textwidth}{!}{
		\input{images/equilateral_ring_h.txt}
		}
		&
		\resizebox{0.23\textwidth}{!}{
		\input{images/equilateral_ring_v.txt}
		}
		&
		\resizebox{0.23\textwidth}{!}{
		\input{images/equilateral_ring_d.txt}
		}
	\end{tabular}
{	\caption{
	Stencil selection (red dots) on a equilateral grid with $\frac32 < L < \sqrt{3}$.
	From left to right: replacement, horizontal insertion, vertical insertion and diagonal insertion.
	The green circle corresponds to $L=\frac32$, the blue circle to $L=\sqrt{3}$.}}
	\label{fig:equilateral_rings}
\end{figure}

{ Figure \ref{fig:equilateral_basic_limit} and Figure \ref{fig:equilateral_basic_limit_L20} show the basic limit functions for different values of $L$ and $W$ in case of equilateral grids. From the figures we see that  when we significantly increase $L$, the support tends to be circular and the basic limit function seems to approach a radially symmetric function. In spite of the specific value of $L$, for $W(\cdot)=1$, the inner part is quite flat, while for $W(\cdot)=1-\cdot$ the slope is similar to the slop of a linear function.}

\begin{figure}[!h]
	\centering
	\begin{tabular}{cc}
		\includegraphics[width=0.4\linewidth,clip]{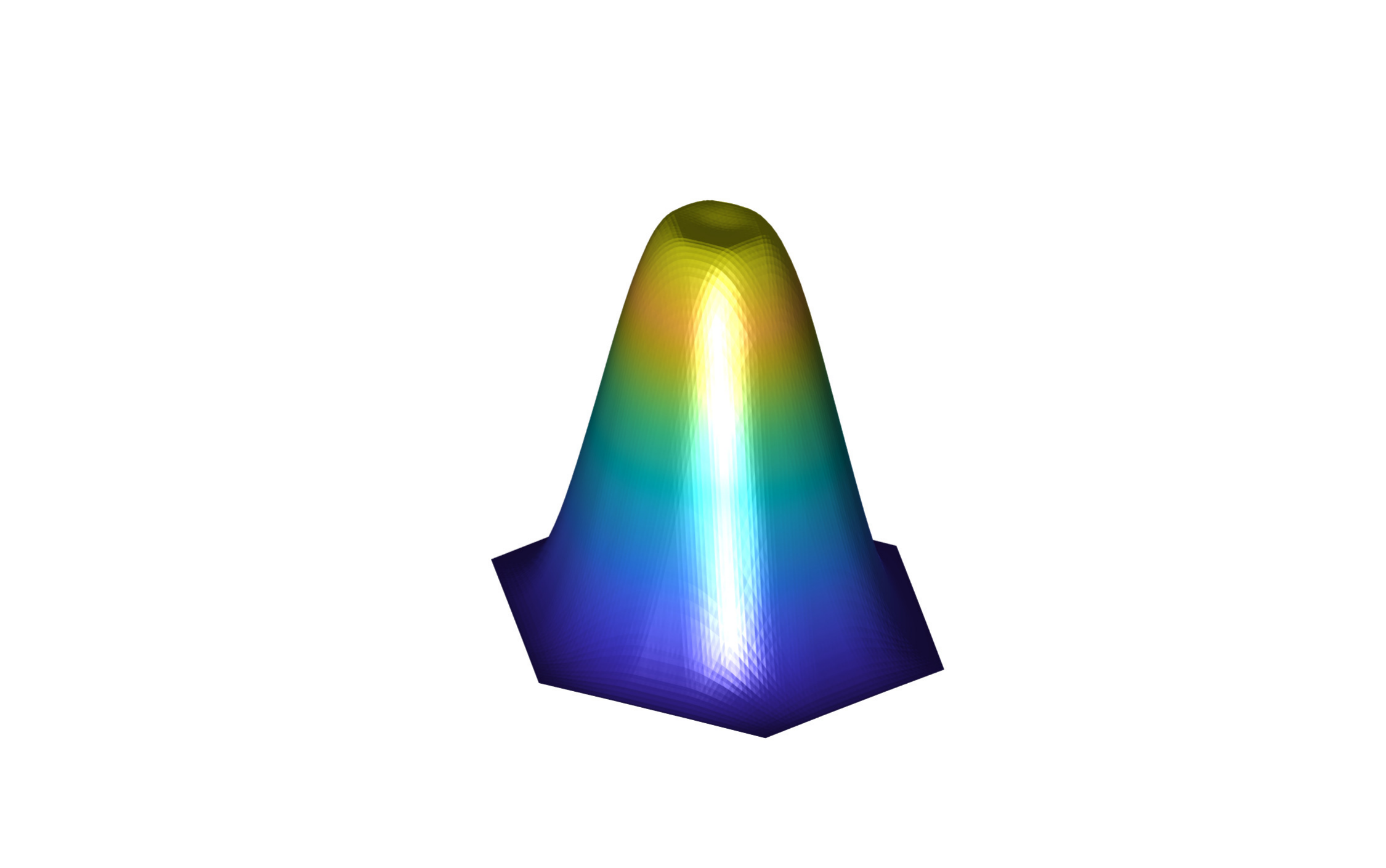} &
		\includegraphics[width=0.4\linewidth,clip]{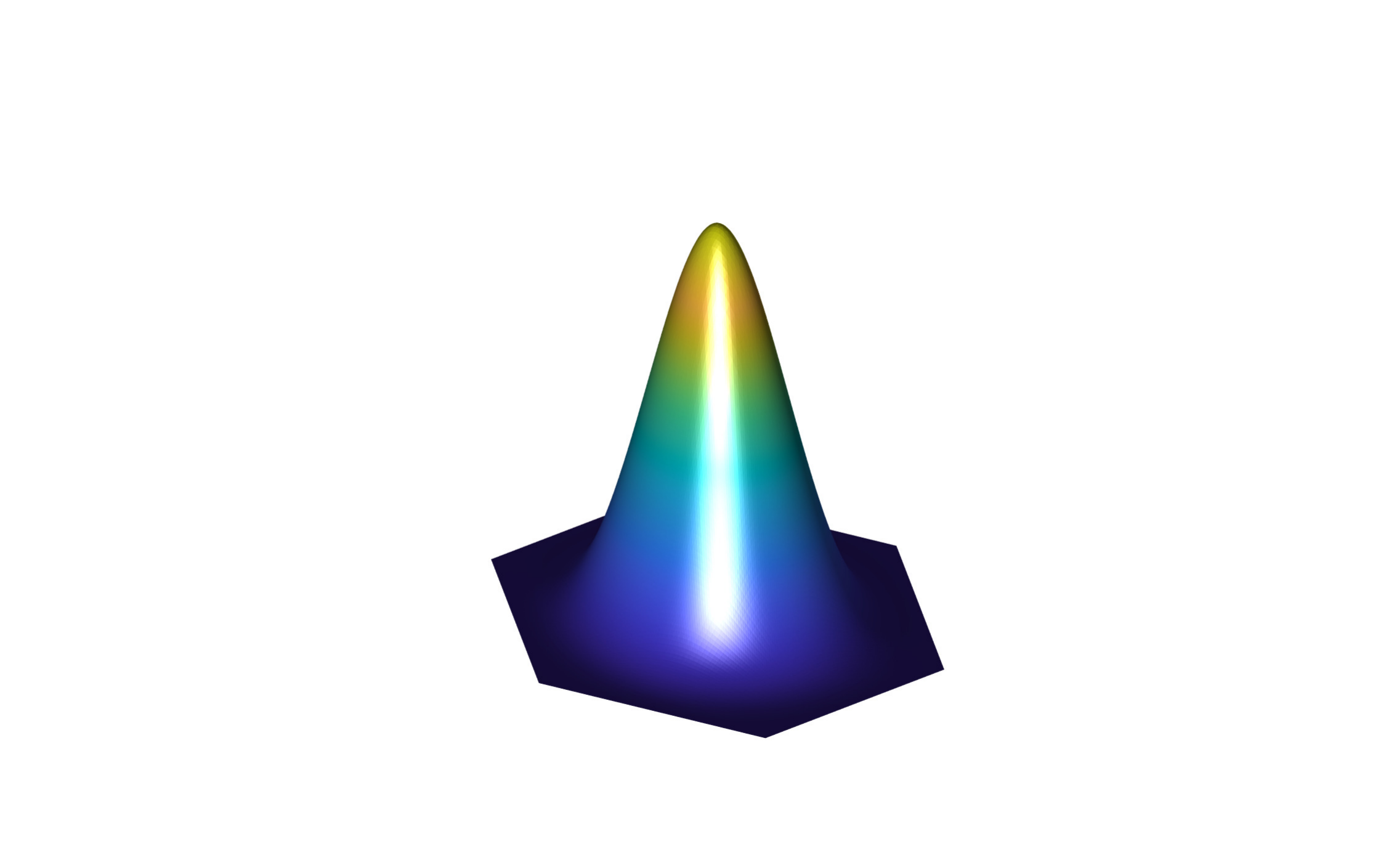} \\
		\includegraphics[width=0.4\linewidth,clip]{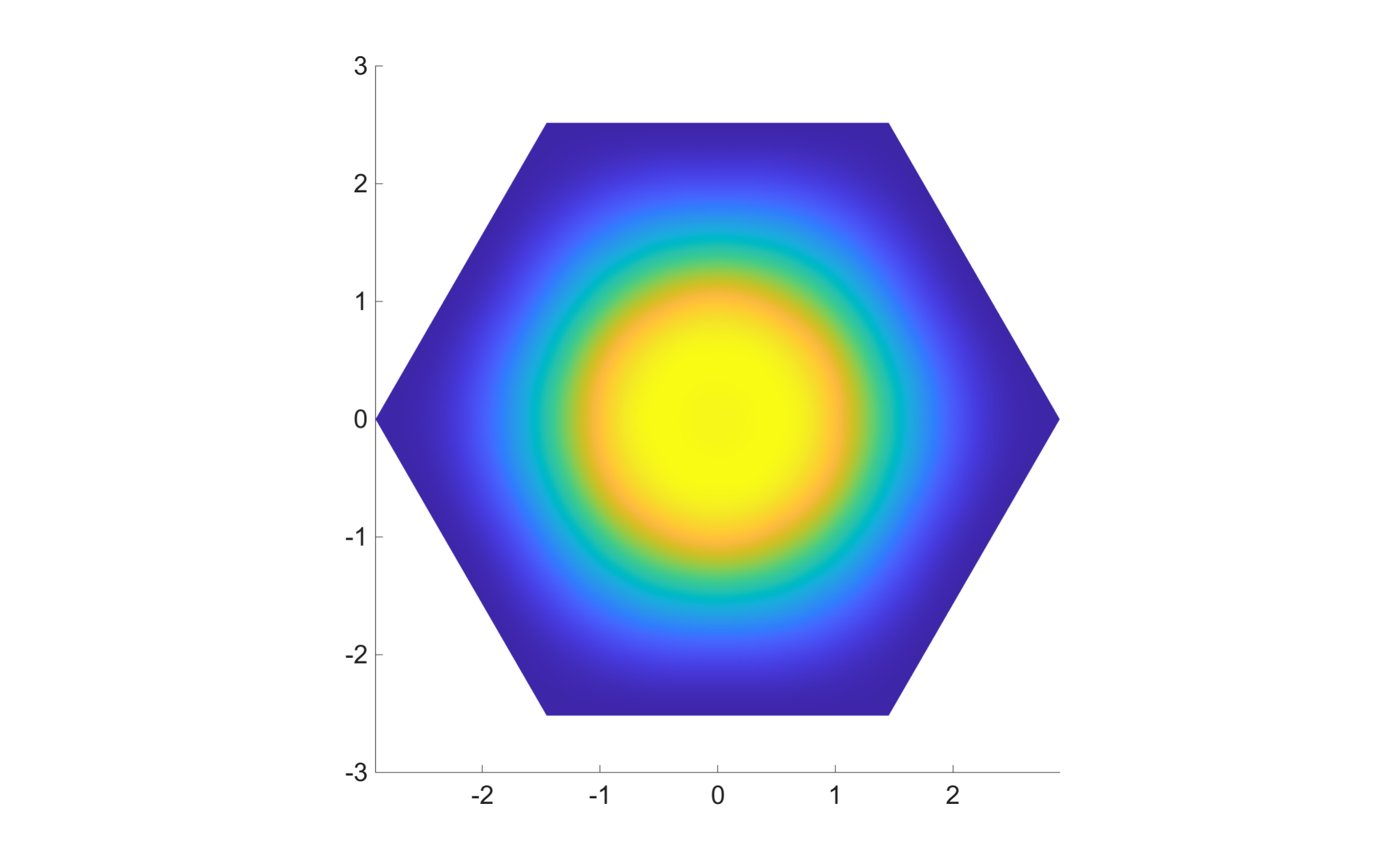} &
		\includegraphics[width=0.4\linewidth,clip]{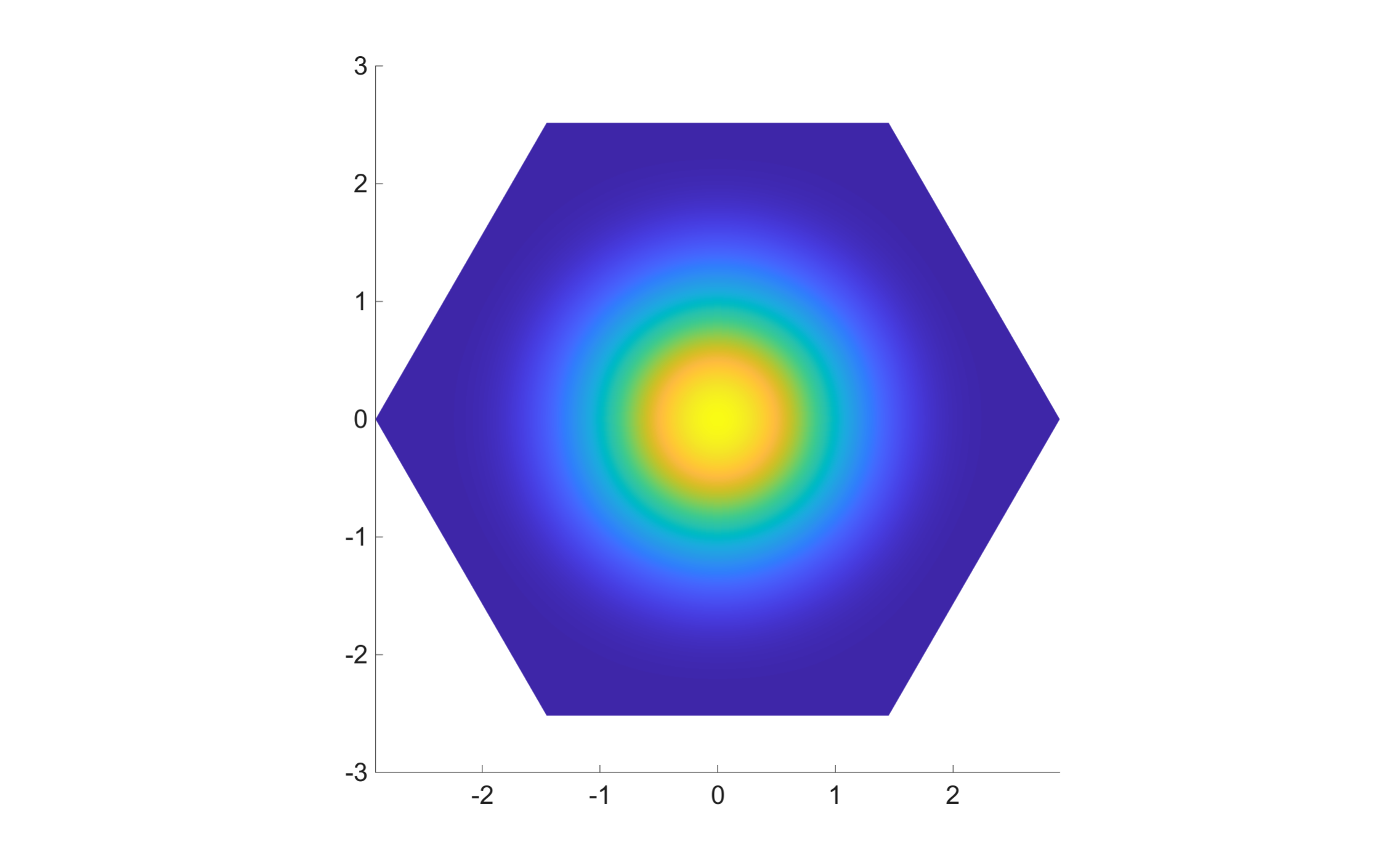} \\
	\end{tabular}
	\caption{Basic limit function for equilateral grid and $L=1.6$: $W(\cdot)=1$ (left) and $W(\cdot)=1-\cdot$ (right). {The top row shows the basic limit functions while the bottom row their top views}. Zero values were cropped to show the shape of their supports.}
	\label{fig:equilateral_basic_limit}
\end{figure}

\begin{figure}[!h]
	\centering
	\begin{tabular}{cc}
		\includegraphics[width=0.4\linewidth,clip]{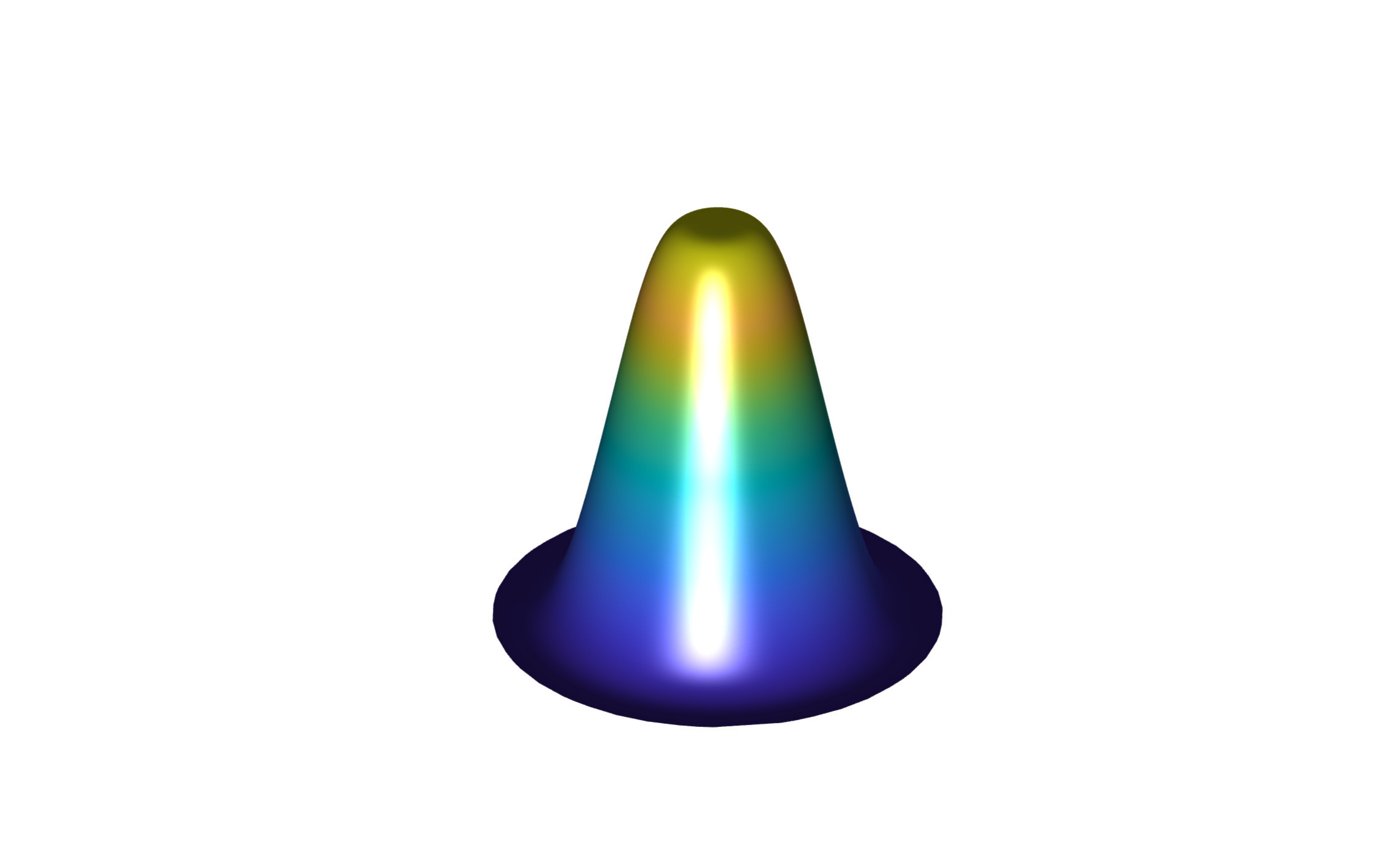} &
		\includegraphics[width=0.4\linewidth,clip]{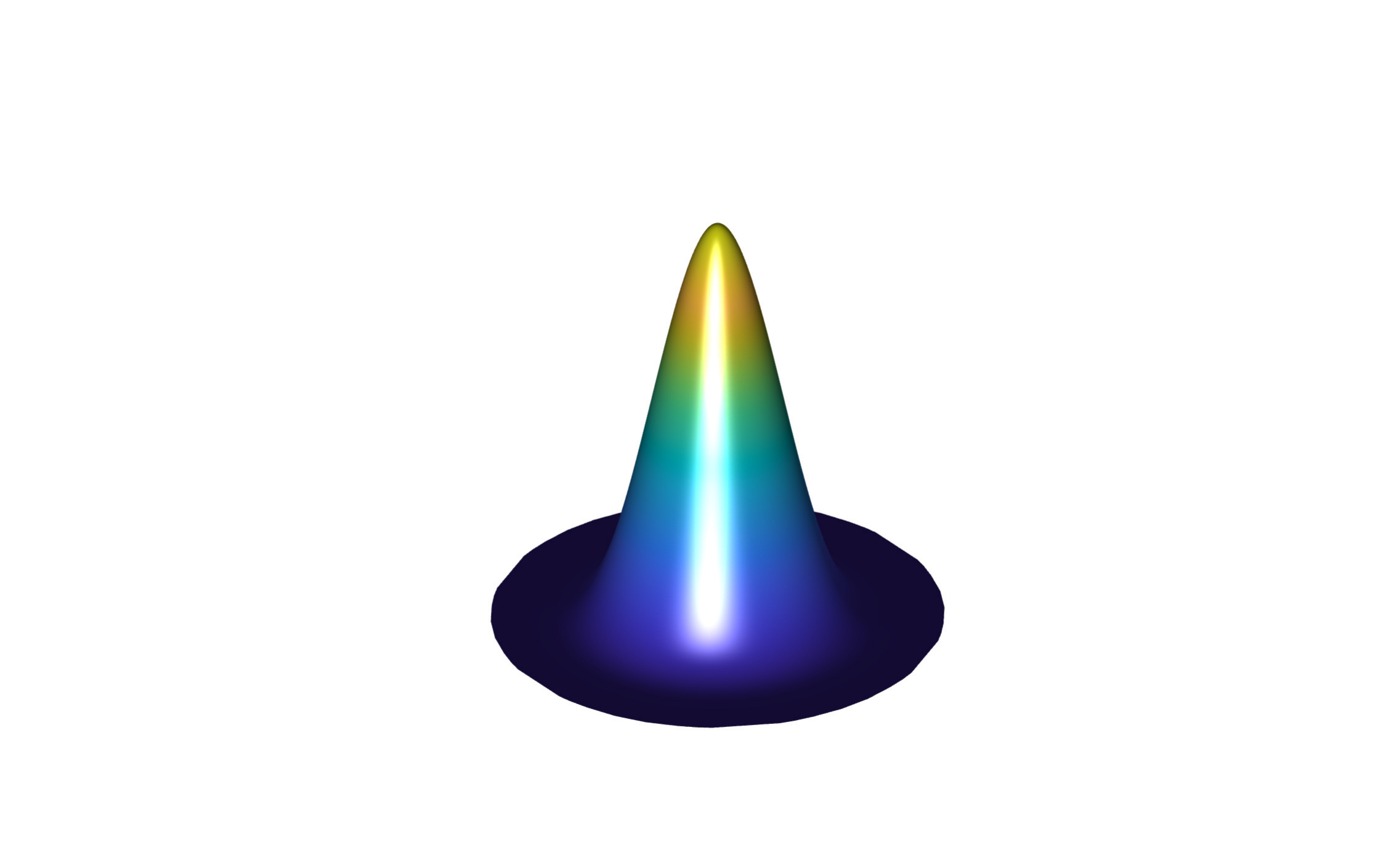} \\
		\includegraphics[width=0.4\linewidth,clip]{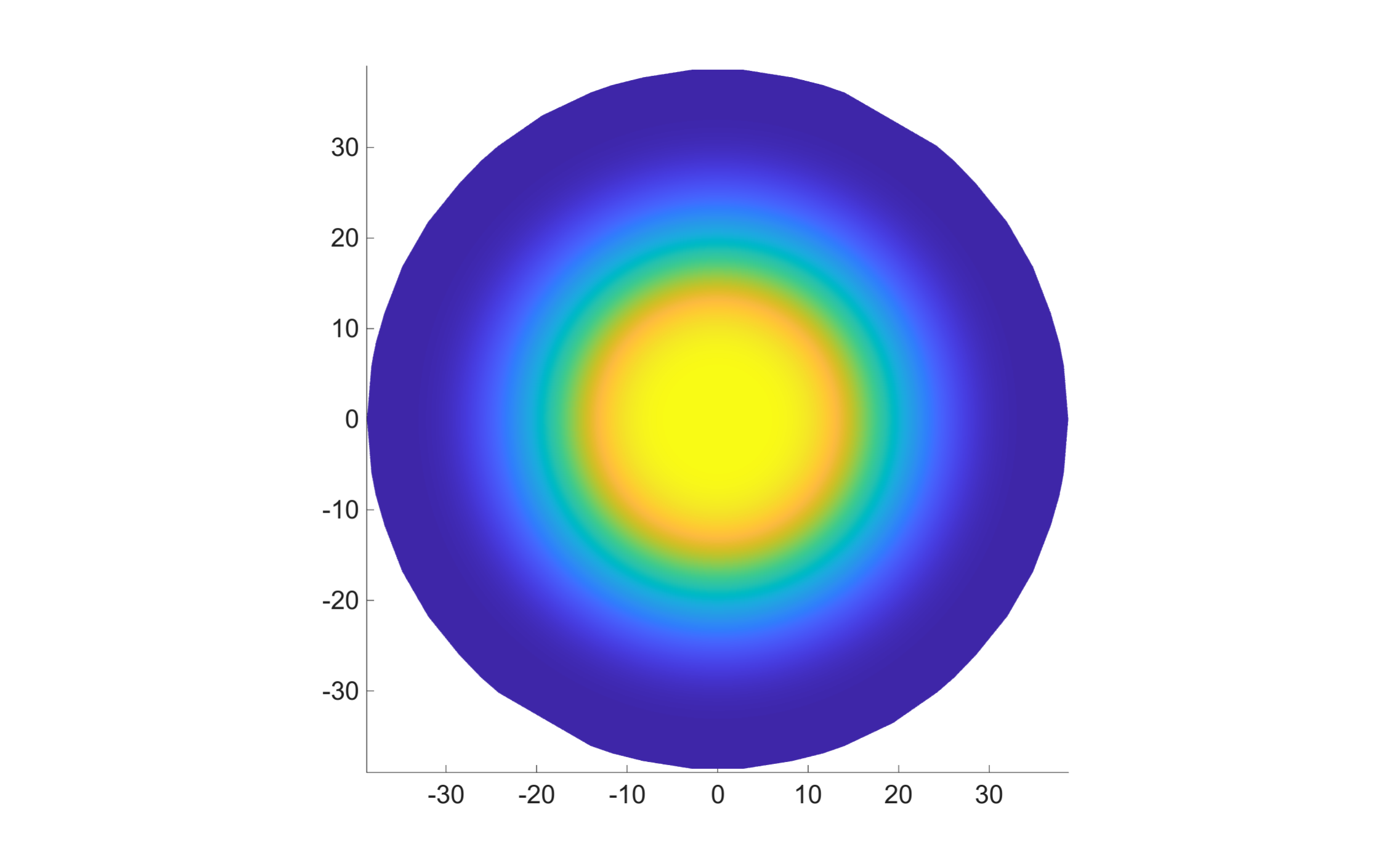} &
		\includegraphics[width=0.4\linewidth,clip]{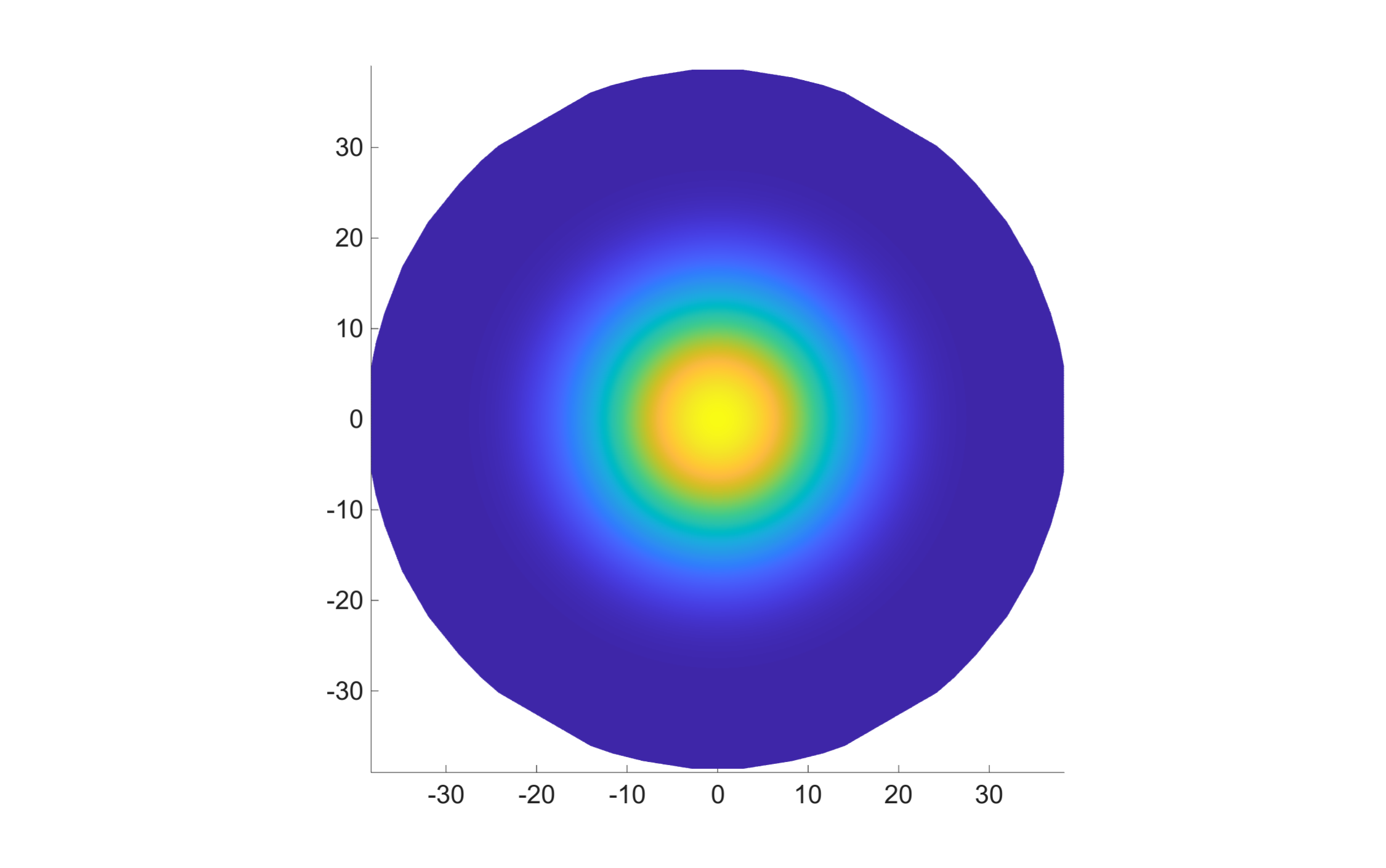} \\
	\end{tabular}
	{\caption{Basic limit function for equilateral grid and $L=20$: $W(\cdot)=1$ (left) and $W(\cdot)=1-\cdot$ (right).The top row shows the basic limit functions while the bottom row their top views. Zero values were cropped to show the shape of their supports.}}
	\label{fig:equilateral_basic_limit_L20}
\end{figure}


\subsection{Triangular-rectangular grid}

We set $\frac32 < L < \frac{\sqrt{13}}2$, leading to the stencils shown in Figure \ref{fig:rectangular_rings}. For $W(\cdot)=1$ the bivariate mask is
$$\ba=\frac1{72}\left(\begin{array}{ccccccc} 0 & 0 & 6 & 9 & 6 & 0 & 0\\ 0 & 8 & 9 & 8 & 9 & 8 & 0\\ 6 & 9 & 6 & 9 & 6 & 9 & 6\\ 9 & 8 & 9 & \circled{8} & 9 & 8 & 9\\ 6 & 9 & 6 & 9 & 6 & 9 & 6\\ 0 & 8 & 9 & 8 & 9 & 8 & 0\\ 0 & 0 & 6 & 9 & 6 & 0 & 0 \end{array}\right).$$
while for $W(\cdot)=1-\cdot$, it is
\begin{equation}
	\ba=\left(\begin{array}{ccccccc} 0 & 0 & a & b & a & 0 & 0\\ 0 & c & d & e & d & c & 0\\ a & d & f & g & f & d & a\\ b & e & g & \circled{h} & g & e & b\\ a & d & f & g & f & d & a\\ 0 & c & d & e & d & c & 0\\ 0 & 0 & a & b & a & 0 & 0 \end{array}\right),
\end{equation}
where $a = -\frac{18\,\left(2\,L-\sqrt{10}\right)}{\sqrt{2}-6\,L+2\,\sqrt{10}}$, $b = -\frac{18\,\left(2\,L-3\right)}{\sqrt{5}-4\,L+2}$, $c=-\frac{72\,\left(L-\sqrt{2}\right)}{4\,\sqrt{2}-9\,L+4}$, $d=-\frac{18\,\left(2\,L-\sqrt{5}\right)}{\sqrt{5}-4\,L+2}$, $e=-\frac{72\,\left(L-1\right)}{4\,\sqrt{2}-9\,L+4}$, $f=-\frac{18\,\left(2\,L-\sqrt{2}\right)}{\sqrt{2}-6\,L+2\,\sqrt{10}}$, $g = -\frac{18\,\left(2\,L-1\right)}{\sqrt{5}-4\,L+2}$, $h=-\frac{72\,L}{4\,\sqrt{2}-9\,L+4}$. It can be checked that all the coefficients are positive since $\frac32 < L < \frac{\sqrt{13}}2$, as more generally proved in Theorem \ref{tmh:uniform}.

\begin{figure}[!h]
	\centering
	\begin{tabular}{cccc}
		\resizebox{0.23\textwidth}{!}{
		\input{images/rectangular_ring_r.txt}
		}
		&
		\resizebox{0.23\textwidth}{!}{
		\input{images/rectangular_ring_h.txt}
		}
		&
		\resizebox{0.23\textwidth}{!}{
		\input{images/rectangular_ring_v.txt}
		}
		&
		\resizebox{0.23\textwidth}{!}{
		\input{images/rectangular_ring_d.txt}
		}
	\end{tabular}
	\caption{
	Stencil selection (red dots) on a triangular-rectangular grid with $\frac32 < L < \frac{\sqrt{13}}2$.
	From left to right: replacement, horizontal insertion, vertical insertion and diagonal insertion.
	The green circle corresponds to $L=\frac32$, the blue circle to $L=\frac{\sqrt{13}}2$.}
	\label{fig:rectangular_rings}
\end{figure}

\begin{figure}[!h]
	\centering
	\begin{tabular}{cc}
		\includegraphics[width=0.4\linewidth,clip]{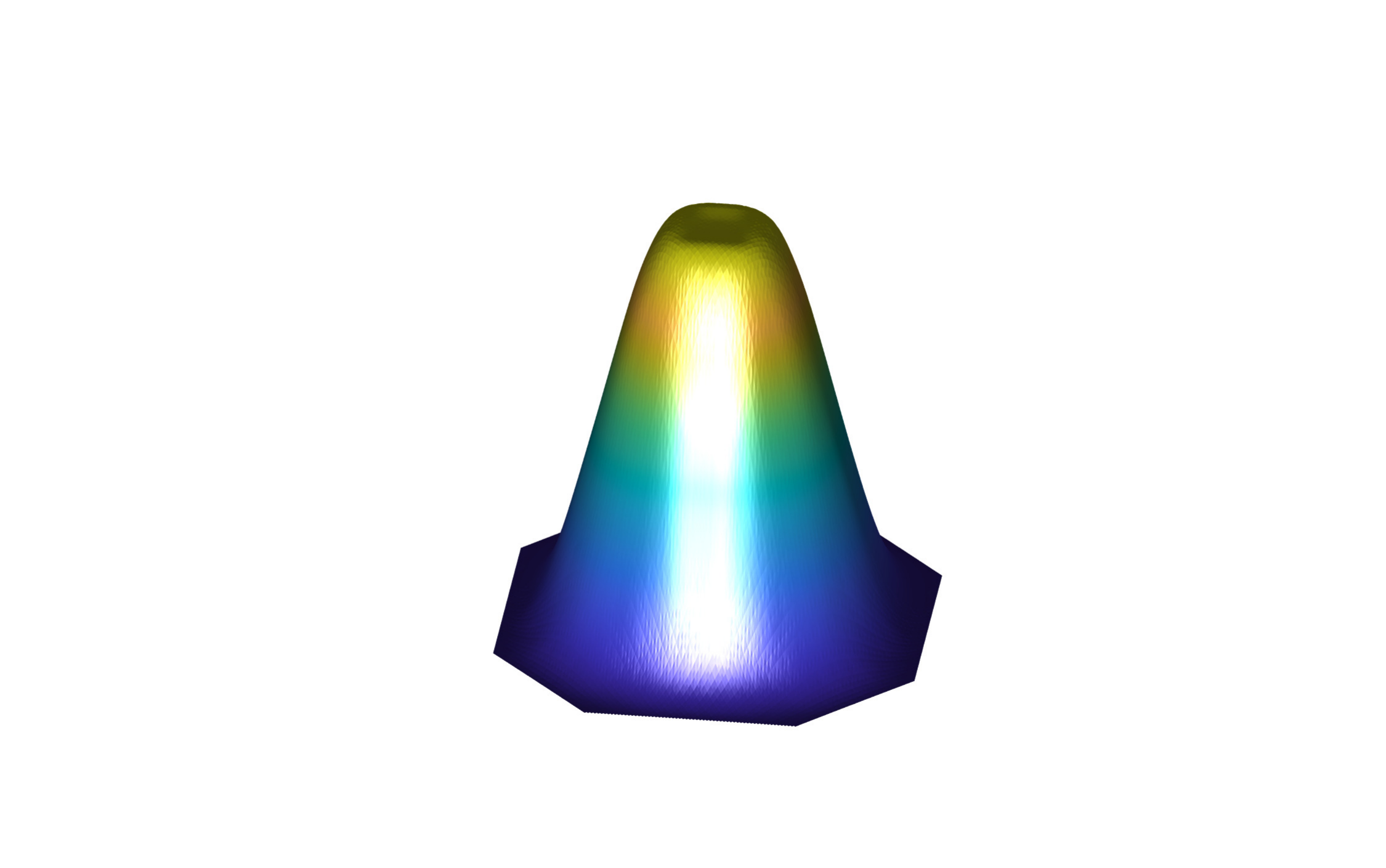} &
		\includegraphics[width=0.4\linewidth,clip]{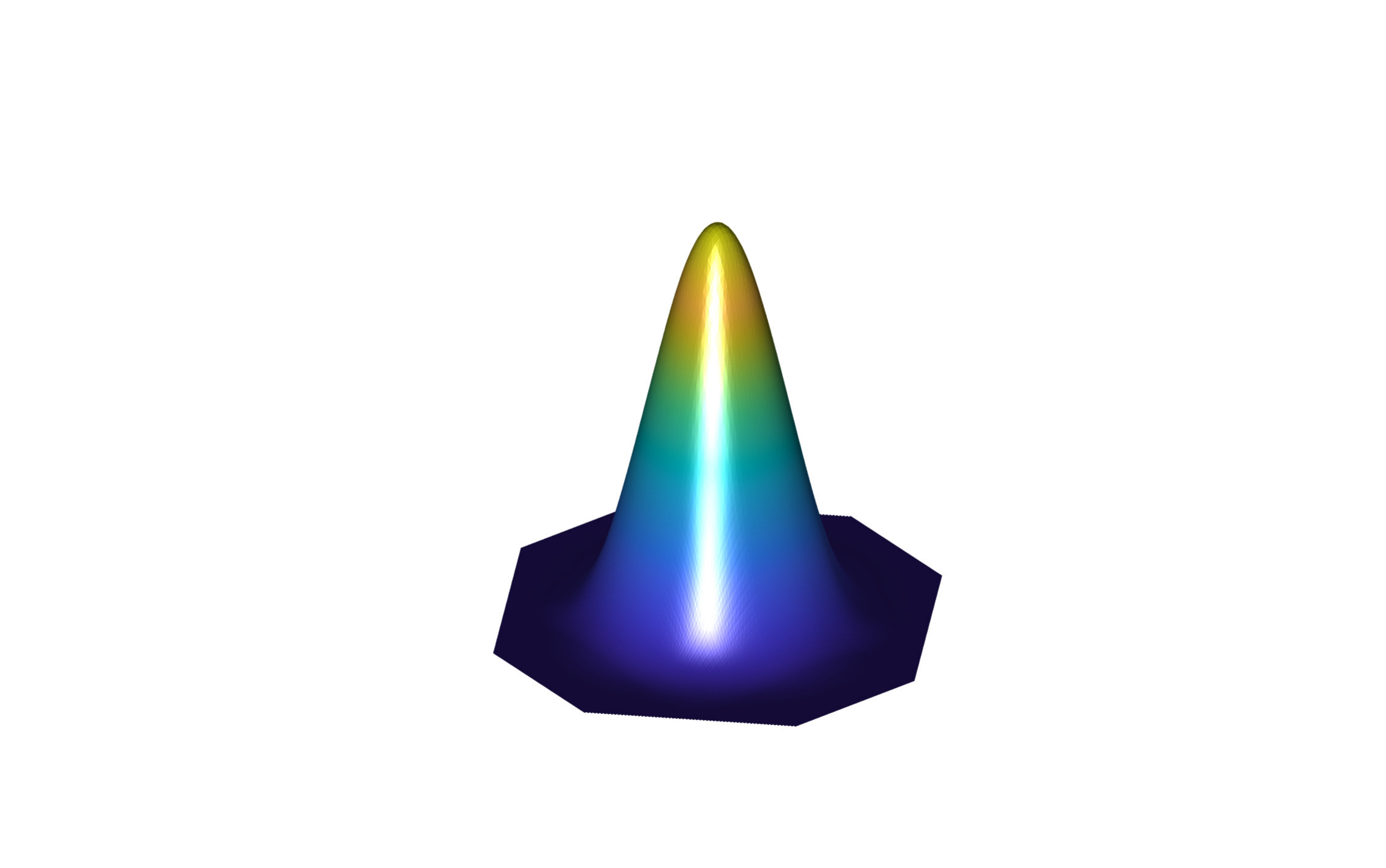} \\
		\includegraphics[width=0.4\linewidth,clip]{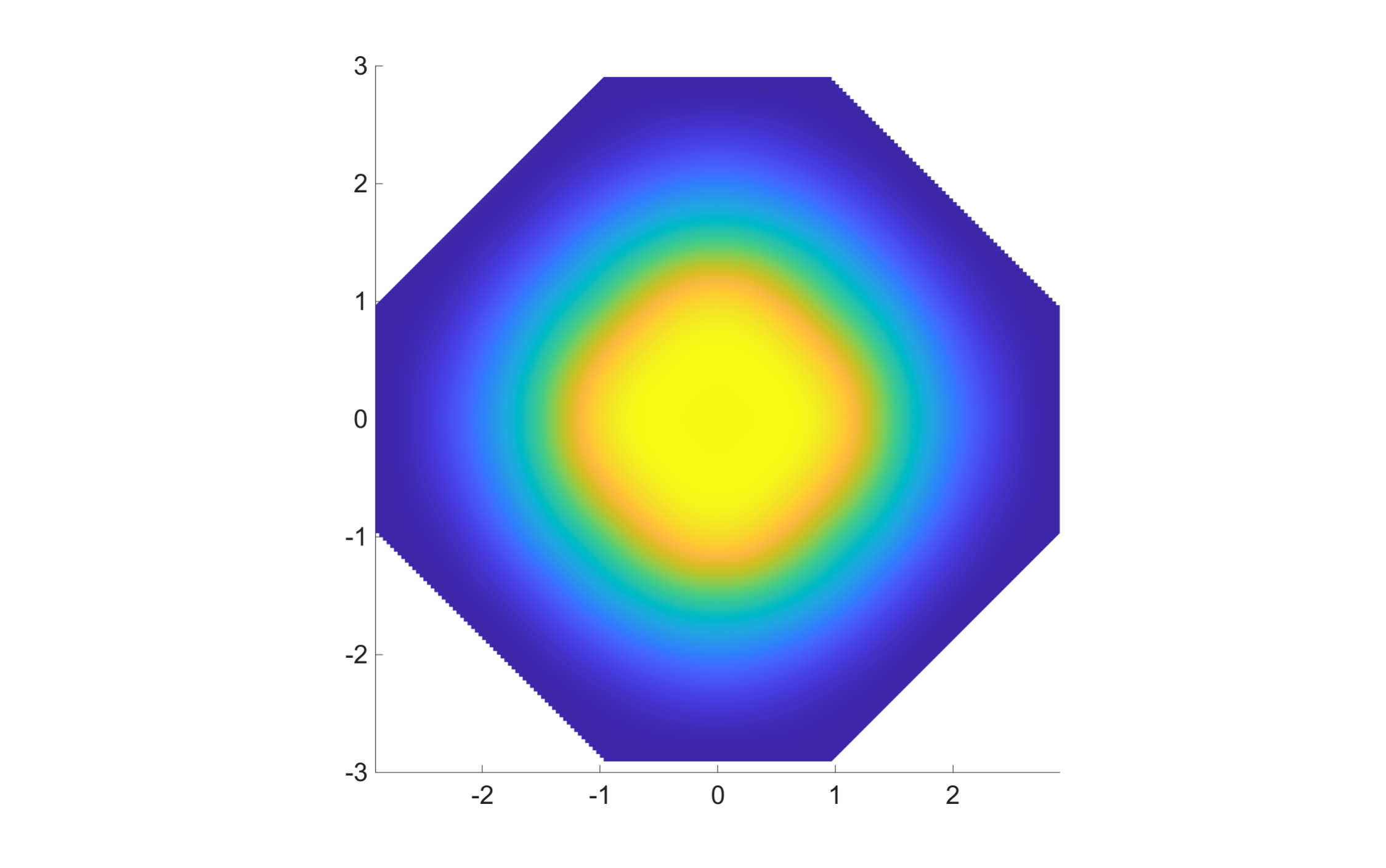} &
		\includegraphics[width=0.4\linewidth,clip]{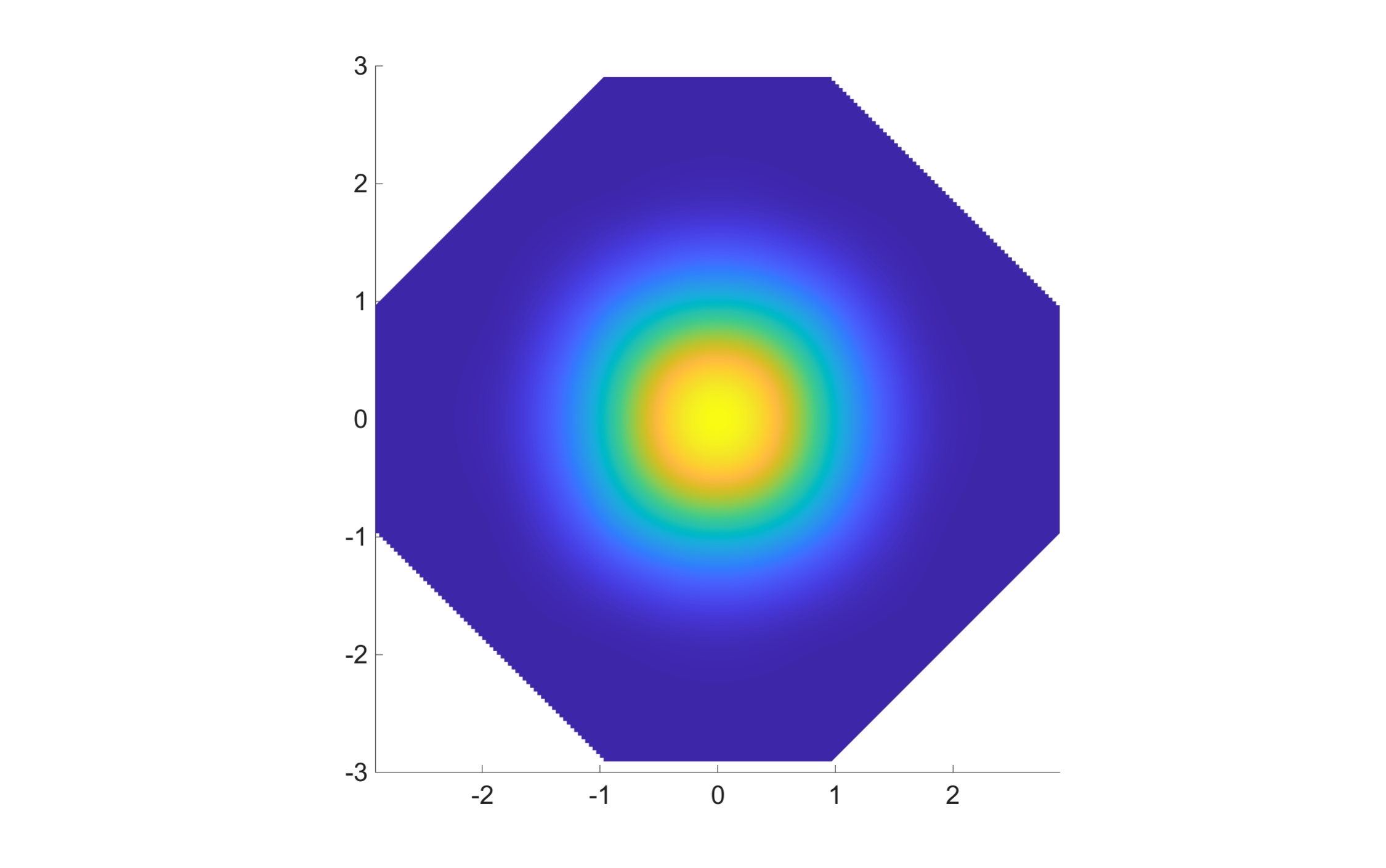} \\
	\end{tabular}
	\caption{Basic limit function for triangular-rectangular grid and $L=1.6$: $W(\cdot)=1$ (left) and $W(\cdot)=1-\cdot$ (right). {The top row shows the basic limit functions while the bottom row their top views}.  Zero values were cropped to enhance shape of their supports.}
	\label{fig:rectangular_basic_limit}
\end{figure}

\begin{figure}[!h]
	\centering
	\begin{tabular}{cc}
		\includegraphics[width=0.4\linewidth,clip]{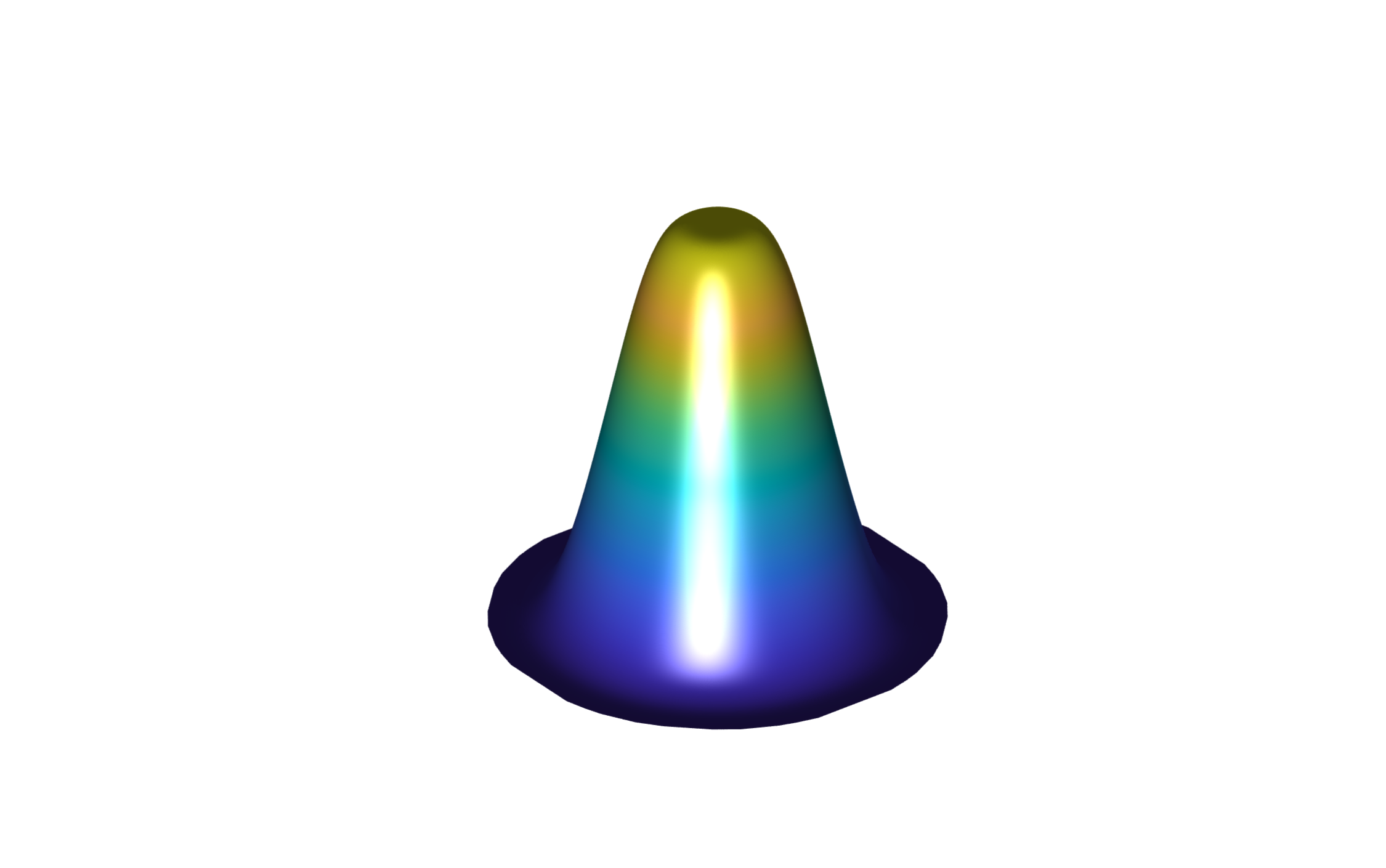} &
		\includegraphics[width=0.4\linewidth,clip]{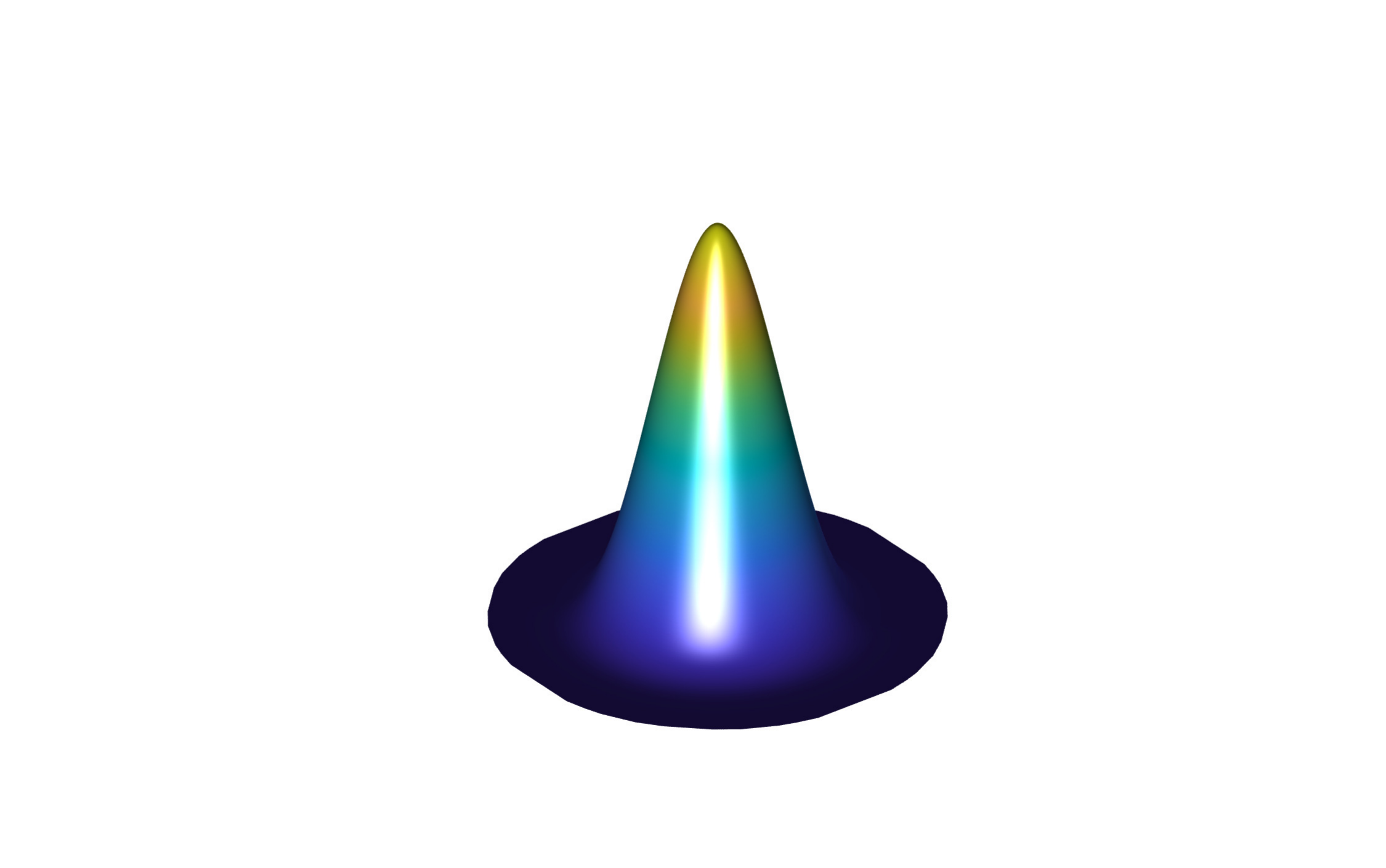} \\
		\includegraphics[width=0.4\linewidth,clip]{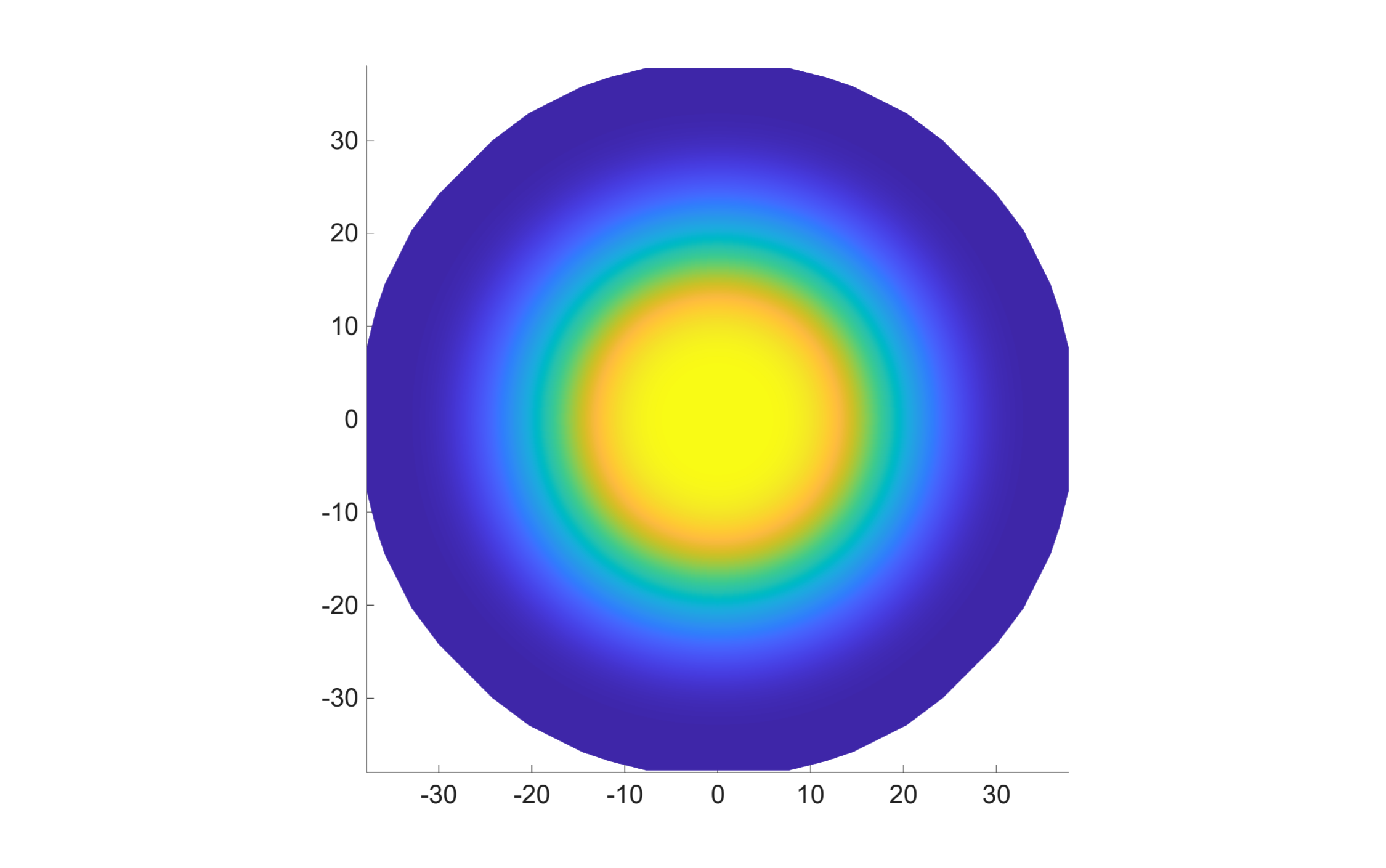} &
		\includegraphics[width=0.4\linewidth,clip]{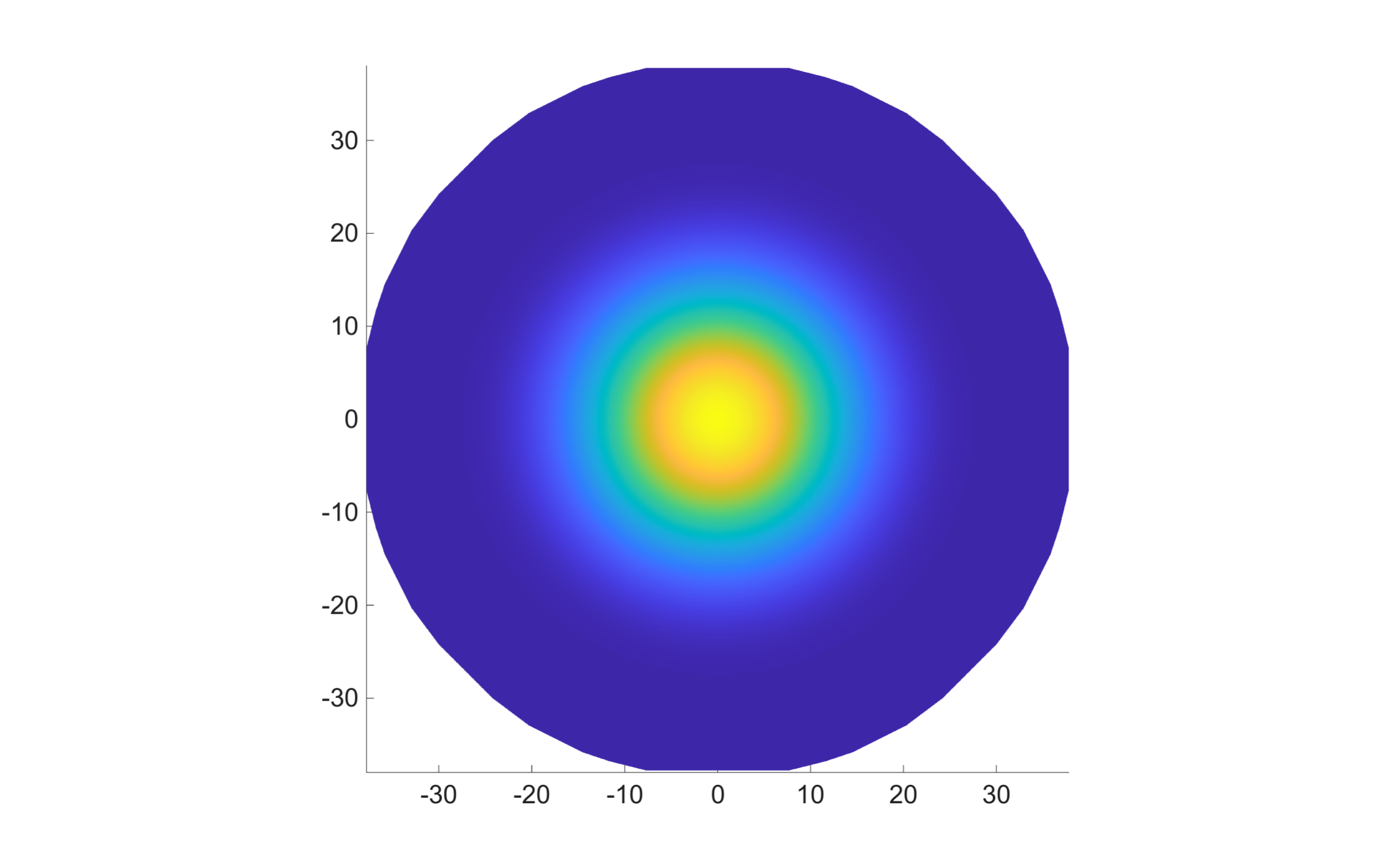} \\
	\end{tabular}
	{\caption{Basic limit function for triangular-rectangular grid and $L=20$: $W(\cdot)=1$ (left) and $W(\cdot)=1-\cdot$ (right). {The top row shows the basic limit functions while the bottom row their top views}. Zero values were cropped to show the shape of their supports.}}
	\label{fig:rectangular_basic_limit_L20}
\end{figure}
{ Figure \ref{fig:rectangular_basic_limit} and Figure \ref{fig:rectangular_basic_limit_L20} show the basic limit functions for different values of $L$ and $W$ in case of a triangular-rectangular grid.
The results are very similar to the equilateral case.}


\subsection{Comparison of the basic limit functions}

In Figures \ref{fig:equilateral_basic_limit} to \ref{fig:rectangular_basic_limit} we show the basic limit functions for two uniform grids (equilateral and triangular-rectangular) and two choices of weight function, $W(\cdot) = 1$ and $W(\cdot)=1-\cdot$, respectively. They are obtained by applying the scheme to the initial ``delta" data set with all zeros except for the vertex at the origin, which is set to one. Five iterations were performed to obtain the plots.


{An interesting feature that we observe comparing Figures \ref{fig:equilateral_basic_limit} and Figure \ref{fig:rectangular_basic_limit}, is that the shape of the support is hexagonal regardless the different grids, and the different masks supports.}

 On the contrary,
even if the mask support is the same, $W$ determines the distribution of the values in the support, so that the shape of the limit function is different. The choice $W(\cdot)=1$ leads to a more uniform distribution of values in the support with a slight depression in the center, while $W(\cdot)=1-\cdot$ leads to a more concentrated distribution of values around the origin with a bell-shaped function around the origin.
From the support of the masks, it is clear that the proposed schemes are not separable (i.e., cannot be obtained by the tensor product of univariate schemes), since separable schemes have rectangular masks supports, and thus rectangular support of the basic limit function.

\section{Numerical Results}\label{sec:num}

This section is devoted to present some numerical results associated with the application of the novel subdivision schemes proposed in this paper. {In Subsection \ref{numerical} some tests are provided to numerically prove the continuity of the new scheme  also when dealing with non-regular grids.}
In Subsection \ref{functional} a comparative analysis is performed with several well-established methods traditionally used for dealing with noisy data, which are briefly reviewed in Subsection \ref{metodi}. Furthermore, Subsection \ref{sec:surface_subdivision} provides numerical results concerning the implementation of the proposed subdivision scheme on noisy \emph{geometric data}. Note that, at the end of the paper, we provide a specific section for readers that are interested to reproduce our experiments.

{\subsection{Numerical continuity at extraordinary points and edges in case of non-regular grids}\label{numerical}

This subsection provides a numerical validation of the continuity of the proposed subdivision schemes on non-regular grids. Figure \ref{fig:random} illustrates the limit function obtained from an irregular, non-uniform triangulation containing several extraordinary vertices. The computed limit surface appears notably smooth, suggesting that a rigorous proof of continuity is plausible.
The technical challenges associated with a rigorous analysis of convergence and smoothness at extraordinary points and across edges are substantial and therefore beyond the scope of this work.

\begin{figure}
	\begin{tabular}{cc}
	\includegraphics[width=0.43\textwidth,clip,trim={250 20 230 50}]{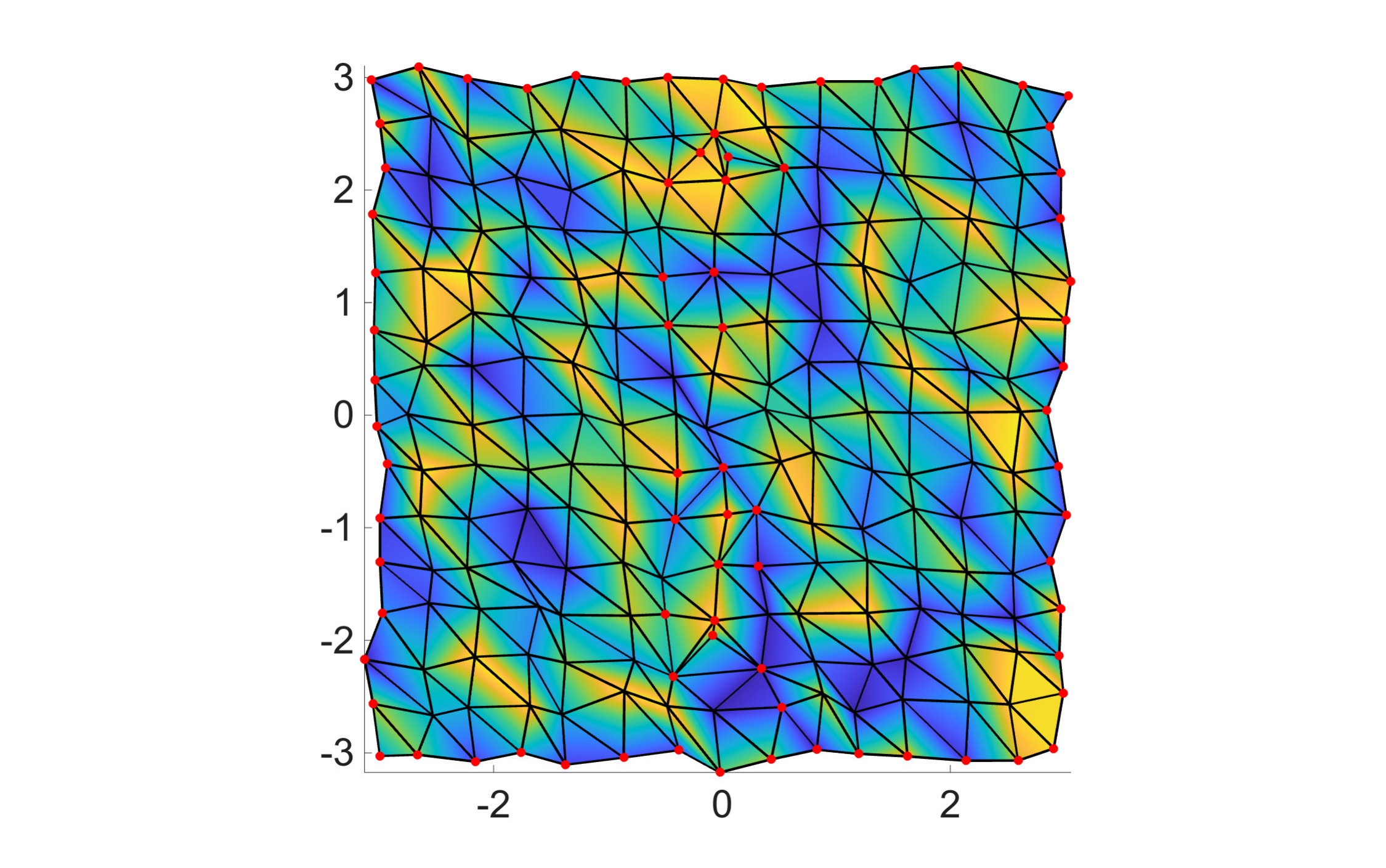}
	&
	\includegraphics[width=0.43\textwidth,clip,trim={290 70 270 50}]{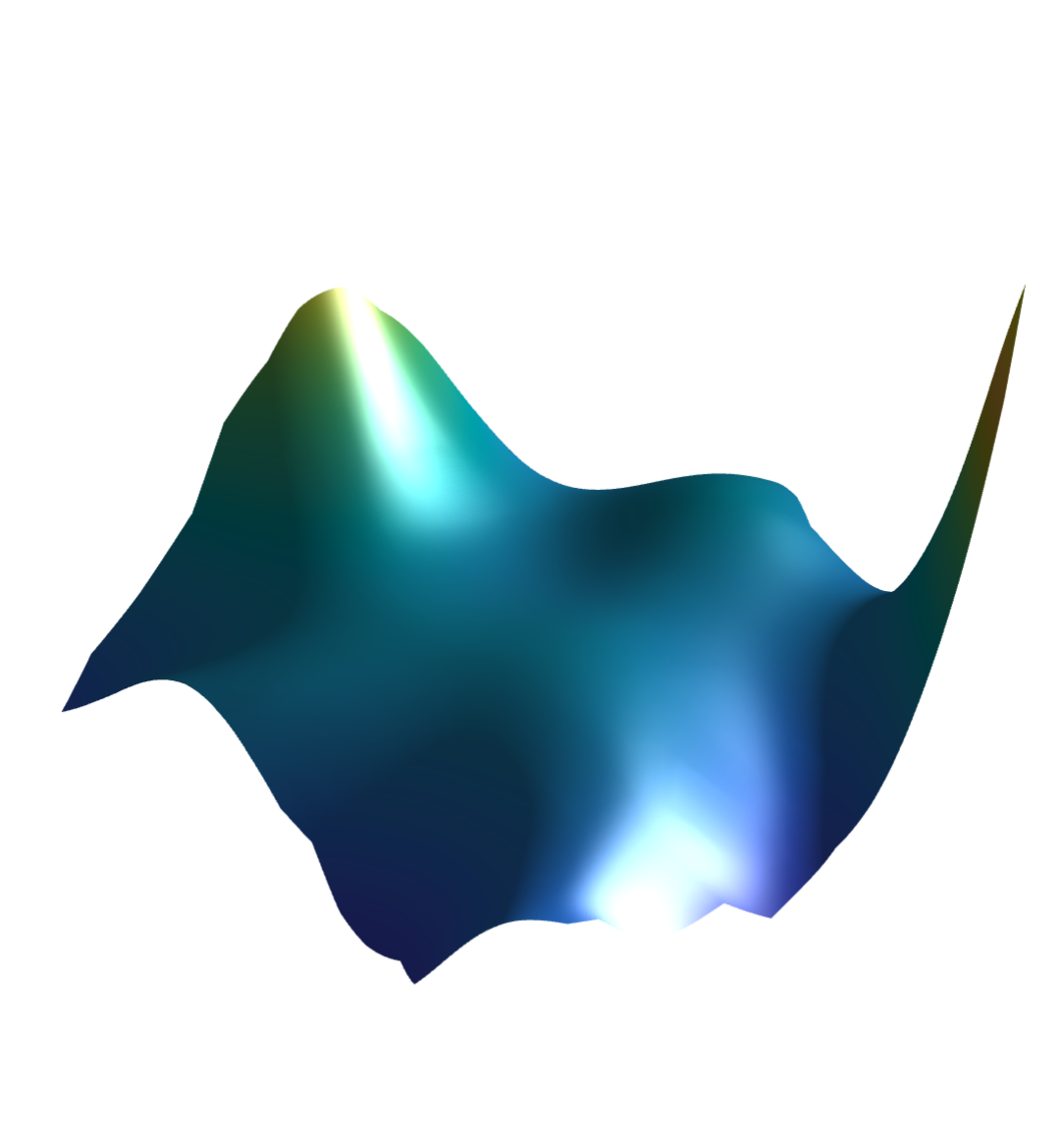}
	\end{tabular}
	{\caption{\label{fig:random} On the irregular non-uniform triangulation with random data associated (color) on the left figure, we apply  5 iterations of the subdivision scheme with $L=2$, $W(\cdot)=1-\cdot$. The extraordinary vertices are marked in red.}}
\end{figure}
}

\subsection{Four classic bivariate local linear regression methods}\label{metodi}

In this subsection, we recall four distinct techniques that have been extensively used for function approximation. These methods are reviewed to facilitate comparison with our newly proposed subdivision scheme in the literature. Specifically, we begin by discussing the Moving Least Squares (MLS) method, which utilizes polynomials of degree up to one, as well as constants (Shepard method). Subsequently, we explore the least squares Radial Basis Functions (RBF) method, and the least squares Tensor Product B-Splines (TPBS) method. Detailed descriptions of these methods can be found in \cite{Deboor} and in \cite{FASSHAUER}.

Throughout this section, we use the same notation introduced in Section 2.1: $\{\mathbf{v}_1, \dots, \mathbf{v}_n\} \subset \Omega$ represent a set of $n$ distinct data points, randomly distributed in $\Omega \subset \mathbb{R}^2$, with associated values
\begin{equation}
	z_i = f(\mathbf{v}_i) + \epsilon_i, \quad i = 1, \dots, n,
\end{equation}
where $f$ is an unknown function and $\epsilon_i,\ i = 1, \dots, n,$ are random values assumed to be independent and identically distributed with mean 0.

We start by presenting the version of the MLS (see e.g. \cite{LEVIN} and Chapter 22 in \cite{FASSHAUER}) implemented in Matlab as the routine {\tt locfit} (see \cite{LOADER}). This method is also known as \emph{local linear regression} method (LLR) in statistics context (see \cite{HASTIE,LOADER}). The bivariate MLS approximation of order one at a target point $\hat \bv \in \Omega$ is the value $p_{\hat \bv}(\hat \bv)$ where $p_{\hat \bv}\in \Pi^1$ is the linear polynomial obtained minimizing among all $p\in \Pi^1$ the weighted least squares problem
\begin{equation}\label{def:LLR}
	\min_{p\in \Pi^1}\sum_{i=1}^{n} W\left(\frac{\|\hat \bv-\mathbf{v}_i\|}{L}\right) (p(\bv_i)-z_i)^2,
\end{equation}
where $W:\R\to[0,1]$, $W|_{(1,\infty)}=0$, is a weighted function which assigns a weight to each $\mathbf{v}_i$ based on its distance from the target point, $\hat \bv\in\Omega$, as we mentioned in Section \ref{subsec:weights}. This function typically distributes the largest weights to data close to $\hat \bv$. The parameter $L$ represents the bandwidth, controlling the size of the local neighbourhood to ensure a well-posed problem.

If we restrict $p$ to be constant, i.e. we consider \eqref{def:LLR} minimizing over $\Pi^0$, the MLS method reduces to a generalization of Shepard method \cite{SHEPARD}, resulting in the following closed approximation formula:
\begin{equation}\label{shepardgen}
	p_0(\hat \bv)=\frac{\displaystyle{\sum_{i=1}^n} z_i W(\|\hat \bv-\mathbf{v}_i\|/L)}{\sum_{j=1}^n W(\|\hat \bv-\mathbf{v}_j\|/L)}.
\end{equation}

By replacing the set of polynomials with other basis functions, such as Radial Basis Functions (RBF), we obtain the least square RBF method \cite{FASSHAUER}. To describe
it we need $\{\mathbf{w}_1, \dots, \mathbf{w}_M\}$ a set of $M$ centers, where $M < n$, and $\Phi_j:\Omega\subset\mathbb{R}^2\to \mathbb{R},\ j=1,\hdots,M$ a set of radial basis functions, defined as
$$\Phi_j(\bv)=W(\|\bv-\mathbf{w}_j\|_2/L),\quad j=1,\hdots,M,$$
where $W$ is a fixed function, such as $W(x) = e^{-(2.5x)^2/2}$ (see \cite{chen,FASSHAUER,wendland} for other examples). The approximation at $\hat{\mathbf{v}}\in \Omega$ is defined as
\begin{equation}\label{rbfapprox}
	{Q}_f(\hat{\mathbf{v}})=\sum_{j=1}^M c_j \Phi_j(\hat{\mathbf{v}}),
\end{equation}
where the coefficients $c_j$ are determined by solving the least squares problem:
$$\min_{c_1,\cdots, c_M \in\R}\sum_{i=1}^n({Q}_f(\mathbf{v}_i)-z_i)^2=\sum_{i=1}^n\left(\sum_{j=1}^M c_j \Phi_j(\mathbf{v}_i)-z_i\right)^2,$$
having a unique solution provided that the collocation matrix \( A \), defined by the entries
\[
A_{ij} = \Phi_j(\mathbf{v}_i), \quad i = 1, \ldots, n, \quad j = 1, \ldots, M,
\]
has full rank (see \cite{FASSHAUER}).

Finally, by substituting the RBFs $\{\Phi_j\}_{j=1}^M$ with Tensor-Product B-Splines (TPBS) $\{\mathcal{G}_{k,j}\}_{j=1}^M$ of order $k$ and knots $\{\mathbf{b}_j\}_{j=1}^M$ (see \cite{Deboor}), we construct the TP-spline approximation
\begin{equation}\label{splineapprox}
	{S}_f(\hat{\mathbf{v}})= \sum_{j=1}^M c_j \mathcal{G}_{k,j}(\hat{\mathbf{v}}).
\end{equation}
As before, the coefficients $c_j$ are computed by minimizing
$$\min_{c_1,\cdots, c_M\in\R}\sum_{i=1}^n\left(\sum_{j=1}^M c_j\mathcal{G}_{k,j}(\mathbf{v}_i)-z_i\right)^2.$$

\subsection{Comparison}\label{functional}

To analyze the behaviour of our new numerical method, we conduct the following numerical experiments: we consider a non-uniform irregular triangulation consisting of $N^0=227$ vertices $\bv^{0}_i \in [-3.2,3.2]^2$, shown in Figure \ref{fig:random}-left, and
$$z^{0}_i=f(\bv^{0}_i)+\epsilon_i, \quad i=1,\hdots,227,$$ for $f(x,y)=\sin(x)\cos(y)$ and each $\epsilon_{i}$ follows a normal distribution with mean 0 and standard deviation $0.2$. {This data set is represented in Figure \ref{figuraDATA}.}
We apply five iterations of our new subdivision scheme to these initial data to obtain $\bv^5$ and $\bz^5$, select the vertices included in the square $[-1,1]^2$, $\tilde{\bv}^5=\bv^5\cap [-1,1]^2$, and their associated values $\tilde{\bz}^5=\{\tilde{z}_i^5\}_{i=1}^{\tilde{N}^5}$
and subsequently compute error metrics:
\begin{equation}
	E_2:=\left(\frac{1}{\tilde{N}^5}\sum_{i=1}^{\tilde{N}^5} (\tilde{z}_i^5-f(\tilde{\bv}^{5}_i))^2\right)^{\frac12}, \quad E_\infty:=\max_{i=1,\hdots,\tilde{N}^5} |\tilde{z}_i^5-f(\tilde{\bv}^{5}_i)|.
\end{equation}

We compare our results with those derived by the methods discussed in the previous section.
The diameter of the considered triangulation is approximately $0.8078$, and $L\in\{1,2\}$, for any method except for the TPBS method, which is not dependent on $L$. As weight functions, we will consider $W(\cdot)=1-\cdot$, except for Shepard and RBF methods, since they would provide non-smooth approximants, and $W(\cdot)=e^{-(2.5\cdot)^2/2}$. The results are summarized in Table \ref{tab:errors}, and some approximants can be visualized in Figure \ref{figuraA}.

Tables \ref{tab:errors} shows that the new subdivision scheme presents error metrics similar to the rest of the methods and that the choice $L=1$ is better than $L=2$ for all methods, except for the RBF method.

We conclude by providing more details about the implementation: The Matlab {\tt locfit} is used to compute the MLS approximation; for the RBF method we choose the centers $\{\mathbf{w}_j\}_{j=1}^{49}=\{-3,-2,-1,0,1,2,3\}^2$ and for the TPBS method, the uniform nodes $\{\mathbf{b}_j\}_{j=1}^{64}=\{-3,-3,-3,-1,1,3,3,3\}^2$ are considered. In the latter case, we perform the approximation on a uniformly spaced grid of $6000^2$ points in the square $[-1,1]^2$ and randomly selected $\tilde{N}^5$ points which we used to compute $E_2,E_\infty$. All the numerical experiments, as well as other figures in this paper, can be reproduced using the code provided at the end of the paper.

\begin{table}[h]
	\centering
	\begin{tabular}{lrrrr}
		{\bf Method}           & $W(\cdot)$                & $L$ & $E_2$      & $E_\infty$   \\\hline
		MLS                    & $1-|\cdot|$               & 1   & 8.649e-02  & 2.165e-01     \\
		MLS                    & $e^{-(2.5\cdot)^2/2}$     & 1   & 8.727e-02  & 2.255e-01     \\
		Shepard                & $e^{-(2.5\cdot)^2/2}$     & 1   & 8.679e-02  & 2.231e-01     \\
		RBF                    & $e^{-(2.5\cdot)^2/2}$     & 1   & 1.309e-01  & 5.380e-01     \\\hline
		New subdivision scheme & $1-|\cdot|$               & 1   & 9.545e-02  & 2.284e-01 \\
		New subdivision scheme & $e^{-(2.5\cdot)^2/2}$     & 1   & 9.159e-02 & 2.188e-01 \\\hline
		MLS                    & $1-|\cdot|$               & 2   & 2.215e-01  & 4.386e-01     \\
		MLS                    & $e^{-(2.5\cdot)^2/2}$     & 2   & 2.318e-01  & 4.513e-01     \\
		Shepard                & $e^{-(2.5\cdot)^2/2}$     & 2   & 2.218e-01  & 4.479e-01     \\
		RBF                    & $e^{-(2.5\cdot)^2/2}$     & 2   & 8.195e-02  & 2.160e-01     \\\hline
		New subdivision scheme & $1-|\cdot|$               & 2   & 2.623e-01  & 5.142e-01 \\
		New subdivision scheme & $e^{-(2.5\cdot)^2/2}$     & 2   & 2.439e-01  & 4.841e-01 \\\hline\hline
		TPBS           &                           &     & 9.199e-02  & 1.766e-01     \\ \hline
		Initial data           &                           &     & 1.929e-01  & 6.464e-01     \\
		\hline
	\end{tabular}
	\caption{Approximation errors for different methods applied to the same noisy data set. }
	\label{tab:errors}
\end{table}

\begin{figure}[b]
	\begin{center}
			\includegraphics[width=0.28\textwidth,clip,trim={290 70 270 50}]{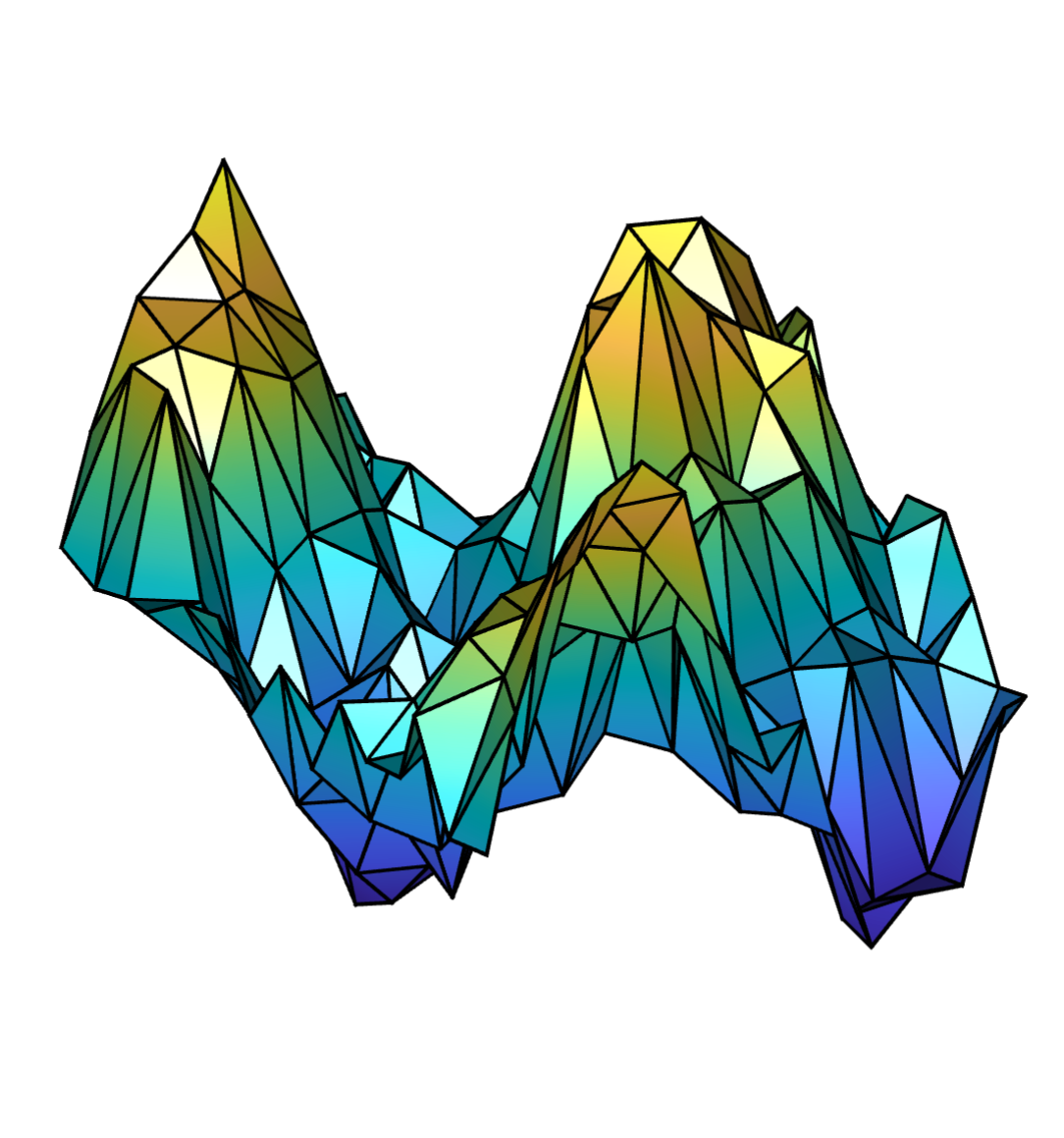}
	\end{center}
	{\caption{Data set in $[-3.2,3.2]^2$.}
	\label{figuraDATA}}
\end{figure}

\begin{figure}[b]
	\begin{center}
		\begin{tabular}{ccc}
			\includegraphics[width=0.28\textwidth,clip,trim={290 70 270 50}]{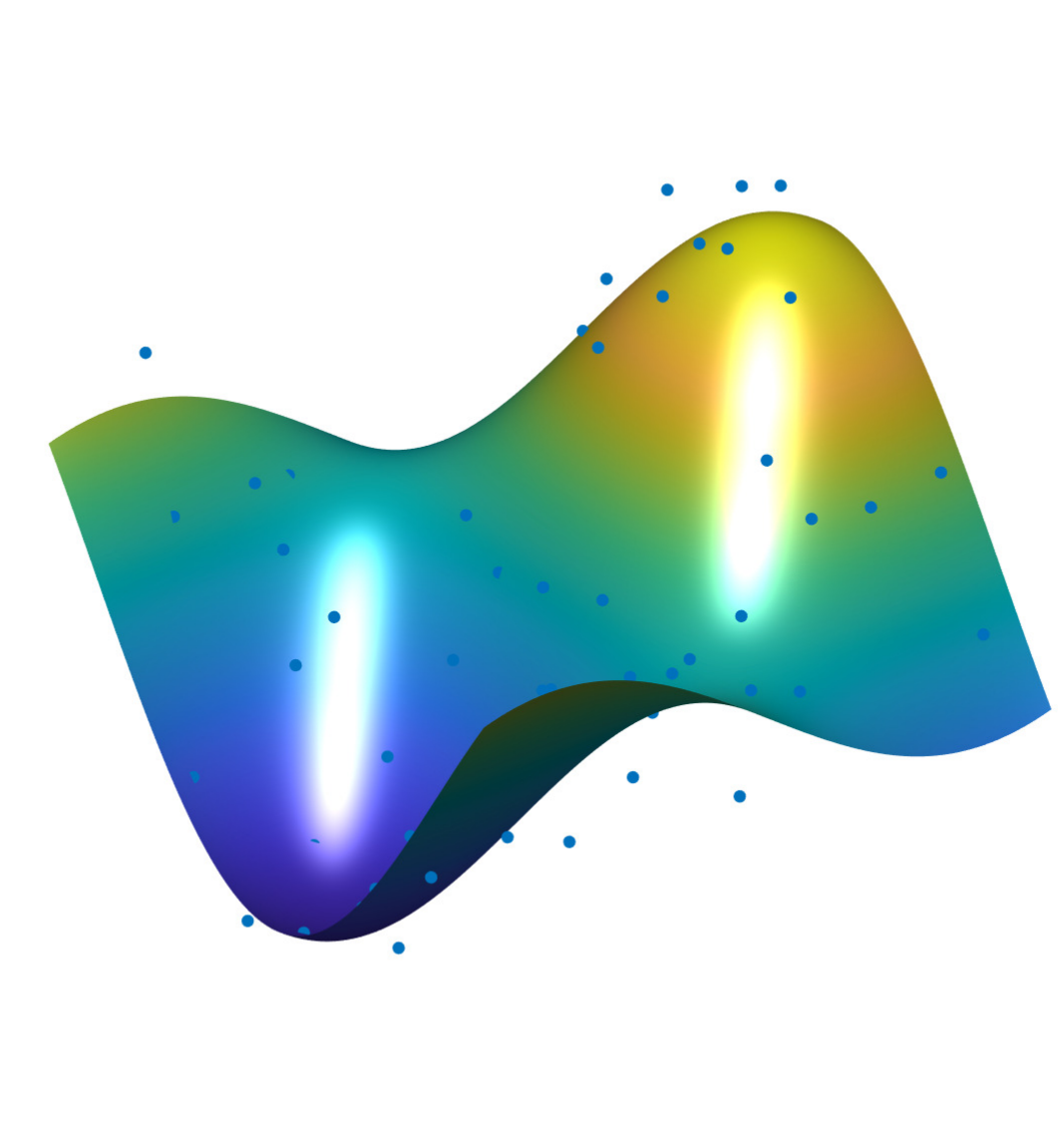}
			& \includegraphics[width=0.28\textwidth,clip,trim={290 70 270 50}]{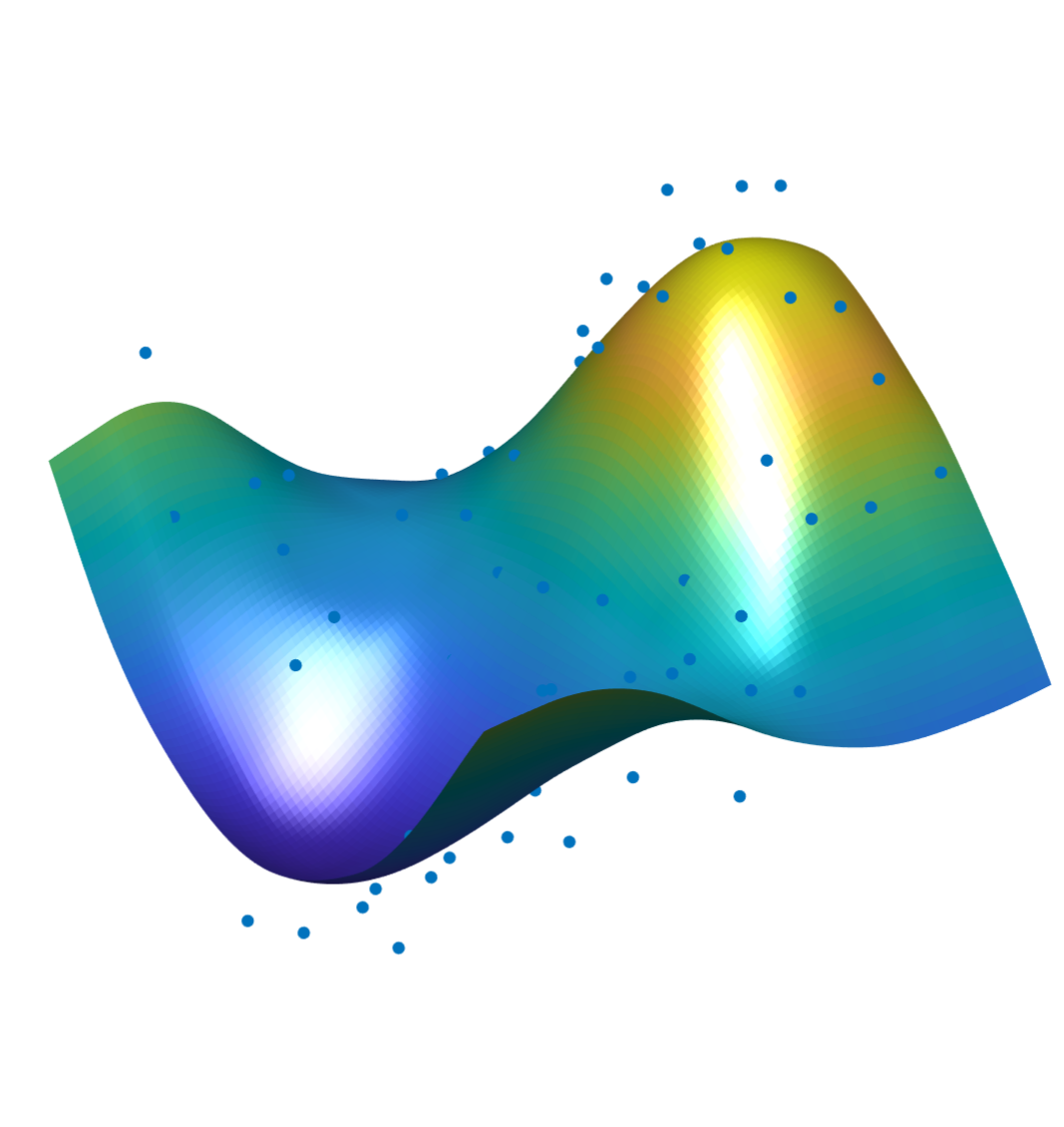}
			& \includegraphics[width=0.28\textwidth,clip,trim={290 70 270 50}]{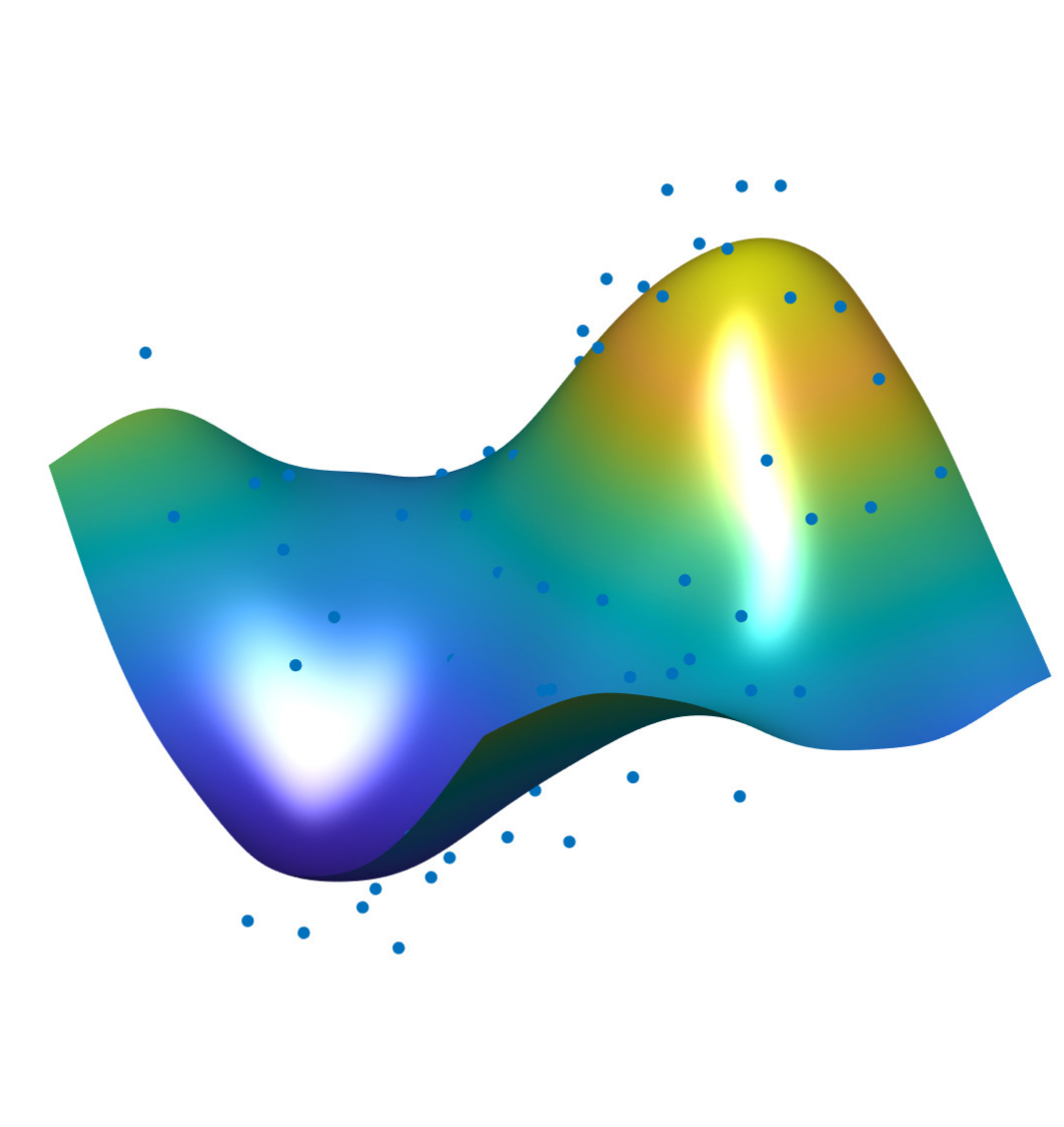} \\
			(a) \hbox{\tiny Function $f$ \& noisy data}& (b) \hbox{\tiny MLS: $W(\cdot)=1-\cdot,\ L=1$}
			& (c) \hbox{\tiny Shepard method $W(\cdot)=e^{-(2.5\cdot)^2/2}$, $L=1$} \\
			\includegraphics[width=0.28\textwidth,clip,trim={290 70 270 50}]{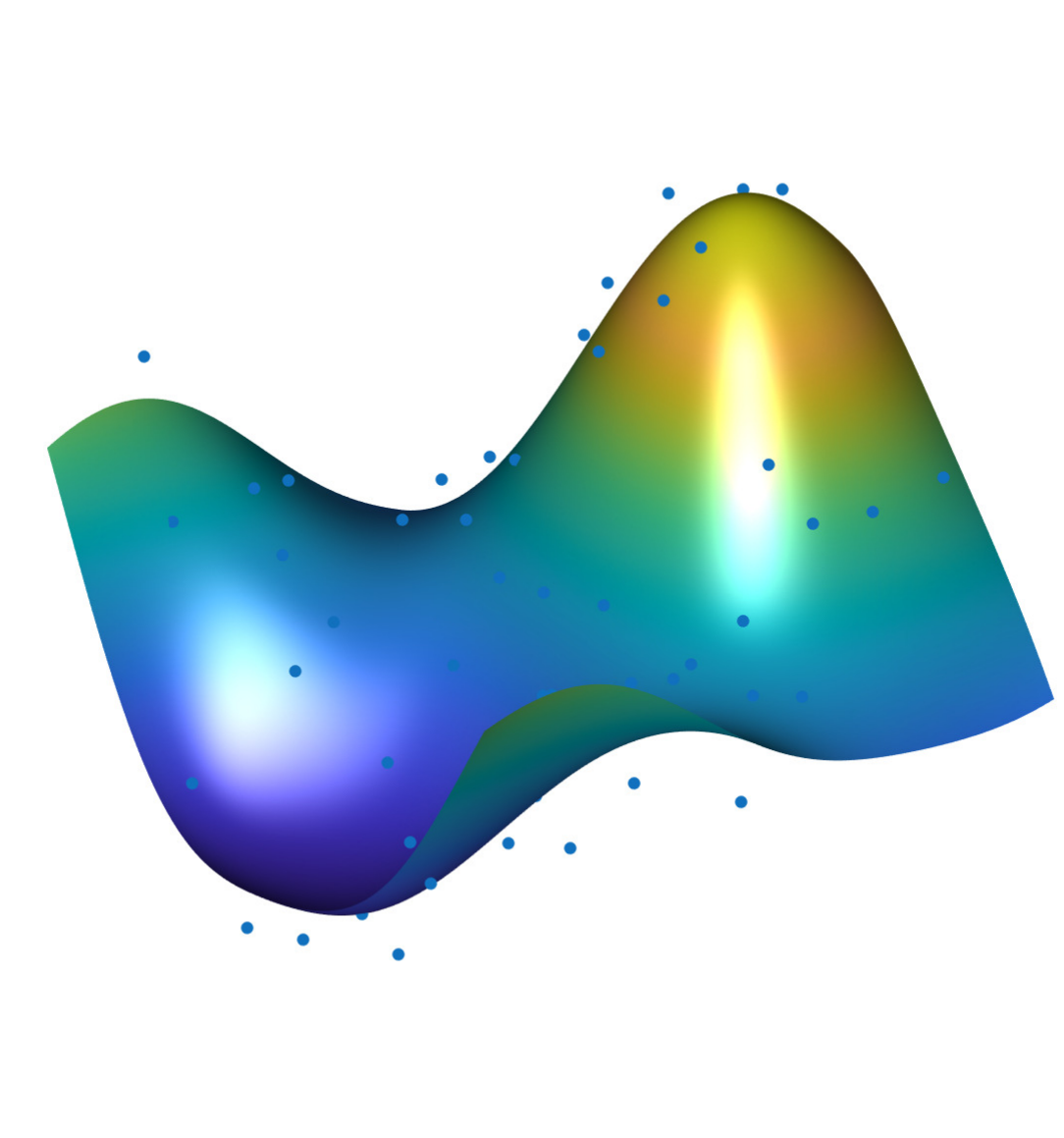}
			& \includegraphics[width=0.28\textwidth,clip,trim={290 70 270 50}]{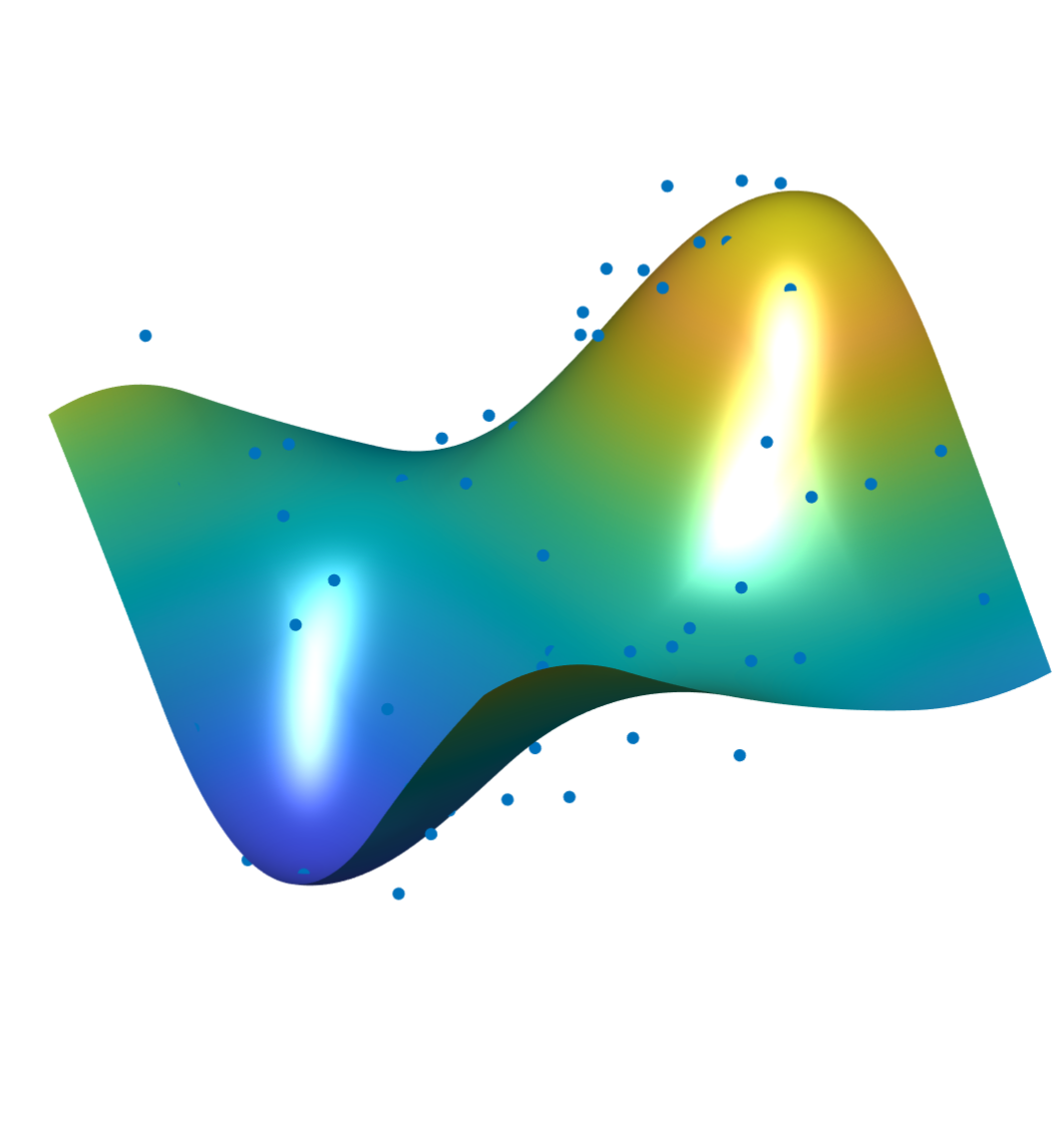}
			& \includegraphics[width=0.28\textwidth,clip,trim={400 70 360 50}]{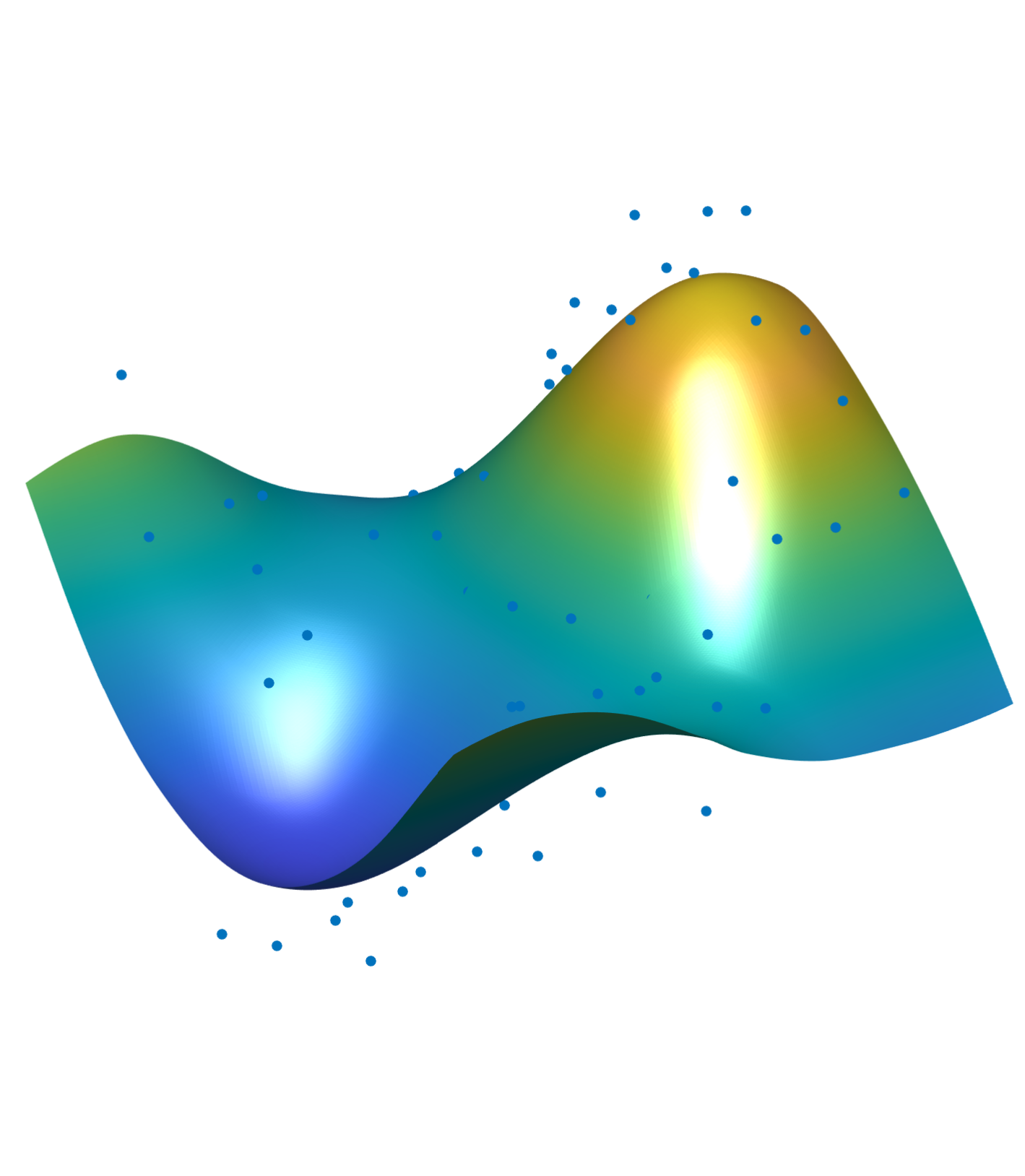}\\
			(d) \hbox{\tiny RBF with $W(\cdot)=e^{-(2.5\cdot)^2/2}$, $L=2$}& (e)\hbox{\tiny B-Splines method} & (f) \hbox{\tiny Subdivision method: $W(\cdot)=1-\cdot,\ L=1$}\\
		\end{tabular}
	\end{center}
	\caption{Comparison of several different approximations models for the same set of noisy data in $[-2,2]^2$.
	\label{figuraA}}
\end{figure}

\subsection{The method for noisy geometric data} \label{sec:surface_subdivision}

This section is to discuss a way to apply the new subdivision scheme to general 3D data. Indeed, so far, we presented a subdivision scheme designed to refine noisy functional data $z^0_i = f(\bv^0_i) + \epsilon_i \in \R$, $i=1,\ldots,n$, associated with a known triangulation of vertices $\{\bv^0_i\}_{i=1}^{n}\subset \R^2$. Consequently, the resulting limit surfaces are always graphs of functions. However, in some applications, the input data are given as $\bz_i \in\R^3$, $i=1,\ldots,n$, connected in a way to form a triangulation in $\R^3$, and  describing a general surfaces in $\R^3$, which may include a closed surface. In the latter case, only a local parametrization of the surface is guaranteed and the proposed subdivision schemes cannot be directly applied as determining a corresponding $\{\bv^0_i\}_{i=1}^{n}$ is, in general, not feasible.

As an initial attempt to develop a subdivision scheme suitable for general 3D data (always connected by a triangulation), we propose a method based on the use of local parametrizations, inspired by \cite{Levin04}. A theoretical analysis of this approach is beyond the scope of this work; instead, we provide a practical implementation to demonstrate its potential. The code for reproducing our results is included in the Reproducibility section.

To compute each new vertex $\bz^{k+1}_\ell$, we define an associated stencil
\begin{equation*}
	\cB^{k+1,\ell} := \{ j \in \{1,2,\ldots,N^k\} \ : \ \ \|\bz^{k+1}_\ell - \bz^{k}_j\| < L\},
\end{equation*}
and establish a local reference system $\mathcal{R}_{\ell} = \{O_{\ell}, \{ b^{1}_{\ell},b^{2}_{\ell},b^{3}_{\ell} \} \} $, where $O_{\ell}\in\R^{3}$ serves as the origin, and $\{b^{1}_{\ell},b^{2}_{\ell},b^{3}_{\ell}\}$ as a basis of $ \R^{3} $. Using the local coordinates of $\{\bz^{k}_j\}_{j\in \cB^{k+1,\ell}}$ relative to $\mathcal{R}_{\ell}$, the first two coordinates are interpreted as vertex locations in $\R^2$, while the third coordinate as noisy data in $\R$. The proposed linear subdivision rule is then applied, yielding the new vertex $\bz^{k+1}_\ell$ expressed in the local reference system $\mathcal{R}_{\ell}$.

Since the data describe a surface, an appropriate choice of the reference system includes: $O_{\ell}$, an estimate of the location where the new vertex will be positioned; $b^{1}_{\ell},b^{2}_{\ell}$, a basis of an approximated tangent plane; and $b^{3}_{\ell} = b^{1}_{\ell} \times b^{2}_{\ell}$, which approximates the surface normal.

A practical selection for $O_{\ell}$ could be either the value replaced by $\bz^{k+1}_\ell$ or the mid-point of the edge where it is inserted. The vectors $b^{1}_{\ell},b^{2}_{\ell}$ can be determined as the first two \emph{principal directions} of the set of vectors \( \{\bz^{k}_j -  O_{\ell}\}_{j\in \cB^{k+1,\ell}}\), which can be computed, for example, via singular value decomposition (SVD).

 { Figure \ref{fig:sphere} illustrates an example of data refinement using this method. It is obtained using
$162$ vertices on the unit sphere, generated using "spheretri" library (see Reproducibility section for more details).
A normal random noise is added on the normal direction of the sphere with standard deviation equal to $0.05$. The subdivision steps are set to $5$ while $L=0.6$ and $W(\cdot) = 1 - \cdot$.
To run the $5$ steps, without any attention to optimize the implementation of the triangulation and the identification of the $L$-ball throw the triangulation, the run time is 94 seconds on a Ubuntu 24.04.3 LTS, ASUS PRIME H410M-E, Intel® Core™ i7-10700 x 16, 32,0 GiB RAM. Optimizing the mentioned steps would certainly reduce the run time.}



\begin{figure}[!h]
	\begin{center}
		\begin{tabular}{cc}
			\includegraphics[width = 0.25 \textwidth,clip,trim = {260 110 230 50}]{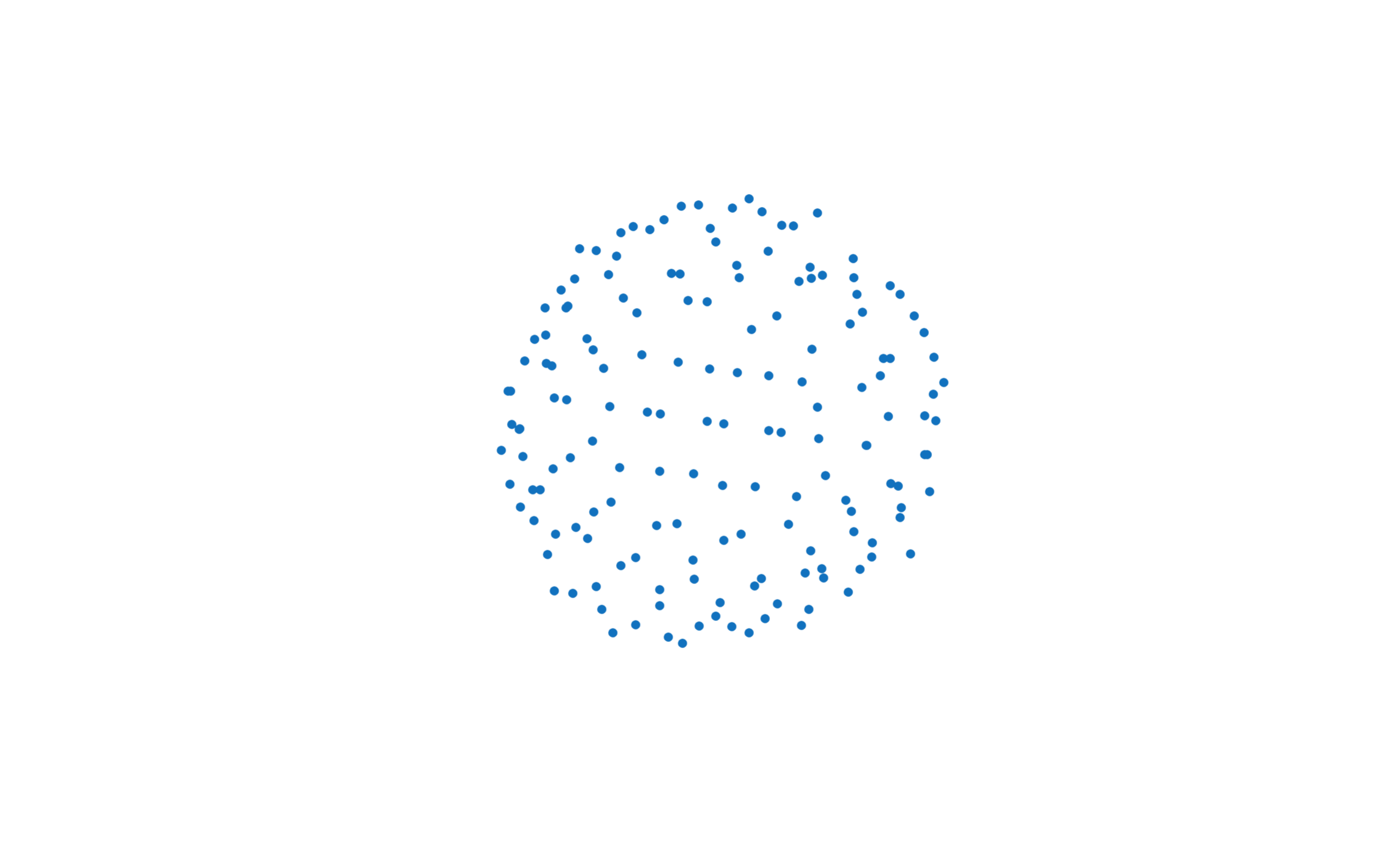} &
			\includegraphics[width = 0.25 \textwidth,clip,trim = {260 110 230 50}]{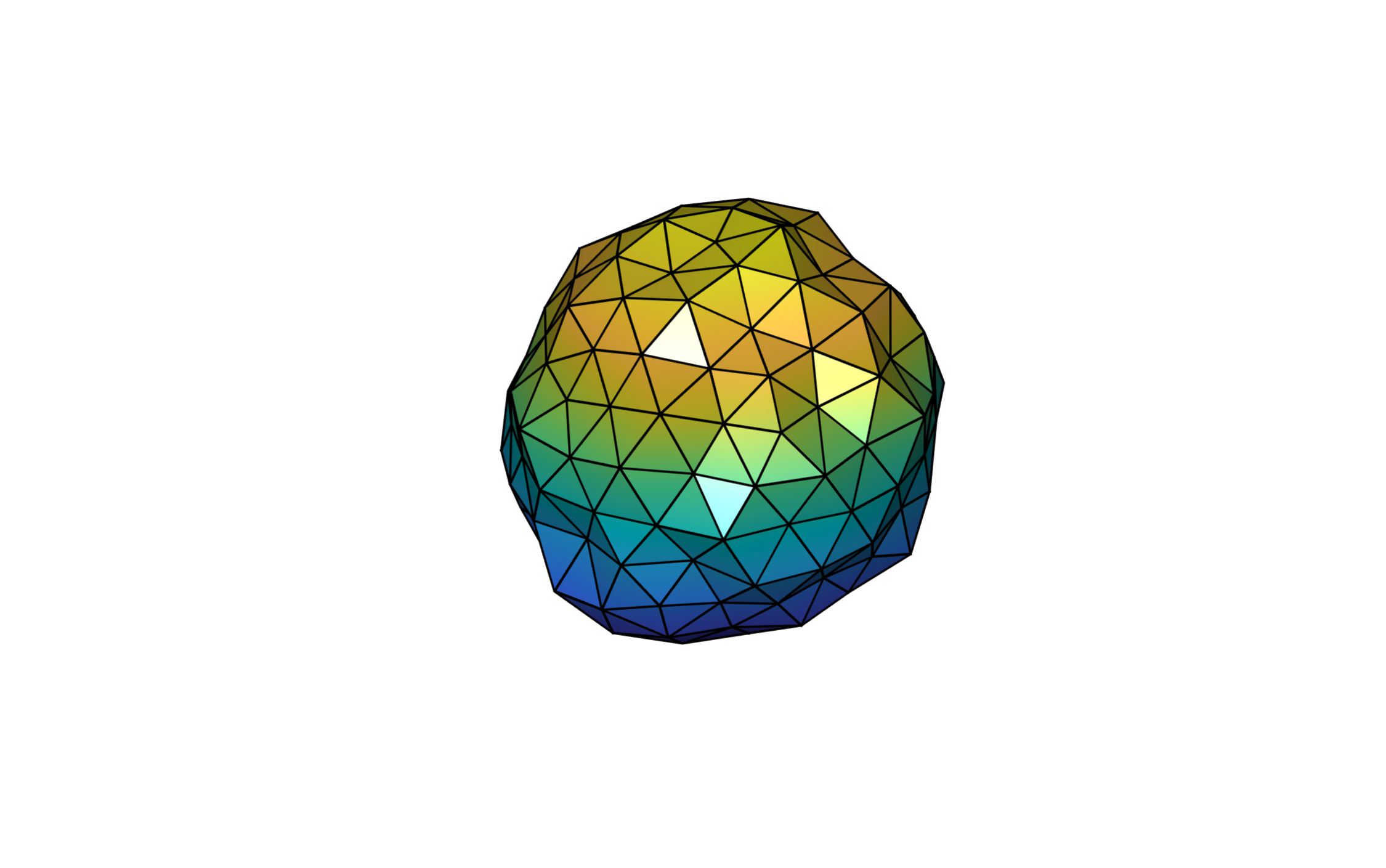}\\
			\includegraphics[width = 0.25 \textwidth,clip,trim = {260 110 230 50}]{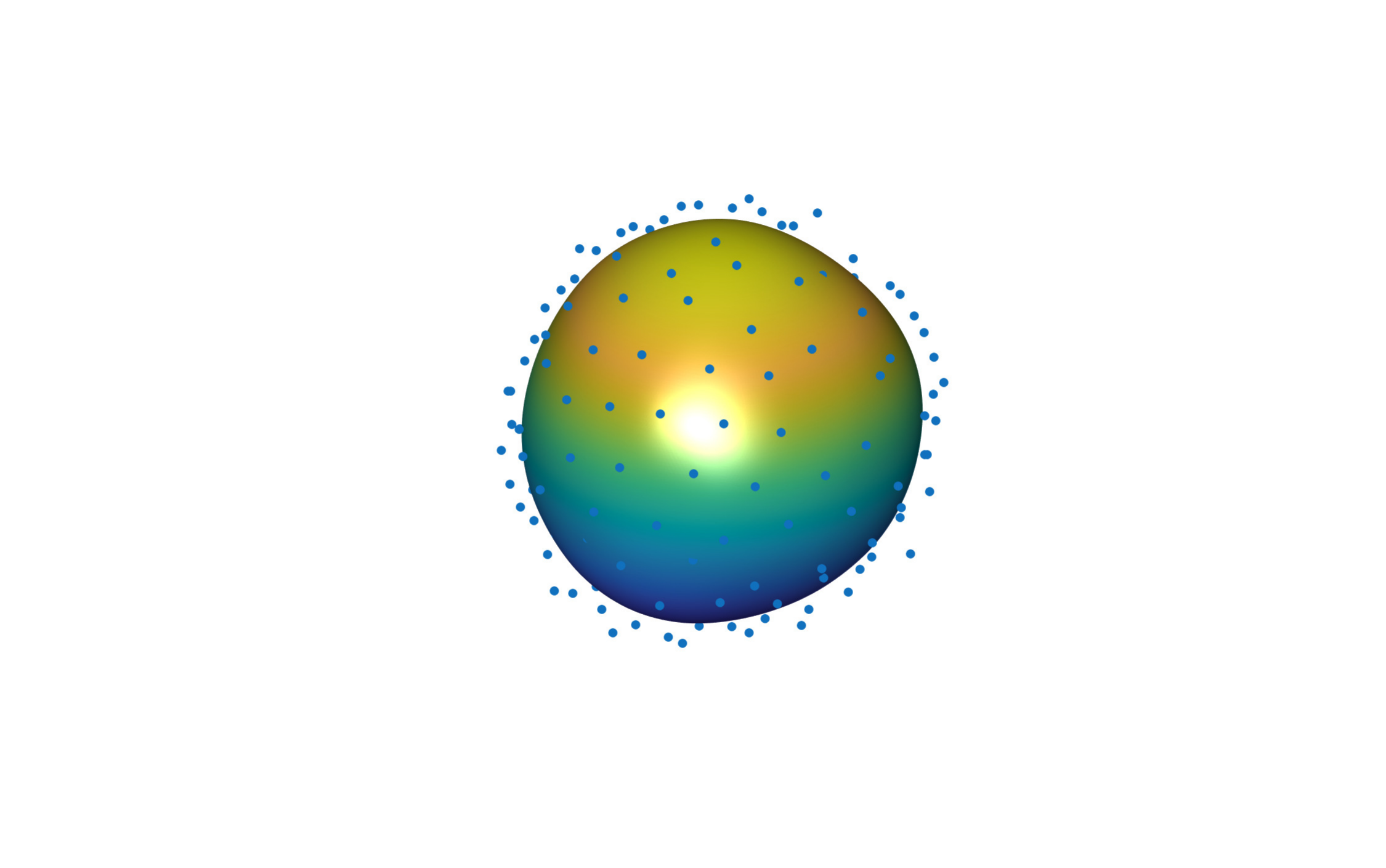}&
			\includegraphics[width = 0.25 \textwidth,clip,trim = {260 110 230 50}]{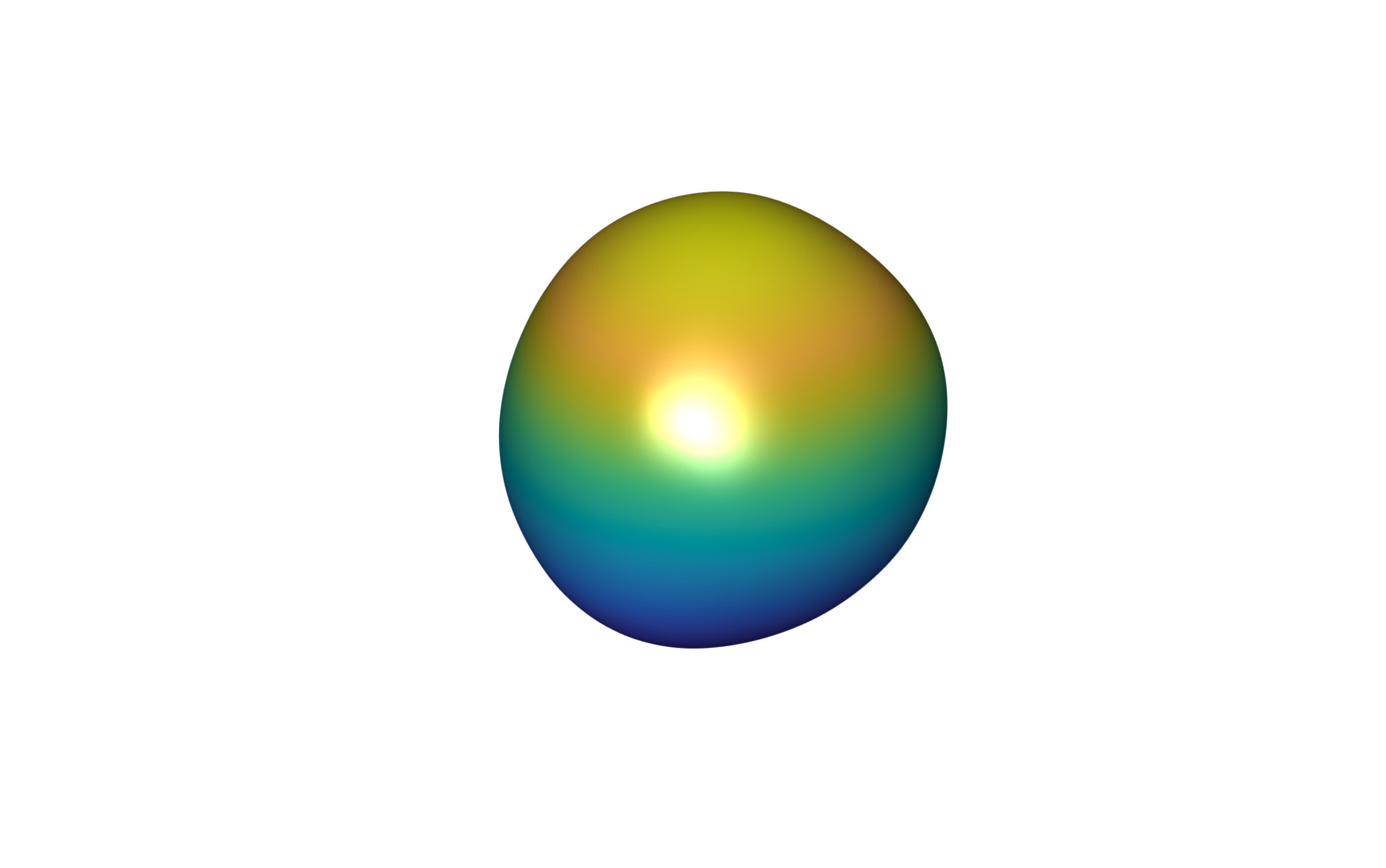}
		\end{tabular}
	\end{center}
	\caption{Top, the initial noisy triangulation is displayed on the right, showing both vertices and edges, while only the initial vertices appear on the left. Bottom, the left graphic shows the computed result overlaid with the initial vertices, and the right panel displays the result alone.}
	\label{fig:sphere}
\end{figure}

\section{Conclusion and future work}\label{sec:Conclusion}

This paper presents and analyses a new family of linear subdivision schemes to refine noisy data given on triangular meshes based on weighted least squares polynomials. The refinement rules are formulated in a way that are applicable to any triangular grid, including finite grids or grids containing extraordinary vertices, but they are geometry-dependent, which may result in non-uniform schemes. We prove linear polynomial reproduction, approximation order two, and denoising capabilities. Even though we cannot provide convergence analysis in general, we give some guidance to prove it and we do it for uniform grids.

Numerical experiments show that the schemes performance is comparable to advanced local linear regression methods and that their subdivision nature makes them suitable to deal with noisy geometric data. Moreover, this type of subdivision schemes could be applied within a multiresolution decomposition context in order to denoise data. Starting at the finest level with a dense set of noisy vertices, the multiresolution decomposition could be computed and thresholding criteria applied to remove the noise. The latter topics are left for future investigation as well as the investigation of some theoretical aspects of the new family of schemes like the computational cost, the convergence and the regularity for non-regular triangulation.

\section*{Acknowledgements}

The first author is members of the INdAM research group GNCS, which has partially supported this work accomplished within RITA and UMI-TAA groups of which she is a member, as well as through the ``Talent Attraction'' invited-researcher stays programme of the Vice-Principal for Research at the Universitat de València.

This research has been supported by Spanish MINECO project PID2023-146836NB-I00, funded by \\ MCIN/AEI/10.13039/501100011033, and GVA project CIAICO/2024/089.

\section*{Reproducibility}

The MATLAB code necessary to obtain the figures on this paper can be found in the repositories \url{https://github.com/serlou/WLS-subdivision-for-triangulations} and \url{https://github.com/serlou/triangulations-workshop}. The first one is more dedicated to the implementation of the subdivision schemes, while the second one contains the code to generate the triangulations and the initial data.

To generate the initial data to be refined in Figure \ref{fig:sphere}, we used the code in the repository \url{https://github.com/pgagarinov/spheretri}.


\begin{thebibliography}{100}
	%
	%
	
	
	\bibitem{C1}Maria Charina, Costanza Conti and Nira Dyn, Multivariate compactly supported $C^\infty$ functions by subdivision, Applied and Computational Harmonic Analysis, vol. 70, 2024
	
	
	
	\bibitem{chen} Wen Chen, Zhuo-Jia Fu and C.S. Chen,  Recent Advances in Radial Basis Function Collocation Methods, Springer Berlin, Heidelberg, 2014.
	
	
	\bibitem{C2} Costanza Conti, Luca Gemignani and Lucia Romani, From symmetric subdivision masks of Hurwitz type to interpolatory subdivision masks, Linear Algebra and Its Applications, vol. 431,
	2009
	
	
	\bibitem{C3}Costanza Conti and Nira Dyn, Non-stationary Subdivision Schemes: State of the Art and Perspectives, Springer Proceedings in Mathematics and Statistics, vol. 336, 2021
	
	\bibitem{C4} Costanza Conti and Svenja Huning, An algebraic approach to polynomial reproduction of Hermite subdivision schemes, Journal of Computational and Applied Mathematics
	vol. 349, 2019
	
	\bibitem{DAUBECHIES} Ingrid Daubechies, Igor Guskov and Wim Sweldens,
	Regularity of Irregular Subdivision. Constructive Approximation, 15, 381-426, 1999.
	
	\bibitem{Deboor} Carl De Boor, A Practical Guide to Splines, Springer, New York, 2001.
	
	\bibitem{DYNHORMANN} Nira Dyn, Allison Heard, Kai Hormann and Nir Sharon, Univariate subdivision schemes for noisy data with geometric applications. Computer Aided Geometric Design, 37, 2015, 85-104.
	
	\bibitem{FASSHAUER} Gregory E. Fasshauer, Meshfree Approximation Methods with Matlab, World Scientific Publishing, Singapore, 2007.
	
	\bibitem{Franke} Richard Franke, Scattered Data Interpolation: Tests of Some Methods, Mathematics of Computation, 38, 157.
	
	\bibitem{HASTIE} Trevor Hastie,  Robert Tibshirani and Jerome Friedman, The Elements of Statistical Learning, Springer, New York, 2009.
	\bibitem{LEVIN} David Levin, The approximation power of Moving least-squares. Mathematics of Computation 67, 224, 1998.
	\bibitem{LEVINLEVIN} Adi Levin and David Levin, Analysis of quasi-uniform subdivision, Applied and Computational Harmonic Analysis, 2003
	
	\bibitem{JiaZhou}Rong-Qing Jia and Ding-Xuan Zhou, Convergence Of Subdivision Schemes Associated With Nonnegative Masks, Siam J. Matrix Anal. Appl. Vol. 21, No. 2, pp. 418-430
	
	\bibitem{LOADER} Clive Loader, Local Regression and likelihook, Springer, New York, 1999.
	\bibitem{LOOP}Charles Loop, Smooth Subdivision Surfaces Based on Triangles, M.S. Mathematics thesis, University of Utah, 1987.
	
	\bibitem{LY24} Sergio L{\'o}pez-Ure{\~n}a and Dionisio F. Y{\'a}{\~n}ez, Subdivision Schemes Based on Weighted Local Polynomial Regression: A New Technique for the Convergence Analysis. Journal of Scientific Computing, 100(1), 2024, 1-44.
	
	\bibitem{STAM1} Jos Stam, Evaluation of Loop Subdivision Surfaces, Computer Graphics Proceedings ACM SIGGRAPH 1998,
	\bibitem{GINKEL} Ingo Ginkel and Georg Umlauf, Analyzing a generalized Loop subdivision scheme. Computing 79.2 (2007): 353-363.
	
	\bibitem{SHEPARD} Donald Shepard,  A Two-Dimensional Interpolation Function for Irregularly Spaced Data, Proc. 23rd Nat. Conf. ACM, 1968: 517-523.
	
	\bibitem{STAM2} Jos Stam and Charles Loop, Quad/triangle subdivision. Computer Graphics Forum 22.1 (2003): 1-10.
	
	
	
	\bibitem{MUSTAFA} G. Mustafa, H. Li, J. Zhang and J. Deng, $\ell_1$-Regression based subdivision schemes for noisy data. Computer-Aided Design, 58:189-199, Jan. 2015.
	
	\bibitem{Reif95} Ulrich Reif, A unified approach to subdivision algorithms near extraordinary vertices. Computer Aided Geometric Design, 1995, vol. 12, no 2, p. 153-174.
	
	\bibitem{Umlauf00} Georg Umlauf, Analyzing the characteristic map of triangular subdivision schemes. Constructive Approximation, 2000, vol. 16, p. 145-155.
	
	
	\bibitem{Levin04} David Levin, Mesh-independent surface interpolation. Geometric modeling for scientific visualization, 2004, p. 37-49.
	
	\bibitem{wendland} Holger Wendland, Scattered Data Approximation. Cambridge University Press, 2004.
	
\end{thebibliography}

\bibliographystyle{plain}

\end{document}